\documentclass{article}
\usepackage[utf8]{inputenc}
\usepackage{amsmath,amsthm,amsfonts,amssymb,bm, mathdots}
\usepackage{xcolor}
\usepackage{geometry}
\usepackage{tikz}
\usepackage{ytableau}
\usepackage{authblk}
\usepackage{enumitem}
\usepackage{indentfirst}
\usepackage{url}
\usepackage{fdsymbol}
\usepackage{adjustbox}
\usepackage[colorlinks=true,linkcolor=blue,citecolor=blue,hyperfootnotes=false]{hyperref}

\usetikzlibrary{decorations.pathreplacing}

\def\diag{ \begin{tikzpicture} \draw[dashed] (-.12,-.12) -- (.42, .42); \end{tikzpicture} }

\newcommand{\Z}{\mathbb{Z}}

\newcommand{\SSYT}{\text{SSYT}}
\newcommand{\coinv}{\text{coinv}}
\newcommand{\bg}{{\bm{\beta}/\bm{\gamma}}}
\newcommand{\I}{\mathbf{I}}
\newcommand{\J}{\mathbf{J}}
\newcommand{\K}{\mathbf{K}}
\renewcommand{\L}{\mathbf{L}}
\newcommand{\0}{\mathbf{0}}

\newcommand{\speciallattice}{
\resizebox{4cm}{!}{
\begin{tikzpicture}[baseline=(current bounding box.center)]
\draw (0,0) -- (0,3); \draw (1,0) -- (1,3); \draw (2,0) -- (2,3); \draw (3,0) -- (3,3); 
\draw (0,0) -- (3,0); \draw (0,1) -- (3,1); \draw (0,2) -- (3,2); \draw (0,3) -- (3,3);
\node[left] at (0,0.5) {$\0$}; \node[left] at (0,1.5) {$\vdots$}; \node[left] at (0,2.5) {$\0$}; 
\node[right] at (3,0.5) {$\0$}; \node[right] at (3,1.5) {$\vdots$}; \node[right] at (3,2.5) {$\0$}; 
\node[above] at (0.5,3) {$\beta(r)$}; \node[above] at (1.5,3) {$\hdots$}; \node[above] at (2.5,3) {$\beta(s)$};
\node[below] at (0.5,0) {$\gamma(r)$}; \node[below] at (1.5,0) {$\hdots$}; \node[below] at (2.5,0) {$\gamma(s)$};
\node at (0.5,0.5) {$x_1$}; \node at (1.5,0.5) {$\hdots$}; \node at (2.5,0.5) {$x_1$};
\node at (0.5,2.5) {$x_n$}; \node at (1.5,2.5) {$\hdots$}; \node at (2.5,2.5) {$x_n$};
\node at (0.5,1.5) {$\vdots$}; \node at (2.5,1.5) {$\vdots$}; 
\end{tikzpicture}}
}

\theoremstyle{plain}
\newtheorem{thm}{Theorem}[section]
\newtheorem*{thm*}{Theorem}
\newtheorem{prop}[thm]{Proposition}
\newtheorem*{prop*}{Proposition}
\newtheorem{cor}[thm]{Corollary}
\newtheorem{lem}[thm]{Lemma}
\newtheorem*{lem*}{Lemma}
\newtheorem{example}[thm]{Example}
\newtheorem{definition}[thm]{Definition}
\newtheorem{remark}[thm]{Remark}

\title{Equivalences of LLT polynomials via lattice paths}
\author{David Keating}
\affil{Department of Mathematics, UC Berkeley \\
{\small\ttfamily dkeating@berkeley.edu}}
\date{}

\begin{document}

\maketitle

\abstract{The LLT polynomials $\mathcal{L}_{\bg}  (X;t)$ are a family of symmetric polynomials indexed by a tuple of (possibly skew-)partitions $\bg= (\beta^{(1)}/\gamma^{(1)},\ldots,\beta^{(k)}/\gamma^{(k)})$.  It has recently been shown that these polynomials can be seen as the partition function of a certain vertex model whose boundary condition is determined by $\bg$. In this paper we describe an algorithm which gives a bijection between the configurations of the vertex model with boundary condition $\bg = (\beta^{(1)}/\gamma^{(1)},\beta^{(2)}/\gamma^{(2)})$ and those with boundary condition $(\bg)_{swap} = (\beta^{(2)}/\gamma^{(2)},\beta^{(1)}/\gamma^{(1)})$. We prove a sufficient condition for when this bijection is weight-preserving up to an overall factor of $t$, which in turn implies that the corresponding LLT polynomials are equal up to the same overall factor. Extending these techniques, we are able to systematically determine linear relations within families of LLT polynomials.}

\section{Introduction}
Originally defined by Lascoux, Leclerc, and Thibon \cite{LLT} as the generating function of a spin statistic on ribbon tableaux, the eponymously named LLT polynomials are a family of symmetric polynomials which can be seen as a $t$-deformation of products of Schur polynomials. In \cite{HHLRU} the LLT polynomials were reformulated as the generating function for an inversion statistic on tuples of semi-standard Young tableaux, with the relationship between ribbon tableaux and tuples of SSYT given by the Stanton-White correspondence \cite{SW}. Most recently, in \cite{shuffle}, the authors used a new formulation of the LLT polynomials in their work on the generalization of the shuffle theorem. We will use this formulation in what follows, and we will refer to them as the {\em coinversion LLT polynomials}.

We study the LLT polynomials from the perspective of vertex models. Vertex models have long been studied in relation to integrable systems and statistical mechanics (see \cite{reshNotes} and references therein). Recently, they have been used to gain new insights on symmetric polynomials and their non-symmetric variants (for example, but by no means an exhaustive list, \cite{BBF,GW,BBBG,BW}). It was shown in \cite{ABW,CGKM} that the LLT polynomials could be expressed as the partition function of a certain vertex model. In \cite{ABW} it was shown that, in fact, the LLT polynomial vertex model was a degeneration of a more general vertex model related to the quantized affine Lie superalgebra $U_q(\widehat{\mathfrak{sl}}(1|n))$. We will not need that level of generality here. 

We say that two LLT polynomials are {\em equivalent} if they are equal up to an overall factor of $t$. In this paper, we use the vertex model structure to prove a sufficient condition for when swapping a pair of partitions in the indexing tuple of an LLT polynomial results in an equivalent LLT polynomial. The following is our main result.
\begin{thm}
If there is a unique non-crossing matching $M$ of the sequence of beads associated to $\bg$, then $\mathcal{L}_{\bg}(X_n;t)$ and $\mathcal{L}_{(\bg)_{swap}}(X_n;t)$ are equivalent. In particular, 
\[
\mathcal{L}_{\bg}(X_n;t) = \left(\prod_{a\in M} w(a)\right) \mathcal{L}_{(\bg)_{swap}}(X_n;t)
\]
where the product is over all arcs $a$ in the matching and the weight of an arc is given by (\ref{eq:arcweight}).
\end{thm}
We then extend our techniques to construct linear relations between certain LLT polynomials. The following theorem demonstrates one possible application of these techniques.
\begin{thm}
For every $\bm{\beta}$ in the family of $\binom{2n}{n}$ partitions given in (\ref{eq:partfam}), the LLT polynomial $\mathcal{L}_{\bm{\beta}}(X_n;t)$ can be written as
\[
\mathcal{L}_{\bm{\beta}}(X_n;t) = \sum_{j=1}^{C_n} t^{n_{j}(\bm{\beta})} g_j(X_n;t)
\]
where $C_n = \frac{1}{n+1}\binom{2n}{n}$ is the $n^{th}$ Catalan number, $n_{j}(\bm{\beta})\in \mathbb{Z}$ for each $j$ and $\bm{\beta}$, and the $g_i$'s are polynomials symmetric in $X_n$.
\end{thm}

The layout of this paper is as follows: In Section \ref{sect:LLT}, we define the coinversion LLT polynomials. We briefly explain how the LLT polynomials can be seen as the partition function of a certain vertex model. In Section \ref{sect:procedure}, we describe an algorithm which selectively swaps the color of certain path segments in a configuration of the vertex model. We use this algorithm to give a bijection between the configurations of the vertex model with boundary condition $\bg = (\beta^{(1)}/\gamma^{(1)},\beta^{(2)}/\gamma^{(2)})$ and those with boundary condition $(\bg)_{swap} = (\beta^{(2)}/\gamma^{(2)},\beta^{(1)}/\gamma^{(1)})$. In Section \ref{sect:weight}, we prove a sufficient condition for when this bijection is weight-preserving up to an explicit overall power of $t$, which implies that the corresponding LLT polynomials are equal up to the same overall factor. In Section \ref{sect:relations}, we show how we can use the tools developed in the previous sections to determine linear relations within families of LLT polynomials.

\noindent

\section{LLT polynomials}\label{sect:LLT}
In this section we give a brief description of the coinversion LLT polynomials. We then review their characterization as lattice paths introduced in \cite{ABW,CGKM}.

Let $\lambda = (\lambda_1 \geq \cdots \geq \lambda_m \geq 0)$ be a partition with $l(\lambda)=m$ parts. Note that we consider our partitions to have a fixed number of parts, but allow for the possibility of parts of size zero.  We associate to $\lambda$ its Young (or Ferrers) diagram $D(\lambda) \subseteq \Z \times \Z$, given as
\[ D(\lambda) = \{(i,j) \mid 1 \leq i \leq \ell(\lambda), \; 1 \leq j \leq \lambda_i \}. \]
We draw our diagrams in French notation as shown in the below example:
\[ \lambda = (4,2,1), \qquad D(\lambda) = 
\ytableausetup{aligntableaux=center}
\begin{ytableau} \\ & \\ & & \bullet & \end{ytableau}. \]
We refer to the elements in $D(\lambda)$ as {\em cells}. The cell marked above has coordinates (1,3). The {\em content} of a cell $u = (i,j)$ in row $i$ and column $j$ of any Young diagram is defined as $c(u) = j-i$. 

A skew-partition is a pair of partitions $\lambda$ and $\mu$ such that $D(\mu)\subseteq D(\lambda)$. We denote it by $\lambda/\mu$. The Young diagram of a skew-partition $\lambda/\mu$ is given by $D(\lambda)-D(\mu)$, the cells are those in the Young diagram of $\lambda$ not in the Young diagram of $\mu$.

A semi-standard filling of a partition of $\lambda$ is a filling of the cells of the Young diagram of $\lambda$ by positive integers such that they are weakly increasing along the rows and strictly increasing along the columns (from left to right, bottom to top). We denote the set of all semi-standard fillings of shape $\lambda$ by $\SSYT(\lambda)$. The analogous definition holds for semi-standard fillings of skew-partitions. 

Given a tuple $\bg = (\beta^{(1)}/\gamma^{(1)}, \ldots, \beta^{(k)}/\gamma^{(k)})$ of skew partitions, define a semi-standard Young tableau $T$ of shape $\bg$ to be a semi-standard Young tableau on each $\beta^{(j)}/\gamma^{(j)}$, that is,
\[ \SSYT(\bg) = \SSYT(\beta^{(1)}/\gamma^{(1)}) \times \cdots \times \SSYT(\beta^{(k)}/\gamma^{(k)}). \]
We can picture this as placing the Young diagrams aligned diagonally ``on content lines" with the first shape in the South-West direction and the last shape in the North-East direction. See Example~\ref{exampleLLT} below.

\begin{example}
	Let $\bg = ((3,1), (2,2,2)/(1,1,1), (1), (2,1)/(2))$. Below is one possible semi-standard filling of $\bg$. The top row labels the contents of each diagonal line.
	
	\ytableausetup{nosmalltableaux}
	\ytableausetup{nobaseline}
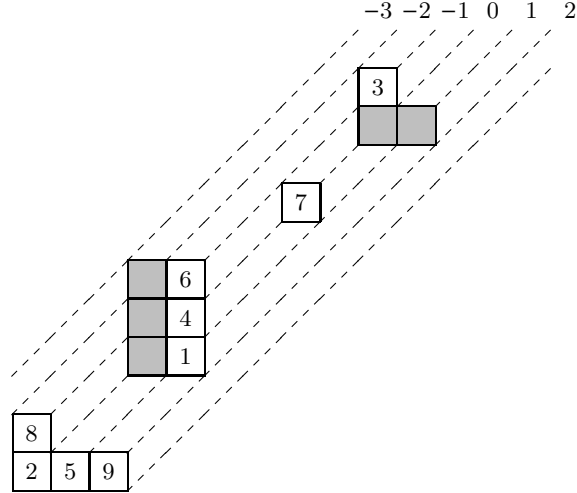
	\begin{center}
	\begin{ytableau}
		\none & \none & \none & \none & \none & \none & \none & \none & \none & \none[-3] & \none[-2] & \none[-1] & \none[0] & \none[1] & \none[2] \\			
		\none & \none & \none & \none & \none & \none & \none & \none &\none[\diag] &\none[\diag] &\none[\diag] &\none[\diag] & \none[\diag] & \none[\diag] \\
		\none & \none & \none & \none & \none & \none & \none &\none[\diag] &\none[\diag] & 3 &\none[\diag] &\none[\diag] & \none[\diag] & \none[\diag] \\
		\none & \none & \none & \none & \none &\none &\none[\diag] &\none[\diag] & \none[\diag] & *(lightgray) & *(lightgray) &\none[\diag] & \none[\diag] \\
		\none & \none & \none & \none &\none &\none[\diag] &\none[\diag] &\none[\diag] &\none[\diag] &\none[\diag] &\none[\diag] &\none[\diag] & \none \\
		\none & \none & \none &\none &\none[\diag] &\none[\diag] &\none[\diag] & 7 &\none[\diag] &\none[\diag] &\none[\diag] &\none \\
		\none & \none &\none &\none[\diag] &\none[\diag] &\none[\diag] &\none[\diag] &\none[\diag] &\none[\diag] &\none[\diag] &\none &\none \\
		\none & \none &\none[\diag] & *(lightgray) & 6 &\none[\diag] &\none[\diag] &\none[\diag] &\none[\diag] &\none &\none & \none \\
		\none & \none[\diag] & \none[\diag] & *(lightgray) & 4 &\none[\diag] &\none[\diag] &\none[\diag] &\none &\none & \none & \none \\
		\none[\diag] & \none[\diag] & \none[\diag] & *(lightgray) & 1 &\none[\diag] &\none[\diag] &\none &\none & \none & \none & \none \\
		\none[\diag] & \none[\diag] & \none[\diag] &\none[\diag] &\none[\diag] &\none[\diag] &\none &\none & \none & \none & \none & \none \\
		8 &\none[\diag] &\none[\diag] &\none[\diag] &\none[\diag] &\none &\none & \none & \none & \none & \none & \none \\
		2 & 5 & 9 &\none[\diag] &\none &\none & \none & \none & \none & \none & \none & \none \\
	\end{ytableau}	
	\end{center}
	\label{exampleLLT}
\end{example}

Given a tuple $\bg$ of skew partitions, we say that three cells $u, v, w$ in the content-aligned Young diagrams of the skew partitions form a {\em triple} of $\bg$ if 
\begin{enumerate}
\item[i] $v \in \bg$, 
\item[ii] they are situated as below,
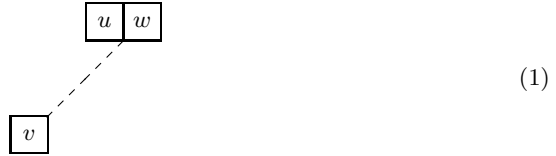
\begin{equation} \label{triple}
\ytableausetup{nobaseline}
\begin{ytableau}
\none & \none & \none &\none & u & w  \\
\none & \none & \none & \none & \none[\diag]  \\
\none & \none & \none  & \none[\diag] \\
\none & \none & v \\
\end{ytableau}
\end{equation}
namely, with $v$ and $w$ on the same content line but in different shapes, with $w$ in (or adjacent to) a row of a later shape, and $u$ on a content line one smaller, in the same row as $w$.
\item[iii] If $u, w$ are in row $r$ of $\beta^{(j)}/\gamma^{(j)}$, then $u$ and $w$ must be between the cells $(r, \gamma^{(j)}_r-1)$ and $(r, \beta^{(j)}_r+1)$, inclusively.
\end{enumerate}
It is important to note that while $v$ must be a cell in $\bg$, we allow the cells $u$ and $w$ to not be in any of the skew shapes. We allow $u$ to possibly be one cell to the left of the start of a row in the skew diagram, and we allow $w$ to possibly be one cell to the right of the end of a row in the skew diagram. If there is a row of length zero then both $u$ and $w$ can lie outside the diagram.

\begin{definition}
Let $\bg$ be a tuple of skew partitions and let $T \in \SSYT(\bg)$. Let $a, b, c$ be the entries in the cells of a triple $(u, v, w)$ respectively, where we set $a = 0$ and $c=\infty$ if the respective cell is not in $\bg$. Given the triple of entries
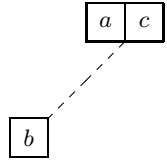
\ytableausetup{nobaseline}
\[
\begin{ytableau}
\none & \none & \none &\none & a & c  \\
\none & \none & \none & \none & \none[\diag]  \\
\none & \none & \none  & \none[\diag] \\
\none & \none & b \\
\end{ytableau}
\]
we say this is a {\em coinversion triple} of $T$ if $a \leq b \leq c$.
\label{inv-coinv-triple}
\end{definition}

\item We will only consider the case when $X$ is a finite alphabet $X_n = \{x_1,\ldots, x_n\}$. For a filling $T$ of $\bg$ we use the notation $x^T$ to mean $\prod_{i=1}^n x_i^{t_i}$ where $t_i$ is the number of times $i$ appears in $T$.

\begin{definition} \label{coinv-LLT}
Let $\bg$ be a tuple of skew partitions. The coinversion LLT polynomial is the generating function
\[ \mathcal{L}_{\bg}(X; t) = \sum_{T \in \SSYT(\bg)} t^{\coinv(T)} x^T \]
where $\coinv(T)$ is the number of coinversion triples of $T$.
\end{definition}
\noindent The coinversion LLT polynomials are related to the inversion LLT polynomials of \cite{HHLRU} in a simple way
\begin{equation} 
\mathcal{L}_{\bg}(X; t) = t^m \mathcal{G}_{\bg}(X; t^{-1}) \label{inv-coinv} \end{equation}
where $\mathcal{G}$ is the inversion LLT polynomial and  $m$ is the total number of triples in $\bg$.

\subsection{LLT vertex model}
In \cite{ABW,CGKM} it was shown that there is a bijection between tuples of SSYT and a certain vertex model consisting of several colors of lattice paths. We review the construction from \cite{CGKM} here, and refer the readers to the original papers for details. 

Consider a lattice model consisting of up-right lattice paths of $k$ different colors where paths of the same color are not allowed to intersect. At each face of our lattice we assign a label in $\{0,1\}^k$ to the sides of the face as follows:
\[ 
\resizebox{2cm}{!}{
\begin{tikzpicture}[baseline=(current bounding box.center)]
\draw (0,0) -- (0,1); \draw (1,0) -- (1,1); 
\draw (0,0) -- (1,0); \draw (0,1) -- (1,1); 
\node[below] at (0.5,0) {$\I$}; \node[above] at (0.5,1) {$\K$};
\node[left] at (0,0.5) {$\J$}; \node[right] at (1,0.5) {$\L$}; 
\node at (0.5,0.5) {$x$};
\end{tikzpicture}}, 
\hspace{1cm} \I,\J,\K,\L \in \{0,1\}^k
\]
\noindent where the $x$ indicates a parameter that will be used in defining the weight of the face and $ \I,\J,\K,\L$ express which paths are crossing the respective sides of the face. One should interpret, for example, a 1 in the $i^{th}$ component of $\I$ as indicating that a path of color $i$ crosses the bottom boundary of the face. The weights of the face are given by
\begin{equation} \label{FaceWeight}
\begin{aligned}
L_x(\I,\J;\K,\L) = & x^{\substack{\text{\# colors exiting the} \\ \text{face to the right}}} \prod_{\substack{\text{colors $i$ exiting the} \\ \text{face to the right}}} t^{\substack{\text{\# colors larger than $i$ that} \\ \text{appear in the face}}}
\end{aligned}
\end{equation}
\noindent whenever $\I+\J = \K+\L$ and there is no $i \in \{1,2,\ldots,k\}$ such that $I_i = J_i = 1$, and otherwise we set $L_x(\I,\J;\K,\L) = 0$.  The condition $\I+\J = \K+\L$ ensures that any path that enters a face from the bottom or left must exit the face from the top or right, while the condition that there is no $i \in \{1,2,\ldots,k\}$ such that $I_i = J_i = 1$ ensures that paths of a given color are non-intersecting.

Let us introduce some notations that will be used to define the boundary condition of our vertex model. Given a tuple of partitions $\bm{\mu} = (\mu^{(1)},\ldots,\mu^{(k)})$ and an integer $i$, let $\mu(i) \in \{0,1\}^k$ be the vector whose $j$-th component, for each index $j \in \{1,\ldots,k\}$, is 1 if and only if, for some $m \in \{1,\ldots,\ell(\mu^{(j)})\}$, the length of the $m^{th}$ row of $\mu^{(j)}$ is $i+m-1$.  Let $\bg = (\beta^{(1)}/\gamma^{(1)},\ldots,\beta^{(k)}/\gamma^{(k)})$ be a tuple of skew partitions. Let
\begin{align*}
    &r = r(\bg) = \min \{ i \in \Z : \gamma(i) \neq \0\}, \\
    &s = s(\bg) = \max \{ i \in \Z : \beta(i) \neq \0\}.
\end{align*}
Note that $s-r+1$ gives the number of columns necessary in the vertex model for $\bg$, $\beta(r+i)$ gives the top boundary condition for the $i^{th}$ column from the left, and $\gamma(r+i)$ gives the bottom boundary condition for the $i^{th}$ column from the left. 
With this notation, we introduce the lattice that will be of particular interest to us:
\begin{equation} L_{\bg} := \speciallattice  \label{speciallattice} \end{equation}

We let $\mathcal{Z}_{\bg}(X_n; t)$ denote the partition function of $L_{\bg}$, that is 
\[ \mathcal{Z}_{\bg}(X_n; t) = \sum_{L \in LC_{\bg}} \text{weight}(L)  \]
where $LC_{\bg}$ denotes the set of path configurations on $L_{\bg}$  satisfying the boundary condition and weight$(L)$ is the product of the weight of each face in $L$. 

\begin{thm} \label{LatticeModel}
Let $\bg$ be a tuple of skew partitions. Then,
\[ \mathcal{Z}_{\bg}(X_n; t) = \mathcal{L}_\bg(X_n;t). \]
\end{thm}
\noindent We refer the reader to \cite{CGKM} for a proof.

\begin{example}\label{ex:vertex}
As an example of the above constructions, consider $\bg = ((2,2)/(1,0), (1))$ with $n=2$. We have
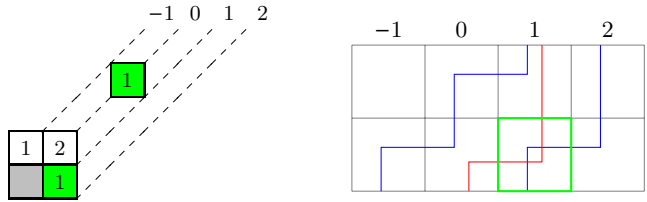
\[
\begin{tabular}{ccc}
	\resizebox{4cm}{!}{
	\ytableausetup{nosmalltableaux}
	\ytableausetup{nobaseline}
	\ytableausetup{aligntableaux=center}
	\begin{ytableau}
	\none & \none & \none & \none & \none[-1] & \none[0] & \none[1] & \none[2] \\	
	\none & \none & \none & \none[\diag] &\none[\diag] &\none[\diag] &\none[\diag] &\none \\
	\none & \none & \none[\diag] &*(green) 1  &\none[\diag]  &\none[\diag] & \none \\
	\none & \none[\diag] & \none[\diag] &\none[\diag] &\none[\diag] &\none & \none \\
	 1 &2 &\none[\diag] &\none[\diag] &\none & \none & \none \\
	*(lightgray)  & *(green) 1 &\none[\diag] &\none & \none & \none & \none
	\end{ytableau}} &  &
	\resizebox{4cm}{!}{
	\begin{tikzpicture}[baseline=(current bounding box.center)]
	\draw[help lines] (0,0) grid (4,2);
	\draw[blue] (0.4,0)--(0.4,0.6)--(1.4,0.6)--(1.4,1.6)--(2.4,1.6)--(2.4,2);
	\draw[blue] (2.4,0)--(2.4,0.6)--(3.4,0.6)--(3.4,2);
	\draw[red] (1.6,0)--(1.6,0.4)--(2.6,0.4)--(2.6,2);
	\draw[green, thick] (2,0) rectangle (3,1);
	\node[above] at (0.5,2) {$-1$}; \node[above] at (1.5,2) {$0$}; \node[above] at (2.5,2) {$1$}; \node[above] at (3.5,2) {$2$};
	\end{tikzpicture}}
\end{tabular}
\]
where the left is a possible semi-standard filling of the tuple of partitions and the right is the corresponding lattice paths. Note that there is a simple bijection between fillings of the sequence of tableaux and lattice paths in which each row of tableaux $i$ is a path of color $i$ and the entries in the row correspond to the height of the horizontal steps of the path. The columns of the vertex model correspond to the content lines in the tableaux formulation (see the labeling in the above example). The green cells in the tableaux correspond to the coinversion triple 
\ytableausetup{nobaseline}
\[
\begin{ytableau}
\none & \none & \none &\none & 1 & \infty  \\
\none & \none & \none & \none & \none[\diag]  \\
\none & \none & \none  & \none[\diag] \\
\none & \none & 1 \\
\end{ytableau}
\]
where the cell not contained in the diagrams is given filling infinity. This triple corresponds to the face highlighted in green in the vertex model.
\end{example}

We now make some definitions and conventions that will be useful for us later on. We will only consider tuples with two partitions, $\bg = (\beta^{(1)}/\gamma^{(1)},\beta^{(2)}/\gamma^{(2)})$. In the vertex model formulation, we will always draw the path corresponding to the first partition in blue, and those corresponding to the second in red.

We say that a path incident to a boundary of the lattice is a {\em singleton} if no paths of any other colors are also incident to the same boundary at the same face. In Example \ref{ex:vertex}, the rightmost blue path on the top boundary is a singleton, and all three paths on the bottom boundary are singletons.

\section{Partition Swapping Algorithm}\label{sect:procedure}
In this section we will construct an algorithm which defines a bijection between path configurations with boundary condition given by $\bg = (\beta^{(1)}/\gamma^{(1)},\beta^{(2)}/\gamma^{(2)})$ to path configurations with boundary condition $(\bg)_{swap} = (\beta^{(2)}/\gamma^{(2)},\beta^{(1)}/\gamma^{(1)})$. To do this, we define a procedure which starts with a configuration with boundary condition $\bg$, selects certain path segments in this configuration, then swaps the colors of the selected path segments, resulting in a configuration with boundary condition $(\bg)_{swap}$.

Suppose we are given a path configuration and we select some point along one of the paths and start following that path's trajectory. We say that we are traveling {\em forward} if while following a path we are traveling upward or to the right, otherwise we say that we are traveling {\em backward}. 

Our basic procedure is as follows:
\begin{enumerate}
\item Choose a singleton red path on the top boundary or a singleton blue path on the bottom boundary.  
\item Follow a segment of the red (resp. blue) path traveling backward (resp. forward) until we hit a face containing a path of the other color. 
\item Switch to following this new segment of path according to the following rules:
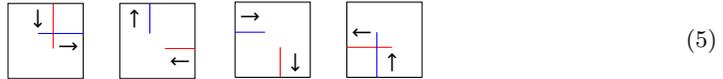
\begin{equation}\label{eq:rules}
\begin{tabular}{cccc}
\begin{tikzpicture}[baseline=(current bounding box.center)] 
\draw (0,0)--(1,0)--(1,1)--(0,1)--(0,0);
\draw[red] (0.6,0.4)--(0.6,1);
\draw[blue] (0.4,0.6)--(1,0.6);
\node[left] at (0.6,0.8) {$\downarrow$};
\node[below] at (0.8,0.6) {$\rightarrow$};
\end{tikzpicture}
& 
\begin{tikzpicture}[baseline=(current bounding box.center)] 
\draw (0,0)--(1,0)--(1,1)--(0,1)--(0,0);
\draw[blue] (0.4,0.6)--(0.4,1);
\draw[red] (0.6,0.4)--(1,0.4);
\node[left] at (0.4,0.8) {$\uparrow$};
\node[below] at (0.8,0.4) {$\leftarrow$};
\end{tikzpicture}
&
\begin{tikzpicture}[baseline=(current bounding box.center)] 
\draw (0,0)--(1,0)--(1,1)--(0,1)--(0,0);
\draw[blue] (0,0.6)--(0.4,0.6);
\draw[red] (0.6,0)--(0.6,0.4);
\node[above] at (0.2,0.6) {$\rightarrow$};
\node[right] at (0.6,0.2) {$\downarrow$};
\end{tikzpicture}
&
\begin{tikzpicture}[baseline=(current bounding box.center)] 
\draw (0,0)--(1,0)--(1,1)--(0,1)--(0,0);
\draw[red] (0,0.4)--(0.6,0.4);
\draw[blue] (0.4,0)--(0.4,0.6);
\node[above] at (0.2,0.4) {$\leftarrow$};
\node[right] at (0.4,0.2) {$\uparrow$};
\end{tikzpicture}
\end{tabular}
\end{equation}
that is, for example, if we enter the face following a red path segment from the top we exit the face following a blue path segment traveling right.  
\item We then repeat this process on the new segment of path. Continue repeating until the path segment we are following ends on the boundary.
\end{enumerate}
We will show in Lemma \ref{lem:end} that this is well-defined and the procedure ends on the boundary after a finite number of steps. For example, consider a configuration with boundary condition given by $\bg= ((8,7,6),(4,3,2)/(2,0,0))$. Using this procedure we have
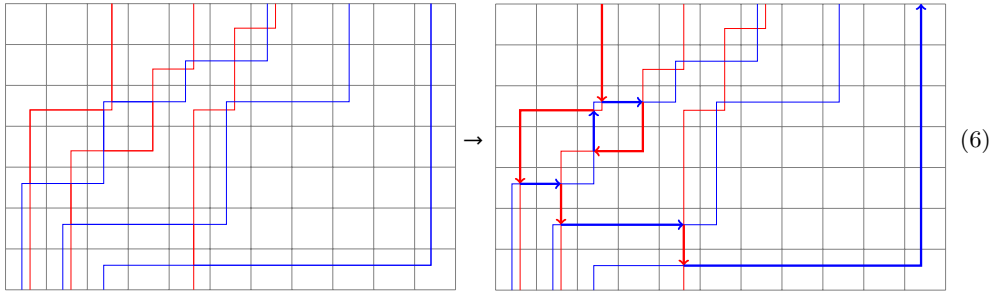
\begin{equation}\label{ex:onewalk}
\resizebox{0.4\textwidth}{!}{ $
\begin{tikzpicture}[baseline=(current bounding box.center)] 
\draw[help lines] (0,0) grid (11,7);

\draw[red] (0.6,0)--(0.6,4.4)--(2.6,4.4)--(2.6,7);
\draw[red] (1.6,0)--(1.6,3.4)--(3.6,3.4)--(3.6,5.4)--(4.6,5.4)--(4.6,7);
\draw[red] (4.6,0)--(4.6,2.4)--(4.6,2.4)--(4.6,4.4)--(5.6,4.4)--(5.6,6.4)--(6.6,6.4)--(6.6,7);

\draw[blue] (0.4,0)--(0.4,2.6)--(2.4,2.6)--(2.4,4.6)--(4.4,4.6)--(4.4,5.6)--(6.4,5.6)--(6.4,7);
\draw[blue] (1.4,0)--(1.4,1.6)--(5.4,1.6)--(5.4,4.6)--(8.4,4.6)--(8.4,7);
\draw[blue] (2.4,0)--(2.4,0.6)--(10.4,0.6)--(10.4,7);

\draw[red] (2.6,7)--(2.6,4.6);
\draw[blue] (2.6,4.6)--(3.6,4.6);
\draw[red] (3.6,4.6)--(3.6,3.4)--(2.4,3.4);
\draw[blue] (2.4,3.4)--(2.4,4.4);
\draw[red] (2.4,4.4)--(0.6,4.4)--(0.6,2.6);
\draw[blue] (0.6,2.6)--(1.6,2.6);
\draw[red] (1.6,2.6)--(1.6,1.6);
\draw[blue] (1.6,1.6)--(4.6,1.6);
\draw[red] (4.6,1.6)--(4.6,0.6);
\draw[blue] (4.6,0.6)--(10.4,0.6)--(10.4,7);
\end{tikzpicture}
$}
\rightarrow
\resizebox{0.4\textwidth}{!}{ $
\begin{tikzpicture}[baseline=(current bounding box.center)] 
\draw[help lines] (0,0) grid (11,7);

\draw[red] (0.6,0)--(0.6,4.4)--(2.6,4.4)--(2.6,7);
\draw[red] (1.6,0)--(1.6,3.4)--(3.6,3.4)--(3.6,5.4)--(4.6,5.4)--(4.6,7);
\draw[red] (4.6,0)--(4.6,2.4)--(4.6,2.4)--(4.6,4.4)--(5.6,4.4)--(5.6,6.4)--(6.6,6.4)--(6.6,7);

\draw[blue] (0.4,0)--(0.4,2.6)--(2.4,2.6)--(2.4,4.6)--(4.4,4.6)--(4.4,5.6)--(6.4,5.6)--(6.4,7);
\draw[blue] (1.4,0)--(1.4,1.6)--(5.4,1.6)--(5.4,4.6)--(8.4,4.6)--(8.4,7);
\draw[blue] (2.4,0)--(2.4,0.6)--(10.4,0.6)--(10.4,7);

\draw[red, ultra thick, ->] (2.6,7)--(2.6,4.6);
\draw[blue, ultra thick, ->] (2.6,4.6)--(3.6,4.6);
\draw[red, ultra thick, ->] (3.6,4.6)--(3.6,3.4)--(2.4,3.4);
\draw[blue, ultra thick, ->] (2.4,3.4)--(2.4,4.4);
\draw[red, ultra thick, ->] (2.4,4.4)--(0.6,4.4)--(0.6,2.6);
\draw[blue, ultra thick, ->] (0.6,2.6)--(1.6,2.6);
\draw[red, ultra thick, ->] (1.6,2.6)--(1.6,1.6);
\draw[blue, ultra thick, ->] (1.6,1.6)--(4.6,1.6);
\draw[red, ultra thick, ->] (4.6,1.6)--(4.6,0.6);
\draw[blue, ultra thick, ->] (4.6,0.6)--(10.4,0.6)--(10.4,7);
\end{tikzpicture}
$}
\end{equation}
where the path segments we follow after starting at the top-left red path are given in bold on the right. Note that the rules given in (\ref{eq:rules}) imply that we alternate traveling backward and forward and that we also alternate colors. One can check that this implies the rules presented in (\ref{eq:rules}) are sufficient to describe all situations that can arise using this procedure.

We call the sequence of path segments that are traversed when running this algorithm starting from a singleton path on the boundary the {\em{walk}} starting from that boundary path. These walks satisfy some straightforward properties.
\begin{lem}\label{lem:end}
The walk cannot enter a loop. In particular, the walk must terminate at either the top or bottom boundaries.
\end{lem}
\begin{proof}
Note that from the rules in (\ref{eq:rules}), if we are at a face where our walk changes color and we know how we exit the face, then we know along which path segment we must have entered. If we are on a path segment in a loop then the previous path segment in the walk must also be in the loop, as it is the only way to get to the current path segment.  Continuing this, we see that every previous path segment must be in the loop.  But this contradicts that the walk starts on the boundary.
\end{proof}
\begin{lem}
If we begin the walk at a red path on the top boundary, it will terminate following either a red path segment on the bottom boundary or a blue path segment on the top boundary. Similarly, if we begin at a blue path on the bottom boundary, the walk will terminate following either a red path segment on the bottom boundary or a blue path segment on the top boundary.
\end{lem}
\begin{proof}
Suppose we begin at a red path on the top boundary traveling backward. At each step of this procedure, we alternate the direction of travel and color, so we will always travel forward on blue path segments and backward on red path segments. By Lemma \ref{lem:end}, the walk must end on the boundary. If it ends on the top boundary, we must be traveling forward, and thus be on a blue path segment. If it ends on the bottom boundary, we must be traveling backward along a red path segment. A similar argument works if we begin at a blue path on the bottom boundary.
\end{proof}

We will repeat this procedure starting from every singleton red path on the top boundary and every singleton blue path on the bottom boundary. In our example, we have
\[
\resizebox{0.6\textwidth}{!}{ $
\begin{tikzpicture}[baseline=(current bounding box.center)] 
\draw[help lines] (0,0) grid (11,7);

\draw[red] (0.6,0)--(0.6,4.4)--(2.6,4.4)--(2.6,7);
\draw[red] (1.6,0)--(1.6,3.4)--(3.6,3.4)--(3.6,5.4)--(4.6,5.4)--(4.6,7);
\draw[red] (4.6,0)--(4.6,2.4)--(4.6,2.4)--(4.6,4.4)--(5.6,4.4)--(5.6,6.4)--(6.6,6.4)--(6.6,7);

\draw[blue] (0.4,0)--(0.4,2.6)--(2.4,2.6)--(2.4,4.6)--(4.4,4.6)--(4.4,5.6)--(6.4,5.6)--(6.4,7);
\draw[blue] (1.4,0)--(1.4,1.6)--(5.4,1.6)--(5.4,4.6)--(8.4,4.6)--(8.4,7);
\draw[blue] (2.4,0)--(2.4,0.6)--(10.4,0.6)--(10.4,7);

\draw[red, ultra thick, ->] (2.6,7)--(2.6,4.6);
\draw[blue, ultra thick, ->] (2.6,4.6)--(3.6,4.6);
\draw[red, ultra thick, ->] (3.6,4.6)--(3.6,3.4)--(2.4,3.4);
\draw[blue, ultra thick, ->] (2.4,3.4)--(2.4,4.4);
\draw[red, ultra thick, ->] (2.4,4.4)--(0.6,4.4)--(0.6,2.6);
\draw[blue, ultra thick, ->] (0.6,2.6)--(1.6,2.6);
\draw[red, ultra thick, ->] (1.6,2.6)--(1.6,1.6);
\draw[blue, ultra thick, ->] (1.6,1.6)--(4.6,1.6);
\draw[red, ultra thick, ->] (4.6,1.6)--(4.6,0.6);
\draw[blue, ultra thick, ->] (4.6,0.6)--(10.4,0.6)--(10.4,7);

\draw[red, ultra thick, ->] (4.6,7)--(4.6,5.6);
\draw[blue, ultra thick, ->] (4.6,5.6)--(5.6,5.6);
\draw[red, ultra thick, ->] (5.6,5.6)--(5.6,4.6);
\draw[blue, ultra thick, ->] (5.6,4.6)--(8.4,4.6)--(8.4,7);

\draw[blue, ultra thick, ->] (2.4,0)--(2.4,0.6)--(4.6,0.6);
\draw[red, ultra thick, ->] (4.6,0.6)--(4.6,0);
\end{tikzpicture}
$}
\]
where we highlight all the traversed path segments. The following lemma is a direct consequence of the rules given in \eqref{eq:rules}.
\begin{lem}
If we start this procedure at two different points on the boundary, then the walks for each cannot cross (although they may touch at a corner).
\end{lem}

Finally, let $\Phi$ be the operator that takes a configuration with boundary condition given by $\bg$, enacts the basic procedure for every singleton red path on the top boundary and every singleton blue path on the bottom boundary, then swaps the color of all the traversed path segments.  In our example, this would give
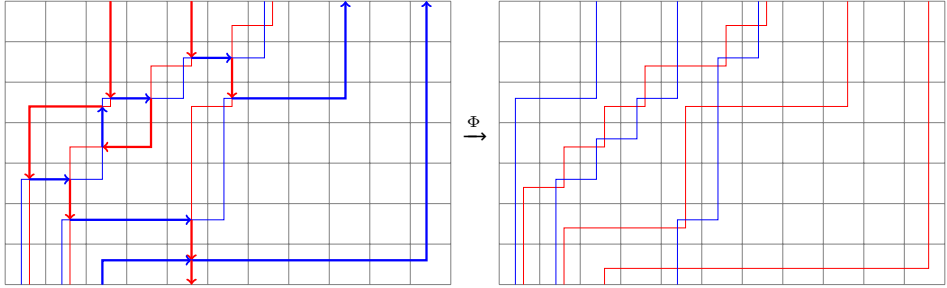
\[
\resizebox{0.4\textwidth}{!}{
\begin{tikzpicture}[baseline=(current bounding box.center)] 
\draw[help lines] (0,0) grid (11,7);

\draw[red] (0.6,0)--(0.6,4.4)--(2.6,4.4)--(2.6,7);
\draw[red] (1.6,0)--(1.6,3.4)--(3.6,3.4)--(3.6,5.4)--(4.6,5.4)--(4.6,7);
\draw[red] (4.6,0)--(4.6,2.4)--(4.6,2.4)--(4.6,4.4)--(5.6,4.4)--(5.6,6.4)--(6.6,6.4)--(6.6,7);

\draw[blue] (0.4,0)--(0.4,2.6)--(2.4,2.6)--(2.4,4.6)--(4.4,4.6)--(4.4,5.6)--(6.4,5.6)--(6.4,7);
\draw[blue] (1.4,0)--(1.4,1.6)--(5.4,1.6)--(5.4,4.6)--(8.4,4.6)--(8.4,7);
\draw[blue] (2.4,0)--(2.4,0.6)--(10.4,0.6)--(10.4,7);

\draw[red, ultra thick, ->] (2.6,7)--(2.6,4.6);
\draw[blue, ultra thick, ->] (2.6,4.6)--(3.6,4.6);
\draw[red, ultra thick, ->] (3.6,4.6)--(3.6,3.4)--(2.4,3.4);
\draw[blue, ultra thick, ->] (2.4,3.4)--(2.4,4.4);
\draw[red, ultra thick, ->] (2.4,4.4)--(0.6,4.4)--(0.6,2.6);
\draw[blue, ultra thick, ->] (0.6,2.6)--(1.6,2.6);
\draw[red, ultra thick, ->] (1.6,2.6)--(1.6,1.6);
\draw[blue, ultra thick, ->] (1.6,1.6)--(4.6,1.6);
\draw[red, ultra thick, ->] (4.6,1.6)--(4.6,0.6);
\draw[blue, ultra thick, ->] (4.6,0.6)--(10.4,0.6)--(10.4,7);

\draw[red, ultra thick, ->] (4.6,7)--(4.6,5.6);
\draw[blue, ultra thick, ->] (4.6,5.6)--(5.6,5.6);
\draw[red, ultra thick, ->] (5.6,5.6)--(5.6,4.6);
\draw[blue, ultra thick, ->] (5.6,4.6)--(8.4,4.6)--(8.4,7);

\draw[blue, ultra thick, ->] (2.4,0)--(2.4,0.6)--(4.6,0.6);
\draw[red, ultra thick, ->] (4.6,0.6)--(4.6,0);
\end{tikzpicture}
}
\overset {\Phi} \mapsto
\resizebox{0.4\textwidth}{!}{
\begin{tikzpicture}[baseline=(current bounding box.center)] 
\draw[help lines] (0,0) grid (11,7);

\draw[red] (0.6,0)--(0.6,2.4)--(1.6,2.4)--(1.6,3.4)--(2.6,3.4)--(2.6,4.4)--(3.6,4.4)--(3.6,5.4)--(5.6,5.4)--(5.6,6.4)--(6.6,6.4)--(6.6,7);
\draw[red] (1.6,0)--(1.6,1.4)--(4.6,1.4)--(4.6,2.4)--(4.6,2.4)--(4.6,4.4)--(8.6,4.4)--(8.6,7);
\draw[red] (2.6,0)--(2.6,0.4)--(10.6,0.4)--(10.6,7);

\draw[blue] (0.4,0)--(0.4,4.6)--(2.4,4.6)--(2.4,7);
\draw[blue] (1.4,0)--(1.4,2.6)--(2.4,2.6)--(2.4,3.6)--(3.4,3.6)--(3.4,4.6)--(4.4,4.6)--(4.4,7);
\draw[blue] (4.4,0)--(4.4,1.6)--(5.4,1.6)--(5.4,5.6)--(6.4,5.6)--(6.4,7);

\end{tikzpicture}
}
\]
where we swap the color of all highlighted path segments in the left configuration to get the right configuration.

\begin{remark}
    When drawing the path configurations after swapping the color of the highlighted path segments, we also shift the effected path segments in order to straighten the paths within a face. For instance, in the first column of the above example, after swapping colors, the now blue path segment is shifted slightly to the left, while the now red path segment is shifted slightly down. Of course, these shifts do not affect the basic procedure and are only to maintain consistency in how the path configurations are drawn.
\end{remark}

\begin{prop}
$\Phi$ is a bijection between configurations with boundary condition given by $\bg$ and configurations with boundary condition given by $(\bg)_{swap}$.
\end{prop}
\begin{proof}
First, one can see that swapping the color of all the highlighted path segments results in a valid configuration of the vertex model by checking all possible changes that can happen at a single face while obeying the rules in \eqref{eq:rules}.

Next, we must show that the boundary condition changes appropriately. Note that the difference between the boundary condition given by $\bg$ and those given by $(\bg)_{swap}$ is only a swapping of the colors of the paths. We must check that under $\Phi$ all singleton blue boundary paths become red and vice versa (we are not concerned with points on the boundary with both colors, as they appear in both the $\bg$ and $(\bg)_{swap}$ boundary conditions). It suffices to show that all singleton boundary path segments are traversed in our procedure. Suppose that this was not the case and there is, say, a singleton red boundary path on the bottom boundary is not included in a walk. Then starting at that path and running the procedure in reverse (that is, reversing all the arrows in the set of rules (\ref{eq:rules})) we must end at either a singleton red boundary path on the top or a singleton blue boundary path on the bottom. But this path segment will have already been included in a walk as we use it as a starting point. 

Finally, note that running this procedure again returns us to our original configuration, so $\Phi$ is invertible. 
\end{proof}

\begin{remark}
This algorithm can be seen as a generalization of the procedure used in proving Proposition 5.5 of \cite{CGKM}. There, the authors show that if each of the partitions in the indexing tuple of $\mathcal{L}_{(\lambda^{(1)},\ldots,\lambda^{(k)})}(X_n;t)$ is a single row, then any rearrangement of the order of the partitions will result in an equivalent LLT polynomial.
\end{remark}

\section{Calculating the weight}\label{sect:weight}
While $\Phi$ is always a bijection, we are interested in the case when the LLT polyinomials $\mathcal{L}_{\bg}(X_n,t)$ and $\mathcal{L}_{(\bg)_{swap}}(X_n,t)$ are equivalent. That is, we wish to determine the possible $\bg$ such that $\Phi$ is weight-preserving up to an overall power of $t$.  

Toward this end, we associate a sequence of colored beads to the boundary condition of the paths given by $\bg$ as follows: We consider two rows of beads. Scanning the columns of our lattice model from left-to-right, for every singleton path along the top boundary we add a bead of the same color to the top row and for every singleton path along the bottom boundary we add a bead of the same color to the bottom row. We ensure that the beads keep the same ordering (from left to right) as the paths.  We label the beads by the total number of boundary paths to the right of their corresponding path.

We define a {\em matching} of this sequence of beads to be a set of arcs such that either the arc connects two beads of different colors in the same row or the arc connects two beads of the same color in different rows. A matching is {\em non-crossing} if the arcs (drawn so they always remain between the two rows) do not cross one another.

In our running example, the sequence of beads associated to the boundary condition is
\[
\resizebox{0.4\textwidth}{!}{
\begin{tikzpicture}[baseline=(current bounding box.center)] 
\draw[help lines] (0,3) grid (11,4);
\draw[help lines] (0,0) grid (11,1);
\draw[help lines] (0,1)--(0,3); \draw[help lines] (11,1)--(11,3);
\draw[blue, very thick] (6.4,3.5)--(6.4,4); \draw[blue, very thick] (8.4,3.5)--(8.4,4); \draw[blue, very thick] (10.4,3.5)--(10.4,4); 
\draw[red, very thick] (2.6,3.5)--(2.6,4); \draw[red, very thick] (4.6,3.5)--(4.6,4);  \draw[red, very thick] (6.6,3.5)--(6.6,4); 
\draw[blue, very thick] (0.4,0)--(0.4,0.5); \draw[blue, very thick] (1.4,0)--(1.4,0.5); \draw[blue, very thick] (2.4,0)--(2.4,0.5);
\draw[red, very thick] (0.6,0)--(0.6,0.5); \draw[red, very thick] (1.6,0)--(1.6,0.5); \draw[red, very thick] (4.6,0)--(4.6,0.5);
\end{tikzpicture}
}
\mapsto
\begin{tikzpicture}[baseline=(current bounding box.center)] 
\draw[dashed] (-0.2,2)--(3.2,2); \draw[dashed] (-0.2,0)--(3.2,0);
\draw[red, fill=red] (0,2) circle (3pt);\draw[red, fill=red] (1,2) circle (3pt); \draw[blue, fill=blue] (2,2) circle (3pt); \draw[blue, fill=blue] (3,2) circle (3pt);
\draw[blue, fill=blue] (1,0) circle (3pt); \draw[red, fill=red] (2,0) circle (3pt);
\node[above] at (0,2) {5}; \node[above] at (1,2) {4}; \node[above] at (2,2) {1}; \node[above] at (3,2) {0};
\node[below] at (1,0) {1}; \node[below] at (2,0) {0};
\end{tikzpicture}
\]

Through the procedure in Section \ref{sect:procedure}, we can associate a non-crossing matching to every configuration of the vertex model. Each walk created by starting the procedure at a singleton path corresponds to an arc in the matching.  The non-crossing matching associated to our example configuration is 
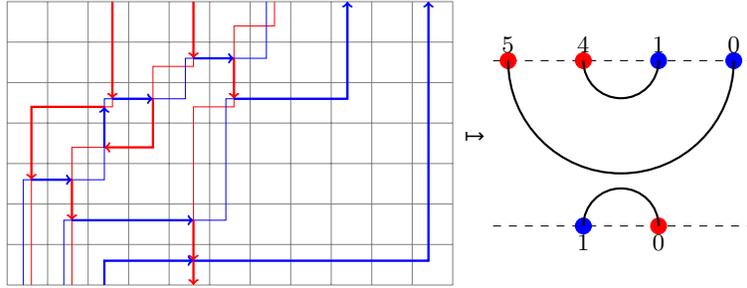
\[
\resizebox{0.4\textwidth}{!}{
\begin{tikzpicture}[baseline=(current bounding box.center)] 
\draw[help lines] (0,0) grid (11,7);

\draw[red] (0.6,0)--(0.6,4.4)--(2.6,4.4)--(2.6,7);
\draw[red] (1.6,0)--(1.6,3.4)--(3.6,3.4)--(3.6,5.4)--(4.6,5.4)--(4.6,7);
\draw[red] (4.6,0)--(4.6,2.4)--(4.6,2.4)--(4.6,4.4)--(5.6,4.4)--(5.6,6.4)--(6.6,6.4)--(6.6,7);

\draw[blue] (0.4,0)--(0.4,2.6)--(2.4,2.6)--(2.4,4.6)--(4.4,4.6)--(4.4,5.6)--(6.4,5.6)--(6.4,7);
\draw[blue] (1.4,0)--(1.4,1.6)--(5.4,1.6)--(5.4,4.6)--(8.4,4.6)--(8.4,7);
\draw[blue] (2.4,0)--(2.4,0.6)--(10.4,0.6)--(10.4,7);

\draw[red, ultra thick, ->] (2.6,7)--(2.6,4.6);
\draw[blue, ultra thick, ->] (2.6,4.6)--(3.6,4.6);
\draw[red, ultra thick, ->] (3.6,4.6)--(3.6,3.4)--(2.4,3.4);
\draw[blue, ultra thick, ->] (2.4,3.4)--(2.4,4.4);
\draw[red, ultra thick, ->] (2.4,4.4)--(0.6,4.4)--(0.6,2.6);
\draw[blue, ultra thick, ->] (0.6,2.6)--(1.6,2.6);
\draw[red, ultra thick, ->] (1.6,2.6)--(1.6,1.6);
\draw[blue, ultra thick, ->] (1.6,1.6)--(4.6,1.6);
\draw[red, ultra thick, ->] (4.6,1.6)--(4.6,0.6);
\draw[blue, ultra thick, ->] (4.6,0.6)--(10.4,0.6)--(10.4,7);

\draw[red, ultra thick, ->] (4.6,7)--(4.6,5.6);
\draw[blue, ultra thick, ->] (4.6,5.6)--(5.6,5.6);
\draw[red, ultra thick, ->] (5.6,5.6)--(5.6,4.6);
\draw[blue, ultra thick, ->] (5.6,4.6)--(8.4,4.6)--(8.4,7);

\draw[blue, ultra thick, ->] (2.4,0)--(2.4,0.6)--(4.6,0.6);
\draw[red, ultra thick, ->] (4.6,0.6)--(4.6,0);
\end{tikzpicture}
}
\mapsto
\begin{tikzpicture}[baseline=(current bounding box.center)] 
\draw[dashed] (-0.2,2.2)--(3.2,2.2); \draw[dashed] (-0.2,0)--(3.2,0);
\draw[red, fill=red] (0,2.2) circle (3pt);\draw[red, fill=red] (1,2.2) circle (3pt); \draw[blue, fill=blue] (2,2.2) circle (3pt); \draw[blue, fill=blue] (3,2.2) circle (3pt);
\draw[blue, fill=blue] (1,0) circle (3pt); \draw[red, fill=red] (2,0) circle (3pt);
\node[above] at (0,2.2) {5}; \node[above] at (1,2.2) {4}; \node[above] at (2,2.2) {1}; \node[above] at (3,2.2) {0}; 
\node[below] at (1,0) {1}; \node[below] at (2,0) {0};
\draw[thick] (0,2.2) arc(180:360:1.5); \draw[thick] (1,2.2) arc(180:360:0.5);
\draw[thick] (2,0) arc(0:180:0.5);
\end{tikzpicture}
\]

This mapping from vertex model configurations to non-crossing matchings is useful since, as we will see, the change in the power of $t$ under $\Phi$ depends only on the matching. As a preliminary step, we first prove several lemmas.
\begin{lem}\label{lem:dircount}
For a walk starting and ending on the same boundary we have
\[
\#\left(\begin{tikzpicture}[baseline=(current bounding box.center)] 
\draw (0,0)--(1,0)--(1,1)--(0,1)--(0,0);
\draw[red] (0.6,0.4)--(0.6,1);
\draw[blue] (0.4,0.6)--(1,0.6);
\node[left] at (0.6,0.8) {$\downarrow$};
\node[below] at (0.8,0.6) {$\rightarrow$};
\end{tikzpicture}\right) 
+
\#\left( \begin{tikzpicture}[baseline=(current bounding box.center)] 
\draw (0,0)--(1,0)--(1,1)--(0,1)--(0,0);
\draw[blue] (0.4,0.6)--(0.4,1);
\draw[red] (0.6,0.4)--(1,0.4);
\node[left] at (0.4,0.8) {$\uparrow$};
\node[below] at (0.8,0.4) {$\leftarrow$};
\end{tikzpicture}\right)
- 
\#\left( \begin{tikzpicture}[baseline=(current bounding box.center)] 
\draw (0,0)--(1,0)--(1,1)--(0,1)--(0,0);
\draw[blue] (0,0.6)--(0.4,0.6);
\draw[red] (0.6,0)--(0.6,0.4);
\node[above] at (0.2,0.6) {$\rightarrow$};
\node[right] at (0.6,0.2) {$\downarrow$};
\end{tikzpicture}\right)
-
\#\left(\begin{tikzpicture}[baseline=(current bounding box.center)] 
\draw (0,0)--(1,0)--(1,1)--(0,1)--(0,0);
\draw[red] (0,0.4)--(0.6,0.4);
\draw[blue] (0.4,0)--(0.4,0.6);
\node[above] at (0.2,0.4) {$\leftarrow$};
\node[right] at (0.4,0.2) {$\uparrow$};
\end{tikzpicture} \right) = \pm1,
\]
where the RHS is +1 for walks starting and ending on the top boundary and -1 for walks starting and ending on the bottom boundary. 

For a walk starting on one boundary and ending on the other we have
\[
\#\left(\begin{tikzpicture}[baseline=(current bounding box.center)] 
\draw (0,0)--(1,0)--(1,1)--(0,1)--(0,0);
\draw[red] (0.6,0.4)--(0.6,1);
\draw[blue] (0.4,0.6)--(1,0.6);
\node[left] at (0.6,0.8) {$\downarrow$};
\node[below] at (0.8,0.6) {$\rightarrow$};
\end{tikzpicture}\right) 
+
\#\left( \begin{tikzpicture}[baseline=(current bounding box.center)] 
\draw (0,0)--(1,0)--(1,1)--(0,1)--(0,0);
\draw[blue] (0.4,0.6)--(0.4,1);
\draw[red] (0.6,0.4)--(1,0.4);
\node[left] at (0.4,0.8) {$\uparrow$};
\node[below] at (0.8,0.4) {$\leftarrow$};
\end{tikzpicture}\right)
- 
\#\left( \begin{tikzpicture}[baseline=(current bounding box.center)] 
\draw (0,0)--(1,0)--(1,1)--(0,1)--(0,0);
\draw[blue] (0,0.6)--(0.4,0.6);
\draw[red] (0.6,0)--(0.6,0.4);
\node[above] at (0.2,0.6) {$\rightarrow$};
\node[right] at (0.6,0.2) {$\downarrow$};
\end{tikzpicture}\right)
-
\#\left(\begin{tikzpicture}[baseline=(current bounding box.center)] 
\draw (0,0)--(1,0)--(1,1)--(0,1)--(0,0);
\draw[red] (0,0.4)--(0.6,0.4);
\draw[blue] (0.4,0)--(0.4,0.6);
\node[above] at (0.2,0.4) {$\leftarrow$};
\node[right] at (0.4,0.2) {$\uparrow$};
\end{tikzpicture} \right) = 0.
\]
\end{lem} 
\begin{proof}
The first two terms in the summand give the number of times we switch from traveling backward to forward, the second two terms give the number of times we switch from traveling forward to backward. For a walk that starts and ends on the same boundary, it must end traveling in the opposite direction of how it starts. So the net difference must be 1. If the walk starts on the top boundary, it starts going backward so there must be an extra switch to traveling forward, and the opposite for the bottom boundary.  

For a walk that crosses between the boundaries, it must end traveling in the same direction that it starts. So the net difference must be 0.
\end{proof}
%\begin{remark}
%Note that Lemma \ref{lem:dircount} is not used in anything that follows.
%\end{remark}

\begin{lem}\label{lem:pathcount}
For a walk starting at a path segment with $j$ paths to its right and ending at a path segment with $i$ paths to its right, we have
\[
\#\left(\begin{tikzpicture}[baseline=(current bounding box.center)] 
\draw (0,0)--(1,0)--(1,1)--(0,1)--(0,0);
\draw[red] (0.6,0.4)--(0.6,1);
\draw[blue] (0.4,0.6)--(1,0.6);
\node[left] at (0.6,0.8) {$\downarrow$};
\node[below] at (0.8,0.6) {$\rightarrow$};
\end{tikzpicture}\right) 
-
\#\left( \begin{tikzpicture}[baseline=(current bounding box.center)] 
\draw (0,0)--(1,0)--(1,1)--(0,1)--(0,0);
\draw[blue] (0.4,0.6)--(0.4,1);
\draw[red] (0.6,0.4)--(1,0.4);
\node[left] at (0.4,0.8) {$\uparrow$};
\node[below] at (0.8,0.4) {$\leftarrow$};
\end{tikzpicture}\right)
+
\#\left( \begin{tikzpicture}[baseline=(current bounding box.center)] 
\draw (0,0)--(1,0)--(1,1)--(0,1)--(0,0);
\draw[blue] (0,0.6)--(0.4,0.6);
\draw[red] (0.6,0)--(0.6,0.4);
\node[above] at (0.2,0.6) {$\rightarrow$};
\node[right] at (0.6,0.2) {$\downarrow$};
\end{tikzpicture}\right)
-
\#\left(\begin{tikzpicture}[baseline=(current bounding box.center)] 
\draw (0,0)--(1,0)--(1,1)--(0,1)--(0,0);
\draw[red] (0,0.4)--(0.6,0.4);
\draw[blue] (0.4,0)--(0.4,0.6);
\node[above] at (0.2,0.4) {$\leftarrow$};
\node[right] at (0.4,0.2) {$\uparrow$};
\end{tikzpicture} \right) = j-i.
\]
\end{lem}
\begin{proof}
Steps of the form 
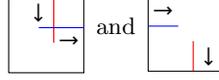
\[
\begin{tikzpicture}[baseline=(current bounding box.center)] 
\draw (0,0)--(1,0)--(1,1)--(0,1)--(0,0);
\draw[red] (0.6,0.4)--(0.6,1);
\draw[blue] (0.4,0.6)--(1,0.6);
\node[left] at (0.6,0.8) {$\downarrow$};
\node[below] at (0.8,0.6) {$\rightarrow$};
\end{tikzpicture}
\text{ and }
\begin{tikzpicture}[baseline=(current bounding box.center)] 
\draw (0,0)--(1,0)--(1,1)--(0,1)--(0,0);
\draw[blue] (0,0.6)--(0.4,0.6);
\draw[red] (0.6,0)--(0.6,0.4);
\node[above] at (0.2,0.6) {$\rightarrow$};
\node[right] at (0.6,0.2) {$\downarrow$};
\end{tikzpicture}
\]
always put us on a new segment of path that has one fewer paths to its right. Steps of the form 
\[
\begin{tikzpicture}[baseline=(current bounding box.center)] 
\draw (0,0)--(1,0)--(1,1)--(0,1)--(0,0);
\draw[blue] (0.4,0.6)--(0.4,1);
\draw[red] (0.6,0.4)--(1,0.4);
\node[left] at (0.4,0.8) {$\uparrow$};
\node[below] at (0.8,0.4) {$\leftarrow$};
\end{tikzpicture}
\text{ and }
\begin{tikzpicture}[baseline=(current bounding box.center)] 
\draw (0,0)--(1,0)--(1,1)--(0,1)--(0,0);
\draw[red] (0,0.4)--(0.6,0.4);
\draw[blue] (0.4,0)--(0.4,0.6);
\node[above] at (0.2,0.4) {$\leftarrow$};
\node[right] at (0.4,0.2) {$\uparrow$};
\end{tikzpicture}
\]
always puts us on a new path segment with one more path to its right.  As the walk starts with $j$ paths to its right and ends with $i$ paths to its right, we have the identity.
\end{proof}

\begin{lem}\label{lem:restriction} 
A 
\begin{tikzpicture}[baseline=(current bounding box.center)] 
\draw (0,0)--(1,0)--(1,1)--(0,1)--(0,0);
\draw[red] (0,0.4)--(0.6,0.4);
\draw[blue] (0.4,0)--(0.4,0.6);
\node[above] at (0.2,0.4) {$\leftarrow$};
\node[right] at (0.4,0.2) {$\uparrow$};
\end{tikzpicture} 
cannot be immediately followed and preceded by a 
\begin{tikzpicture}[baseline=(current bounding box.center)] 
\draw (0,0)--(1,0)--(1,1)--(0,1)--(0,0);
\draw[red] (0.6,0.4)--(0.6,1);
\draw[blue] (0.4,0.6)--(1,0.6);
\node[left] at (0.6,0.8) {$\downarrow$};
\node[below] at (0.8,0.6) {$\rightarrow$};
\end{tikzpicture}
and vice versa.  Similarly, a
\begin{tikzpicture}[baseline=(current bounding box.center)] 
\draw (0,0)--(1,0)--(1,1)--(0,1)--(0,0);
\draw[blue] (0.4,0.6)--(0.4,1);
\draw[red] (0.6,0.4)--(1,0.4);
\node[left] at (0.4,0.8) {$\uparrow$};
\node[below] at (0.8,0.4) {$\leftarrow$};
\end{tikzpicture}
cannot be immediately  followed and preceded by a
\begin{tikzpicture}[baseline=(current bounding box.center)] 
\draw (0,0)--(1,0)--(1,1)--(0,1)--(0,0);
\draw[blue] (0,0.6)--(0.4,0.6);
\draw[red] (0.6,0)--(0.6,0.4);
\node[above] at (0.2,0.6) {$\rightarrow$};
\node[right] at (0.6,0.2) {$\downarrow$};
\end{tikzpicture}
and vice versa.  
\end{lem}
\begin{proof}
Consider a face in which the walk takes a step of the form
\begin{tikzpicture}[baseline=(current bounding box.center)] 
\draw (0,0)--(1,0)--(1,1)--(0,1)--(0,0);
\draw[red] (0,0.4)--(0.6,0.4);
\draw[blue] (0.4,0)--(0.4,0.6);
\node[above] at (0.2,0.4) {$\leftarrow$};
\node[right] at (0.4,0.2) {$\uparrow$};
\end{tikzpicture}. 
Suppose that it was immediately followed and preceded by  
\begin{tikzpicture}[baseline=(current bounding box.center)] 
\draw (0,0)--(1,0)--(1,1)--(0,1)--(0,0);
\draw[red] (0.6,0.4)--(0.6,1);
\draw[blue] (0.4,0.6)--(1,0.6);
\node[left] at (0.6,0.8) {$\downarrow$};
\node[below] at (0.8,0.6) {$\rightarrow$};
\end{tikzpicture}. Then we must have a configuration of the form
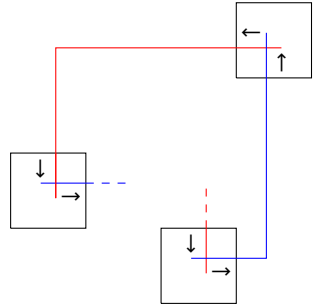
\[
\begin{tikzpicture}[baseline=(current bounding box.center)] 
\draw (2,2)--(3,2)--(3,3)--(2,3)--(2,2);
\draw[red] (2,2.4)--(2.6,2.4);
\draw[blue] (2.4,2)--(2.4,2.6);
\node[above] at (2.2,2.4) {$\leftarrow$};
\node[right] at (2.4,2.2) {$\uparrow$};

\draw (-1,0)--(0,0)--(0,1)--(-1,1)--(-1,0);
\draw[red] (-0.4,0.4)--(-0.4,1);
\draw[blue] (-0.6,0.6)--(0,0.6);
\node[left] at (-0.4,0.8) {$\downarrow$};
\node[below] at (-0.2,0.6) {$\rightarrow$};

\draw (1,-1)--(2,-1)--(2,0)--(1,0)--(1,-1);
\draw[red] (1.6,-0.6)--(1.6,0);
\draw[blue] (1.4,-0.4)--(2,-0.4);
\node[left] at (1.6,-0.2) {$\downarrow$};
\node[below] at (1.8,-0.4) {$\rightarrow$};

\draw[blue] (2,-0.4)--(2.4,-0.4)--(2.4,2);
\draw[red] (2,2.4)--(-0.4,2.4)--(-0.4,0.4);
\draw[blue, dashed] (0,0.6)--(0.6,0.6); 
\draw[red, dashed] (1.6,0)--(1.6,0.6);

\end{tikzpicture}.
\]
No matter how the two lower faces are arranged, either the dashed blue path segment or the dashed red path segment must intersect one of the solid path segments. This would mean that there is a step of our walk in between those drawn, contradicting that these steps immediately follow and precede \begin{tikzpicture}[baseline=(current bounding box.center)] 
\draw (0,0)--(1,0)--(1,1)--(0,1)--(0,0);
\draw[red] (0,0.4)--(0.6,0.4);
\draw[blue] (0.4,0)--(0.4,0.6);
\node[above] at (0.2,0.4) {$\leftarrow$};
\node[right] at (0.4,0.2) {$\uparrow$};
\end{tikzpicture}.

The other cases can be done similarly.
\end{proof}
\begin{lem}\label{lem:cornercount}
For a walk starting and ending on the top boundary we have
\[
\#\left(\begin{tikzpicture}[baseline=(current bounding box.center)] 
\draw (0,0)--(1,0)--(1,1)--(0,1)--(0,0);
\draw[red] (0.6,0.4)--(0.6,1);
\draw[blue] (0.4,0.6)--(1,0.6);
\node[left] at (0.6,0.8) {$\downarrow$};
\node[below] at (0.8,0.6) {$\rightarrow$};
\end{tikzpicture}\right) 
-
\#\left( \begin{tikzpicture}[baseline=(current bounding box.center)] 
\draw (0,0)--(1,0)--(1,1)--(0,1)--(0,0);
\draw[blue] (0.4,0.6)--(0.4,1);
\draw[red] (0.6,0.4)--(1,0.4);
\node[left] at (0.4,0.8) {$\uparrow$};
\node[below] at (0.8,0.4) {$\leftarrow$};
\end{tikzpicture}\right)
- 
\#\left( \begin{tikzpicture}[baseline=(current bounding box.center)] 
\draw (0,0)--(1,0)--(1,1)--(0,1)--(0,0);
\draw[blue] (0,0.6)--(0.4,0.6);
\draw[red] (0.6,0)--(0.6,0.4);
\node[above] at (0.2,0.6) {$\rightarrow$};
\node[right] at (0.6,0.2) {$\downarrow$};
\end{tikzpicture}\right)
+
\#\left(\begin{tikzpicture}[baseline=(current bounding box.center)] 
\draw (0,0)--(1,0)--(1,1)--(0,1)--(0,0);
\draw[red] (0,0.4)--(0.6,0.4);
\draw[blue] (0.4,0)--(0.4,0.6);
\node[above] at (0.2,0.4) {$\leftarrow$};
\node[right] at (0.4,0.2) {$\uparrow$};
\end{tikzpicture} \right) = \pm1
\]
where the RHS is $+1$ when the walk ends to the right of where it starts and $-1$ if it ends to the left. For a walk starting and ending on the bottom boundary we have
\[
\#\left(\begin{tikzpicture}[baseline=(current bounding box.center)] 
\draw (0,0)--(1,0)--(1,1)--(0,1)--(0,0);
\draw[red] (0.6,0.4)--(0.6,1);
\draw[blue] (0.4,0.6)--(1,0.6);
\node[left] at (0.6,0.8) {$\downarrow$};
\node[below] at (0.8,0.6) {$\rightarrow$};
\end{tikzpicture}\right) 
-
\#\left( \begin{tikzpicture}[baseline=(current bounding box.center)] 
\draw (0,0)--(1,0)--(1,1)--(0,1)--(0,0);
\draw[blue] (0.4,0.6)--(0.4,1);
\draw[red] (0.6,0.4)--(1,0.4);
\node[left] at (0.4,0.8) {$\uparrow$};
\node[below] at (0.8,0.4) {$\leftarrow$};
\end{tikzpicture}\right)
- 
\#\left( \begin{tikzpicture}[baseline=(current bounding box.center)] 
\draw (0,0)--(1,0)--(1,1)--(0,1)--(0,0);
\draw[blue] (0,0.6)--(0.4,0.6);
\draw[red] (0.6,0)--(0.6,0.4);
\node[above] at (0.2,0.6) {$\rightarrow$};
\node[right] at (0.6,0.2) {$\downarrow$};
\end{tikzpicture}\right)
+
\#\left(\begin{tikzpicture}[baseline=(current bounding box.center)] 
\draw (0,0)--(1,0)--(1,1)--(0,1)--(0,0);
\draw[red] (0,0.4)--(0.6,0.4);
\draw[blue] (0.4,0)--(0.4,0.6);
\node[above] at (0.2,0.4) {$\leftarrow$};
\node[right] at (0.4,0.2) {$\uparrow$};
\end{tikzpicture} \right) = \pm1
\]
where the RHS is $+1$ when the walk ends to the left of where it starts and $-1$ if it ends to the right. For a walk starting on one boundary and ending on the other we have
\[
\#\left(\begin{tikzpicture}[baseline=(current bounding box.center)] 
\draw (0,0)--(1,0)--(1,1)--(0,1)--(0,0);
\draw[red] (0.6,0.4)--(0.6,1);
\draw[blue] (0.4,0.6)--(1,0.6);
\node[left] at (0.6,0.8) {$\downarrow$};
\node[below] at (0.8,0.6) {$\rightarrow$};
\end{tikzpicture}\right) 
-
\#\left( \begin{tikzpicture}[baseline=(current bounding box.center)] 
\draw (0,0)--(1,0)--(1,1)--(0,1)--(0,0);
\draw[blue] (0.4,0.6)--(0.4,1);
\draw[red] (0.6,0.4)--(1,0.4);
\node[left] at (0.4,0.8) {$\uparrow$};
\node[below] at (0.8,0.4) {$\leftarrow$};
\end{tikzpicture}\right)
- 
\#\left( \begin{tikzpicture}[baseline=(current bounding box.center)] 
\draw (0,0)--(1,0)--(1,1)--(0,1)--(0,0);
\draw[blue] (0,0.6)--(0.4,0.6);
\draw[red] (0.6,0)--(0.6,0.4);
\node[above] at (0.2,0.6) {$\rightarrow$};
\node[right] at (0.6,0.2) {$\downarrow$};
\end{tikzpicture}\right)
+
\#\left(\begin{tikzpicture}[baseline=(current bounding box.center)] 
\draw (0,0)--(1,0)--(1,1)--(0,1)--(0,0);
\draw[red] (0,0.4)--(0.6,0.4);
\draw[blue] (0.4,0)--(0.4,0.6);
\node[above] at (0.2,0.4) {$\leftarrow$};
\node[right] at (0.4,0.2) {$\uparrow$};
\end{tikzpicture} \right) = 0.
\]\end{lem}
\begin{proof}
Let us consider the case where the walk starts and ends on the top boundary. Note that the walk is made up of straight sections of paths and right-angled corners. These corners can either occur on a path segment or at a face where the walk switches from one segment to the other. 

Suppose that the walk ends to the right of where it starts. Then we must have that $\#($left turns$) = \#($right turns$)+2$, where left and right are defined relative to the direction of travel of the walk. We can restate the lemma as
\[
\#(\text{left turns where the walk switches color}) - \#(\text{right turns where the walk switches colors})  = 1.
\]

For turns that occur on a path segment, the difference between the number of left turns and the number of right turns can only be $\pm1$ or 0. In fact, we can say precisely how this difference depends on the corners the path segment starts and ends at:
\[
\begin{aligned}
\begin{tikzpicture}[baseline=(current bounding box.center)] 
\draw (0,0) rectangle (0.5,0.5); \draw (0.5,0.25)--(1,0.25);\draw (1,0) rectangle (1.5,0.5);
\node at (0.25,0.25) {$R$}; \node[above] at (0.75,0.25) {$R$}; \node at (1.25,0.25) {$R$};
\end{tikzpicture} \\
\begin{tikzpicture}[baseline=(current bounding box.center)] 
\draw (0,0) rectangle (0.5,0.5); \draw (0.5,0.25)--(1,0.25);\draw (1,0) rectangle (1.5,0.5);
\node at (0.25,0.25) {$R$}; \node[above] at (0.75,0.25) {$0$}; \node at (1.25,0.25) {$L$};
\end{tikzpicture} \\
\begin{tikzpicture}[baseline=(current bounding box.center)] 
\draw (0,0) rectangle (0.5,0.5); \draw (0.5,0.25)--(1,0.25);\draw (1,0) rectangle (1.5,0.5);
\node at (0.25,0.25) {$L$}; \node[above] at (0.75,0.25) {$0$}; \node at (1.25,0.25) {$R$};
\end{tikzpicture} \\
\begin{tikzpicture}[baseline=(current bounding box.center)] 
\draw (0,0) rectangle (0.5,0.5); \draw (0.5,0.25)--(1,0.25);\draw (1,0) rectangle (1.5,0.5);
\node at (0.25,0.25) {$L$}; \node[above] at (0.75,0.25) {$L$}; \node at (1.25,0.25) {$L$};
\end{tikzpicture}
\end{aligned}
\]
where the box
 \begin{tikzpicture}[baseline=(current bounding box.center)] 
\draw (0,0) rectangle (0.75,0.75);
\node at (0.375,0.375) {$L/R$}; 
\end{tikzpicture} 
indicates a left/right turn at a corner where the walk switches color, and the line segment
\begin{tikzpicture}
\draw (0,0.25)--(1,0.25);
\node[above] at (0.5,0.25) {$L/R/0$}; 
\end{tikzpicture} 
 indicates the segment of path connecting the corners has a net left/right/equal number of turns. We also need to consider what can happen to the first and the last path segments, there are two possibilities for each
\[
\begin{tabular}{c|c}
\text{First: } &  \text{Last: } \\
\hline
\begin{tikzpicture}[baseline=(current bounding box.center)] 
\draw (0.5,0.25)--(1,0.25); \draw (1,0) rectangle (1.5,0.5);
\node[above] at (0.75,0.25) {$R$}; \node at (1.25,0.25) {$R$};
\end{tikzpicture} 
&
\begin{tikzpicture}[baseline=(current bounding box.center)] 
\draw (0,0) rectangle (0.5,0.5); \draw (0.5,0.25)--(1,0.25);
\node at (0.25,0.25) {$R$}; \node[above] at (0.75,0.25) {$0$}; 
\end{tikzpicture} 
\\
\begin{tikzpicture}[baseline=(current bounding box.center)] 
\draw (0.5,0.25)--(1,0.25);\draw (1,0) rectangle (1.5,0.5);
\node[above] at (0.75,0.25) {$0$}; \node at (1.25,0.25) {$L$};
\end{tikzpicture} 
&
\begin{tikzpicture}[baseline=(current bounding box.center)] 
\draw (0,0) rectangle (0.5,0.5); \draw (0.5,0.25)--(1,0.25);
\node at (0.25,0.25) {$L$}; \node[above] at (0.75,0.25) {$L$};
\end{tikzpicture}
\end{tabular}
\]

\noindent For the walk in example (\ref{ex:onewalk}), reproduced below, the sequence of turns is given by
\begin{center}
\resizebox{0.4\textwidth}{!}{
\begin{tikzpicture}[baseline=(current bounding box.center)] 
\draw[help lines] (0,0) grid (11,7);
\draw[red, ultra thick, ->] (2.6,7)--(2.6,4.6);
\draw[blue, ultra thick, ->] (2.6,4.6)--(3.6,4.6);
\draw[red, ultra thick, ->] (3.6,4.6)--(3.6,3.4)--(2.4,3.4);
\draw[blue, ultra thick, ->] (2.4,3.4)--(2.4,4.4);
\draw[red, ultra thick, ->] (2.4,4.4)--(0.6,4.4)--(0.6,2.6);
\draw[blue, ultra thick, ->] (0.6,2.6)--(1.6,2.6);
\draw[red, ultra thick, ->] (1.6,2.6)--(1.6,1.6);
\draw[blue, ultra thick, ->] (1.6,1.6)--(4.6,1.6);
\draw[red, ultra thick, ->] (4.6,1.6)--(4.6,0.6);
\draw[blue, ultra thick, ->] (4.6,0.6)--(10.4,0.6)--(10.4,7);
\end{tikzpicture}
}
\\
\begin{tikzpicture}[baseline=(current bounding box.center)] 
\draw (0.5,0.25)--(1,0.25);\draw (1,0) rectangle (1.5,0.5);
\draw (1.5,0.25)--(2,0.25);\draw (2,0) rectangle (2.5,0.5);
\draw (2.5,0.25)--(3,0.25);\draw (3,0) rectangle (3.5,0.5);
\draw (3.5,0.25)--(4,0.25);\draw (4,0) rectangle (4.5,0.5);
\draw (4.5,0.25)--(5,0.25);\draw (5,0) rectangle (5.5,0.5);
\draw (5.5,0.25)--(6,0.25);\draw (6,0) rectangle (6.5,0.5);
\draw (6.5,0.25)--(7,0.25);\draw (7,0) rectangle (7.5,0.5);
\draw (7.5,0.25)--(8,0.25);\draw (8,0) rectangle (8.5,0.5);
\draw (8.5,0.25)--(9,0.25);\draw (9,0) rectangle (9.5,0.5);
\draw (9.5,0.25)--(10,0.25);

\node[above] at (0.75,0.25) {$0$}; \node at (1.25,0.25) {$L$};
\node[above] at (1.75,0.25) {$0$}; \node at (2.25,0.25) {$R$};
\node[above] at (2.75,0.25) {$R$}; \node at (3.25,0.25) {$R$};
\node[above] at (3.75,0.25) {$0$}; \node at (4.25,0.25) {$L$};
\node[above] at (4.75,0.25) {$L$}; \node at (5.25,0.25) {$L$};
\node[above] at (5.75,0.25) {$0$}; \node at (6.25,0.25) {$R$};
\node[above] at (6.75,0.25) {$0$}; \node at (7.25,0.25) {$L$};
\node[above] at (7.75,0.25) {$0$}; \node at (8.25,0.25) {$R$};
\node[above] at (8.75,0.25) {$0$}; \node at (9.25,0.25) {$L$};
\node[above] at (9.75,0.25) {$L$};
\end{tikzpicture}.
\end{center}

Lemma  \ref{lem:restriction} can be interpreted as saying that, looking only at turns occurring where the walk switches colors, we cannot have three rights or three lefts in a row.  That is, we cannot have 
\begin{center}
\begin{tikzpicture}[baseline=(current bounding box.center)] 
\draw (1,0) rectangle (1.5,0.5);
\draw (1.5,0.25)--(2,0.25);\draw (2,0) rectangle (2.5,0.5);
\draw (2.5,0.25)--(3,0.25);\draw (3,0) rectangle (3.5,0.5);
\node at (1.25,0.25) {$L$};
\node at (2.25,0.25) {$L$};
\node at (3.25,0.25) {$L$};
\end{tikzpicture} 
 \;\;\; or \;\;\; 
\begin{tikzpicture}[baseline=(current bounding box.center)] 
\draw (1,0) rectangle (1.5,0.5);
\draw (1.5,0.25)--(2,0.25);\draw (2,0) rectangle (2.5,0.5);
\draw (2.5,0.25)--(3,0.25);\draw (3,0) rectangle (3.5,0.5);
\node at (1.25,0.25) {$R$};
\node at (2.25,0.25) {$R$};
\node at (3.25,0.25) {$R$};
\end{tikzpicture}.

\end{center}

We see that left turns where the walk switches colors occur in pairs 
\begin{tikzpicture}
\draw (1,0) rectangle (1.5,0.5);
\draw (1.5,0.25)--(2,0.25);\draw (2,0) rectangle (2.5,0.5);
\node at (1.25,0.25) {$L$};
\node at (2.25,0.25) {$L$};
\node[above] at (1.75,0.25) {$L$}; 
\end{tikzpicture} 
or alone 
\begin{tikzpicture}
\draw (0.5,0.25)--(1,0.25);
\draw (1,0) rectangle (1.5,0.5);
\draw (1.5,0.25)--(2,0.25);
\node at (1.25,0.25) {$L$};
\end{tikzpicture},
similarly for right turns where the walk changes color.  Let $L_1$ (resp. $R_1$) be the number of times the left (resp. right) turns appear alone and $L_2$ (resp. $R_2$) be the number of times they form pairs. Each pair of left turns where the walk switches color contributes three left turns in total, as the path segment connecting them also contribute a left turn, while the lefts that occur alone only contribute a single left turn as the path segment connecting them to the adjacent rights contribute a net zero. Taking into the account the first and the last path segment, from the preceding discussion we have
\[
\begin{aligned}
\#(\text{left turns}) - \#(\text{right turns}) =  
\begin{cases}
L_1 + 3L_2 - R_1 - 3R_2 - 1, & \text{if (first,last)}=\left(
\begin{tikzpicture}
\draw (0.5,0.25)--(1,0.25); \draw (1,0) rectangle (1.5,0.5);
\node[above] at (0.75,0.25) {$R$}; \node at (1.25,0.25) {$R$};
\end{tikzpicture}, 
\begin{tikzpicture}
\draw (0,0) rectangle (0.5,0.5); \draw (0.5,0.25)--(1,0.25);
\node at (0.25,0.25) {$R$}; \node[above] at (0.75,0.25) {$0$}; 
\end{tikzpicture}\right)  \\
L_1 + 3L_2 - R_1 - 3R_2, & \text{if (first,last)}=\left(
\begin{tikzpicture}
\draw (0.5,0.25)--(1,0.25); \draw (1,0) rectangle (1.5,0.5);
\node[above] at (0.75,0.25) {$R$}; \node at (1.25,0.25) {$R$};
\end{tikzpicture}, 
\begin{tikzpicture}
\draw (0,0) rectangle (0.5,0.5); \draw (0.5,0.25)--(1,0.25);
\node at (0.25,0.25) {$L$}; \node[above] at (0.75,0.25) {$L$}; 
\end{tikzpicture}\right)  \\ 
L_1 + 3L_2 - R_1 - 3R_2, & \text{if (first,last)}=\left(
\begin{tikzpicture}
\draw (0.5,0.25)--(1,0.25); \draw (1,0) rectangle (1.5,0.5);
\node[above] at (0.75,0.25) {$0$}; \node at (1.25,0.25) {$L$};
\end{tikzpicture}, 
\begin{tikzpicture}
\draw (0,0) rectangle (0.5,0.5); \draw (0.5,0.25)--(1,0.25);
\node at (0.25,0.25) {$R$}; \node[above] at (0.75,0.25) {$0$}; 
\end{tikzpicture}\right)  \\
L_1 + 3L_2 - R_1 - 3R_2 + 1, & \text{if (first,last)}=\left(
\begin{tikzpicture}
\draw (0.5,0.25)--(1,0.25); \draw (1,0) rectangle (1.5,0.5);
\node[above] at (0.75,0.25) {$0$}; \node at (1.25,0.25) {$L$};
\end{tikzpicture}, 
\begin{tikzpicture}
\draw (0,0) rectangle (0.5,0.5); \draw (0.5,0.25)--(1,0.25);
\node at (0.25,0.25) {$L$}; \node[above] at (0.75,0.25) {$L$}; 
\end{tikzpicture}\right)  \\
\end{cases} = 2.
\end{aligned}
\]
Counting the number of the lefts and rights, we must also have
\[
L_1+L_2-R_1-R_2 = 
\begin{cases}
- 1, & \text{if (first,last)}=\left(
\begin{tikzpicture}
\draw (0.5,0.25)--(1,0.25); \draw (1,0) rectangle (1.5,0.5);
\node[above] at (0.75,0.25) {$R$}; \node at (1.25,0.25) {$R$};
\end{tikzpicture}, 
\begin{tikzpicture}
\draw (0,0) rectangle (0.5,0.5); \draw (0.5,0.25)--(1,0.25);
\node at (0.25,0.25) {$R$}; \node[above] at (0.75,0.25) {$0$}; 
\end{tikzpicture}\right)  \\
0, & \text{if (first,last)}=\left(
\begin{tikzpicture}
\draw (0.5,0.25)--(1,0.25); \draw (1,0) rectangle (1.5,0.5);
\node[above] at (0.75,0.25) {$R$}; \node at (1.25,0.25) {$R$};
\end{tikzpicture}, 
\begin{tikzpicture}
\draw (0,0) rectangle (0.5,0.5); \draw (0.5,0.25)--(1,0.25);
\node at (0.25,0.25) {$L$}; \node[above] at (0.75,0.25) {$L$}; 
\end{tikzpicture}\right)  \\ 
0, & \text{if (first,last)}=\left(
\begin{tikzpicture}
\draw (0.5,0.25)--(1,0.25); \draw (1,0) rectangle (1.5,0.5);
\node[above] at (0.75,0.25) {$0$}; \node at (1.25,0.25) {$L$};
\end{tikzpicture}, 
\begin{tikzpicture}
\draw (0,0) rectangle (0.5,0.5); \draw (0.5,0.25)--(1,0.25);
\node at (0.25,0.25) {$R$}; \node[above] at (0.75,0.25) {$0$}; 
\end{tikzpicture}\right)  \\
1, & \text{if (first,last)}=\left(
\begin{tikzpicture}
\draw (0.5,0.25)--(1,0.25); \draw (1,0) rectangle (1.5,0.5);
\node[above] at (0.75,0.25) {$0$}; \node at (1.25,0.25) {$L$};
\end{tikzpicture}, 
\begin{tikzpicture}
\draw (0,0) rectangle (0.5,0.5); \draw (0.5,0.25)--(1,0.25);
\node at (0.25,0.25) {$L$}; \node[above] at (0.75,0.25) {$L$}; 
\end{tikzpicture}\right)  \\
\end{cases}.
\]
Using the above equations, a little algebra shows
\[
\begin{aligned}
&\#(\text{left turns where walk switches color}) - \#(\text{right turns where walk switches color}) \\
&=  L_1+2L_2 - R_1 -2R_2 \\
&=  1
\end{aligned}
\]
as desired. If the walk instead ends to the left of where it starts we would have
\[
\#(\text{left turns}) - \#(\text{right turns}) = -2
\]
but otherwise the same equations. Computing gives
\[
\begin{aligned}
\#(\text{left turns where walk switches color}) - \#(\text{right turns where walk switches color}) = -1.
\end{aligned}
\]

\noindent Here we worked out the case where the walk starts and ends on the top boundary. A similar analysis works for the other walks.
\end{proof}

For each type of arc between the labeled beads define the weight of the arc by
\begin{equation}\label{eq:arcweight}
\begin{tabular}{ccccccc}
$a$: &
\begin{tikzpicture}[baseline=(current bounding box.center)] 
\draw[thick] (0,2)--(1,0);
\draw[blue, fill=blue] (0,2) circle (3pt); \node[above] at (0,2) {$i$};
\draw[blue, fill=blue] (1,0) circle (3pt); \node[below] at (1,0) {$j$};
\end{tikzpicture}
&
\begin{tikzpicture}[baseline=(current bounding box.center)] 
\draw[thick] (0,2)--(1,0);
\draw[red, fill=red] (0,2) circle (3pt); \node[above] at (0,2) {$j$};
\draw[red, fill=red] (1,0) circle (3pt); \node[below] at (1,0) {$i$};
\end{tikzpicture}
&
\begin{tikzpicture}[baseline=(current bounding box.center)] 
\draw[thick] (0,0) arc(180:360:0.5);
\draw[red, fill=red] (0,0) circle (3pt); \draw[blue, fill=blue] (1,0) circle (3pt);
\node[above] at (0,0) {$j$}; \node[above] at (1,0) {$i$}; 
\end{tikzpicture}
&
\begin{tikzpicture}[baseline=(current bounding box.center)] 
\draw[thick] (0,0) arc(180:360:0.5);
\draw[blue, fill=blue] (0,0) circle (3pt); \draw[red, fill=red] (1,0) circle (3pt);
\node[above] at (0,0) {$i$}; \node[above] at (1,0) {$j$}; 
\end{tikzpicture}
&
\begin{tikzpicture}[baseline=(current bounding box.center)] 
\draw[thick] (1,0) arc(0:180:0.5);
\draw[red, fill=red] (0,0) circle (3pt); \draw[blue, fill=blue] (1,0) circle (3pt);
\node[below] at (0,0) {$i$}; \node[below] at (1,0) {$j$}; 
\end{tikzpicture}
&
\begin{tikzpicture}[baseline=(current bounding box.center)] 
\draw[thick] (1,0) arc(0:180:0.5);
\draw[blue, fill=blue] (0,0) circle (3pt); \draw[red, fill=red] (1,0) circle (3pt);
\node[below] at (0,0) {$j$}; \node[below] at (1,0) {$i$}; 
\end{tikzpicture}
\\
$w(a)$: & $t^{(j-i)/2}$ & $t^{(j-i)/2}$ & $t^{(j-i+1)/2}$ & $t^{(j-i-1)/2}$ & $t^{(j-i+1)/2}$ & $t^{(j-i-1)/2}$
\end{tabular}.
\end{equation}
\begin{prop}\label{prop:tpower}
For a single walk, the change in the power of $t$ after swapping the colors of all the path segments traversed in the walk depends only on its corresponding arc in the matching. In particular, the change in the power of $t$ is equal to the weight of the arc.
\end{prop}
\begin{proof}
The change in the power of $t$ after swapping the colors of all the path segments of a walk is given by
\[
\#\left(\begin{tikzpicture}[baseline=(current bounding box.center)] 
\draw (0,0)--(1,0)--(1,1)--(0,1)--(0,0);
\draw[red] (0.6,0.4)--(0.6,1);
\draw[blue] (0.4,0.6)--(1,0.6);
\node[left] at (0.6,0.8) {$\downarrow$};
\node[below] at (0.8,0.6) {$\rightarrow$};
\end{tikzpicture}\right) 
-
\#\left( \begin{tikzpicture}[baseline=(current bounding box.center)] 
\draw (0,0)--(1,0)--(1,1)--(0,1)--(0,0);
\draw[blue] (0.4,0.6)--(0.4,1);
\draw[red] (0.6,0.4)--(1,0.4);
\node[left] at (0.4,0.8) {$\uparrow$};
\node[below] at (0.8,0.4) {$\leftarrow$};
\end{tikzpicture}\right).
\]
This can be computed using the previous lemmas. For example, suppose a walk starts on the top boundary and ends on the top boundary to the right of where it starts. From Lemmas \ref{lem:pathcount} and \ref{lem:cornercount} we have
\[
\begin{aligned}
\#\left(\begin{tikzpicture}[baseline=(current bounding box.center)] 
\draw (0,0)--(1,0)--(1,1)--(0,1)--(0,0);
\draw[red] (0.6,0.4)--(0.6,1);
\draw[blue] (0.4,0.6)--(1,0.6);
\node[left] at (0.6,0.8) {$\downarrow$};
\node[below] at (0.8,0.6) {$\rightarrow$};
\end{tikzpicture}\right) 
-
\#\left( \begin{tikzpicture}[baseline=(current bounding box.center)] 
\draw (0,0)--(1,0)--(1,1)--(0,1)--(0,0);
\draw[blue] (0.4,0.6)--(0.4,1);
\draw[red] (0.6,0.4)--(1,0.4);
\node[left] at (0.4,0.8) {$\uparrow$};
\node[below] at (0.8,0.4) {$\leftarrow$};
\end{tikzpicture}\right)
+
\#\left( \begin{tikzpicture}[baseline=(current bounding box.center)] 
\draw (0,0)--(1,0)--(1,1)--(0,1)--(0,0);
\draw[blue] (0,0.6)--(0.4,0.6);
\draw[red] (0.6,0)--(0.6,0.4);
\node[above] at (0.2,0.6) {$\rightarrow$};
\node[right] at (0.6,0.2) {$\downarrow$};
\end{tikzpicture}\right)
-
\#\left(\begin{tikzpicture}[baseline=(current bounding box.center)] 
\draw (0,0)--(1,0)--(1,1)--(0,1)--(0,0);
\draw[red] (0,0.4)--(0.6,0.4);
\draw[blue] (0.4,0)--(0.4,0.6);
\node[above] at (0.2,0.4) {$\leftarrow$};
\node[right] at (0.4,0.2) {$\uparrow$};
\end{tikzpicture} \right) & = j-i, \\
\#\left(\begin{tikzpicture}[baseline=(current bounding box.center)] 
\draw (0,0)--(1,0)--(1,1)--(0,1)--(0,0);
\draw[red] (0.6,0.4)--(0.6,1);
\draw[blue] (0.4,0.6)--(1,0.6);
\node[left] at (0.6,0.8) {$\downarrow$};
\node[below] at (0.8,0.6) {$\rightarrow$};
\end{tikzpicture}\right) 
-
\#\left( \begin{tikzpicture}[baseline=(current bounding box.center)] 
\draw (0,0)--(1,0)--(1,1)--(0,1)--(0,0);
\draw[blue] (0.4,0.6)--(0.4,1);
\draw[red] (0.6,0.4)--(1,0.4);
\node[left] at (0.4,0.8) {$\uparrow$};
\node[below] at (0.8,0.4) {$\leftarrow$};
\end{tikzpicture}\right)
- 
\#\left( \begin{tikzpicture}[baseline=(current bounding box.center)] 
\draw (0,0)--(1,0)--(1,1)--(0,1)--(0,0);
\draw[blue] (0,0.6)--(0.4,0.6);
\draw[red] (0.6,0)--(0.6,0.4);
\node[above] at (0.2,0.6) {$\rightarrow$};
\node[right] at (0.6,0.2) {$\downarrow$};
\end{tikzpicture}\right)
+
\#\left(\begin{tikzpicture}[baseline=(current bounding box.center)] 
\draw (0,0)--(1,0)--(1,1)--(0,1)--(0,0);
\draw[red] (0,0.4)--(0.6,0.4);
\draw[blue] (0.4,0)--(0.4,0.6);
\node[above] at (0.2,0.4) {$\leftarrow$};
\node[right] at (0.4,0.2) {$\uparrow$};
\end{tikzpicture} \right) & = 1.
\end{aligned}
\]
Adding them together gives 
\[
\#\left(\begin{tikzpicture}[baseline=(current bounding box.center)] 
\draw (0,0)--(1,0)--(1,1)--(0,1)--(0,0);
\draw[red] (0.6,0.4)--(0.6,1);
\draw[blue] (0.4,0.6)--(1,0.6);
\node[left] at (0.6,0.8) {$\downarrow$};
\node[below] at (0.8,0.6) {$\rightarrow$};
\end{tikzpicture}\right) 
-
\#\left( \begin{tikzpicture}[baseline=(current bounding box.center)] 
\draw (0,0)--(1,0)--(1,1)--(0,1)--(0,0);
\draw[blue] (0.4,0.6)--(0.4,1);
\draw[red] (0.6,0.4)--(1,0.4);
\node[left] at (0.4,0.8) {$\uparrow$};
\node[below] at (0.8,0.4) {$\leftarrow$};
\end{tikzpicture}\right) = (j-i+1)/2
\]
which agrees with
\[
w\left( 
\begin{tikzpicture}[baseline=(current bounding box.center)] 
\draw[thick] (0,0) arc(180:360:0.5);
\draw[red, fill=red] (0,0) circle (3pt); \draw[blue, fill=blue] (1,0) circle (3pt);
\node[above] at (0,0) {$j$}; \node[above] at (1,0) {$i$}; 
\end{tikzpicture} 
\right) = t^{(j-i+1)/2}.
\]
Note that this implies that $j-i+1$ is even.  Alternatively, we can see this since on the boundary between the starting and ending path segments there must be an even number of paths, either pairs of singleton boundary paths that are connected by a walk or non-singleton boundary paths. \\

For a walk starting at a red path on the top boundary and ending at a red path on the bottom boundary, Lemmas \ref{lem:pathcount} and \ref{lem:cornercount} give
\[
\begin{aligned}
\#\left(\begin{tikzpicture}[baseline=(current bounding box.center)] 
\draw (0,0)--(1,0)--(1,1)--(0,1)--(0,0);
\draw[red] (0.6,0.4)--(0.6,1);
\draw[blue] (0.4,0.6)--(1,0.6);
\node[left] at (0.6,0.8) {$\downarrow$};
\node[below] at (0.8,0.6) {$\rightarrow$};
\end{tikzpicture}\right) 
-
\#\left( \begin{tikzpicture}[baseline=(current bounding box.center)] 
\draw (0,0)--(1,0)--(1,1)--(0,1)--(0,0);
\draw[blue] (0.4,0.6)--(0.4,1);
\draw[red] (0.6,0.4)--(1,0.4);
\node[left] at (0.4,0.8) {$\uparrow$};
\node[below] at (0.8,0.4) {$\leftarrow$};
\end{tikzpicture}\right)
+
\#\left( \begin{tikzpicture}[baseline=(current bounding box.center)] 
\draw (0,0)--(1,0)--(1,1)--(0,1)--(0,0);
\draw[blue] (0,0.6)--(0.4,0.6);
\draw[red] (0.6,0)--(0.6,0.4);
\node[above] at (0.2,0.6) {$\rightarrow$};
\node[right] at (0.6,0.2) {$\downarrow$};
\end{tikzpicture}\right)
-
\#\left(\begin{tikzpicture}[baseline=(current bounding box.center)] 
\draw (0,0)--(1,0)--(1,1)--(0,1)--(0,0);
\draw[red] (0,0.4)--(0.6,0.4);
\draw[blue] (0.4,0)--(0.4,0.6);
\node[above] at (0.2,0.4) {$\leftarrow$};
\node[right] at (0.4,0.2) {$\uparrow$};
\end{tikzpicture} \right) & = j-i, \\
\#\left(\begin{tikzpicture}[baseline=(current bounding box.center)] 
\draw (0,0)--(1,0)--(1,1)--(0,1)--(0,0);
\draw[red] (0.6,0.4)--(0.6,1);
\draw[blue] (0.4,0.6)--(1,0.6);
\node[left] at (0.6,0.8) {$\downarrow$};
\node[below] at (0.8,0.6) {$\rightarrow$};
\end{tikzpicture}\right) 
-
\#\left( \begin{tikzpicture}[baseline=(current bounding box.center)] 
\draw (0,0)--(1,0)--(1,1)--(0,1)--(0,0);
\draw[blue] (0.4,0.6)--(0.4,1);
\draw[red] (0.6,0.4)--(1,0.4);
\node[left] at (0.4,0.8) {$\uparrow$};
\node[below] at (0.8,0.4) {$\leftarrow$};
\end{tikzpicture}\right)
- 
\#\left( \begin{tikzpicture}[baseline=(current bounding box.center)] 
\draw (0,0)--(1,0)--(1,1)--(0,1)--(0,0);
\draw[blue] (0,0.6)--(0.4,0.6);
\draw[red] (0.6,0)--(0.6,0.4);
\node[above] at (0.2,0.6) {$\rightarrow$};
\node[right] at (0.6,0.2) {$\downarrow$};
\end{tikzpicture}\right)
+
\#\left(\begin{tikzpicture}[baseline=(current bounding box.center)] 
\draw (0,0)--(1,0)--(1,1)--(0,1)--(0,0);
\draw[red] (0,0.4)--(0.6,0.4);
\draw[blue] (0.4,0)--(0.4,0.6);
\node[above] at (0.2,0.4) {$\leftarrow$};
\node[right] at (0.4,0.2) {$\uparrow$};
\end{tikzpicture} \right) & = 0
\end{aligned}
\]
from which we have
\[
\#\left(\begin{tikzpicture}[baseline=(current bounding box.center)] 
\draw (0,0)--(1,0)--(1,1)--(0,1)--(0,0);
\draw[red] (0.6,0.4)--(0.6,1);
\draw[blue] (0.4,0.6)--(1,0.6);
\node[left] at (0.6,0.8) {$\downarrow$};
\node[below] at (0.8,0.6) {$\rightarrow$};
\end{tikzpicture}\right) 
-
\#\left( \begin{tikzpicture}[baseline=(current bounding box.center)] 
\draw (0,0)--(1,0)--(1,1)--(0,1)--(0,0);
\draw[blue] (0.4,0.6)--(0.4,1);
\draw[red] (0.6,0.4)--(1,0.4);
\node[left] at (0.4,0.8) {$\uparrow$};
\node[below] at (0.8,0.4) {$\leftarrow$};
\end{tikzpicture}\right) = (j-i)/2
\]
which agrees with
\[
w\left( 
\begin{tikzpicture}[baseline=(current bounding box.center)] 
\draw[thick] (0.5,1)--(1,0);
\draw[red, fill=red] (0.5,1) circle (3pt); \node[above] at (0.5,1) {$j$};
\draw[red, fill=red] (1,0) circle (3pt); \node[below] at (1,0) {$i$};
\end{tikzpicture} 
\right) = t^{(j-i)/2}.
\]
Note that in this case we see that $j-i$ is even.

The same analysis can be done for walks starting on blue paths.
\end{proof}

Proposition \ref{prop:tpower} gives a sufficient condition for when the LLT polynomials $\mathcal{L}_{\bg}(X_n;t)$ and $\mathcal{L}_{(\bg)_{swap}}(X_n;t)$ are equivalent.
\begin{thm}\label{thm:swap}
If there is a unique non-crossing matching $M$ of the sequence of beads associated to $\bg$, then $\mathcal{L}_{\bg}(X_n;t)$ and $\mathcal{L}_{(\bg)_{swap}}(X_n;t)$ are equivalent. In particular, 
\[
\mathcal{L}_{\bg}(X_n;t) = \left(\prod_{a\in M} w(a)\right) \mathcal{L}_{(\bg)_{swap}}(X_n;t)
\]
where the product is over all arcs $a$ in the matching and the weight of an arc is given by (\ref{eq:arcweight}).
\end{thm}
\begin{proof}
We know that the algorithm from Section \ref{sect:procedure} associates to every configuration of the vertex model a non-crossing matching which determines the change in weight of the configuration under the bijection $\Phi$. If there is a unique non-crossing matching then each configuration is associated to the same matching and the change in weight is the same for all configurations. Thus, the bijection is weight-preserving up to an overall power of $t$.
\end{proof}

In our running example, we have the tuple of partitions $\bg= ((8,7,6),(4,3,2)/(2,0,0))$ for which there is a unique non-crossing matching for the sequence of beads associated to $\bg$. It is given by 
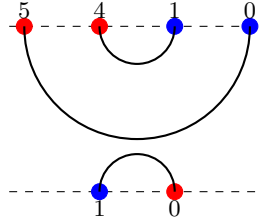
\[
\begin{tikzpicture}[baseline=(current bounding box.center)] 
\draw[dashed] (-0.2,2.2)--(3.2,2.2); \draw[dashed] (-0.2,0)--(3.2,0);
\draw[red, fill=red] (0,2.2) circle (3pt);\draw[red, fill=red] (1,2.2) circle (3pt); \draw[blue, fill=blue] (2,2.2) circle (3pt); \draw[blue, fill=blue] (3,2.2) circle (3pt);
\draw[blue, fill=blue] (1,0) circle (3pt); \draw[red, fill=red] (2,0) circle (3pt);
\node[above] at (0,2.2) {5}; \node[above] at (1,2.2) {4}; \node[above] at (2,2.2) {1}; \node[above] at (3,2.2) {0}; 
\node[below] at (1,0) {1}; \node[below] at (2,0) {0};
\draw[thick] (0,2.2) arc(180:360:1.5); \draw[thick] (1,2.2) arc(180:360:0.5);
\draw[thick] (2,0) arc(0:180:0.5);
\end{tikzpicture}.
\]
Using Theorem \ref{thm:swap} we have
\[
\mathcal{L}_{\bg}(X_n;t) = t^5 \mathcal{L}_{(\bg)_{swap}}(X_n;t)
\]
where $(\bg)_{swap}= ((4,3,2)/(2,0,0), (8,7,6))$.

From this general theorem, we can make a statement about specific families of partitions. For example, as a corollary to Theorem \ref{thm:swap}, we can show that LLT polynomials indexed by nested rectangular partitions are equivalent.
\begin{cor}
Fix $n>0$. Let $\lambda^{(1)} = (\underbrace{m_1,\ldots,m_1}_{k_1},\underbrace{0,\ldots,0}_{n-k_1})$ and $\lambda^{(2)}=(\underbrace{m_2,\ldots,m_2}_{k_2},\underbrace{0,\ldots,0}_{n-k_2})$, with $m_1\ge m_2 \ge 0$ and $k_1\ge k_2\ge 0$.  That is,  rectangular partitions such that $\lambda^{(2)}\subseteq \lambda^{(1)}$. Then $\mathcal{L}_{(\lambda^{(1)},\lambda^{(2)})}(X_n;t)$ and $\mathcal{L}_{(\lambda^{(2)},\lambda^{(1)})}(X_n;t)$ are equivalent.
\end{cor}
\begin{proof}
    There are three cases to consider:
    \begin{enumerate}
        \item $m_1 \ge k_1-k_2$, $m_2 \le m_1-k_1$
        \item $m_1 \ge k_1-k_2$, $m_2 > m_1-k_1$
        \item $m_1 < k_1-k_2$, $m_2 > m_1-k_1$
    \end{enumerate}
     In each case, one can check that there is a unique non crossing matching of the form
    \[
    \resizebox{0.3\textwidth}{!}{
\begin{tikzpicture}[baseline=(current bounding box.center)] 
\draw[dashed] (-0.2,0)--(5.2,0); 
\draw[red, fill=red] (0,0) circle (3pt); \draw[red, fill=red] (1,0) circle (3pt); \draw[red, fill=red] (2,0) circle (3pt); \draw[blue, fill=blue] (3,0) circle (3pt); \draw[blue, fill=blue] (4,0) circle (3pt); \draw[blue, fill=blue] (5,0) circle (3pt);
\draw[thick] (0,0) arc(180:360:2.5); \draw[thick] (1,0) arc(180:360:1.5); \draw[thick] (2,0) arc(180:360:0.5);
\end{tikzpicture}
}
    \]
    where the number of beads and the labeling depends on the case. 
\end{proof}

\begin{remark}
Suppose 
\[
\bg = (\ldots, \beta^{(i)}/\gamma^{(i)},\beta^{(i+1)}/\gamma^{(i+1)},\ldots) \qquad \text{and} \qquad (\bg)_{swap} = (\ldots, \beta^{(i+1)}/\gamma^{(i+1)},\beta^{(i)}/\gamma^{(i)},\ldots)
\]
are two tuples of partitions which are the same except for having their $i^{th}$ and $(i+1)^{st}$ partitions swapped. Then, if there is a unique non-crossing matching $M$ of the sequence of beads associated to the tuple of partitions $(\bg)_{i,i+1} = (\beta^{(i)}/\gamma^{(i)},\beta^{(i+1)}/\gamma^{(i+1)})$, it still holds that
\[
\mathcal{L}_{\bg}(X_n;t) = \left(\prod_{a\in M} w(a)\right) \mathcal{L}_{(\bg)_{swap}}(X_n;t)
\]
where the product is over all arcs $a$ in the matching and the weight of an arc is given by (\ref{eq:arcweight}).
\end{remark}

In Appendix A, we classify all sequences of beads which have unique non-crossing matchings. These results are summarized in Proposition \ref{prop:beads1} and Proposition \ref{prop:beads2}.
\begin{prop}\label{prop:beads1}
For a single row, sequences of beads that have a unique non-crossing matching are of the form
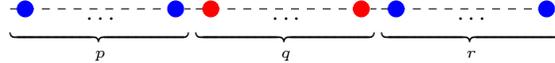
\[
\underbrace{
\begin{tikzpicture}[baseline=(current bounding box.center)] 
\draw[dashed] (-0.2,0)--(2.2,0);
\draw[blue, fill=blue] (0,0) circle (3pt); \node[below] at (1,0) {$\ldots$}; \draw[blue, fill=blue] (2,0) circle (3pt);
\end{tikzpicture}
}_{p}
\underbrace{
\begin{tikzpicture}[baseline=(current bounding box.center)] 
\draw[dashed] (-0.2,0)--(2.2,0);
\draw[red, fill=red] (0,0) circle (3pt); \node[below] at (1,0) {$\ldots$}; \draw[red, fill=red] (2,0) circle (3pt);
\end{tikzpicture}
}_{q}
\underbrace{
\begin{tikzpicture}[baseline=(current bounding box.center)] 
\draw[dashed] (-0.2,0)--(2.2,0);
\draw[blue, fill=blue] (0,0) circle (3pt); \node[below] at (1,0) {$\ldots$}; \draw[blue, fill=blue] (2,0) circle (3pt);
\end{tikzpicture}
}_{r}
\]
with $p+r=q$, or the same configurations as above with red and blue swapped.
\end{prop}
\begin{prop}\label{prop:beads2}
With two rows, sequence of beads that have a unique non-crossing matching are given by
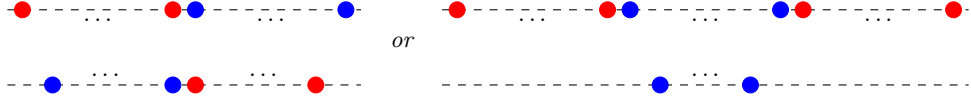
\[
\begin{aligned}
\begin{tikzpicture}[baseline=(current bounding box.center)] 
\draw[dashed] (-0.2,1)--(4.5,1);
\draw[dashed] (-0.2,0)--(4.5,0);
\draw[red, fill=red] (0,1) circle (3pt); \node[below] at (1,1) {$\ldots$}; \draw[red, fill=red] (2,1) circle (3pt);
\draw[blue, fill=blue] (2.3,1) circle (3pt); \node[below] at (3.3,1) {$\ldots$}; \draw[blue, fill=blue] (4.3,1) circle (3pt);
\draw[blue, fill=blue] (0.4,0) circle (3pt); \node[above] at (1.1,0) {$\ldots$}; \draw[blue, fill=blue] (2,0) circle (3pt);
\draw[red, fill=red] (2.3,0) circle (3pt); \node[above] at (3.2,0) {$\ldots$}; \draw[red, fill=red] (3.9,0) circle (3pt);
\end{tikzpicture}
\;\;\;
\text{ or }
\;\;\;
\begin{tikzpicture}[baseline=(current bounding box.center)] 
\draw[dashed] (-0.2,1)--(6.8,1);
\draw[dashed] (-0.2,0)--(6.8,0);
\draw[red, fill=red] (0,1) circle (3pt); \node[below] at (1,1) {$\ldots$}; \draw[red, fill=red] (2,1) circle (3pt);
\draw[blue, fill=blue] (2.3,1) circle (3pt); \node[below] at (3.3,1) {$\ldots$}; \draw[blue, fill=blue] (4.3,1) circle (3pt);
\draw[red, fill=red] (4.6,1) circle (3pt); \node[below] at (5.6,1) {$\ldots$}; \draw[red, fill=red] (6.6,1) circle (3pt);
\draw[blue, fill=blue] (2.7,0) circle (3pt); \node[above] at (3.3,0) {$\ldots$}; \draw[blue, fill=blue] (3.9,0) circle (3pt);
\end{tikzpicture}
\end{aligned}
\]
or the same sequences as above with the rows or colors swapped, where the difference between the number of red and blue beads in the top row is equal to the difference in the number of red and blue beads in the bottom row.
\end{prop}
\noindent We then pull back these conditions to conditions on the tuple $\bg$ for which Theorem \ref{thm:swap} holds.

While the procedure in Section \ref{sect:procedure} associates to every vertex model configuration a non-crossing matching of a sequence of labeled beads, it is not true in general that there is a vertex model configuration associated to every non-crossing matching. For example, let $\bg = ((5,4,4)/(2,2,0), (3,1,1))$. There are two non-crossing matchings for the sequence of beads associated to the tuple
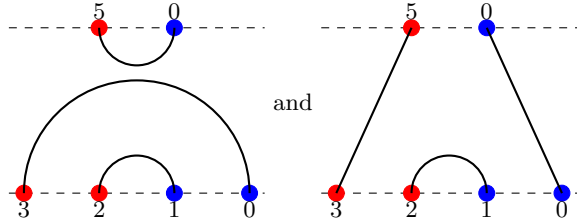
\[
\begin{tikzpicture}[baseline=(current bounding box.center)] 
\draw[dashed] (-0.2,2.2)--(3.2,2.2); \draw[dashed] (-0.2,0)--(3.2,0);
\draw[red, fill=red] (0,0) circle (3pt);\draw[red, fill=red] (1,0) circle (3pt); \draw[blue, fill=blue] (2,0) circle (3pt); \draw[blue, fill=blue] (3,0) circle (3pt);
\draw[red, fill=red] (1,2.2) circle (3pt); \draw[blue, fill=blue] (2,2.2) circle (3pt);
\node[below] at (0,0) {3}; \node[below] at (1,0) {2}; \node[below] at (2,0) {1}; \node[below] at (3,0) {0}; 
\node[above] at (1,2.2) {5}; \node[above] at (2,2.2) {0};
\draw[thick] (3,0) arc(0:180:1.5); \draw[thick] (2,0) arc(0:180:0.5);
\draw[thick] (1,2.2) arc(180:360:0.5);
\end{tikzpicture}
\text{ and }
\begin{tikzpicture}[baseline=(current bounding box.center)] 
\draw[dashed] (-0.2,2.2)--(3.2,2.2); \draw[dashed] (-0.2,0)--(3.2,0);
\draw[red, fill=red] (0,0) circle (3pt);\draw[red, fill=red] (1,0) circle (3pt); \draw[blue, fill=blue] (2,0) circle (3pt); \draw[blue, fill=blue] (3,0) circle (3pt);
\draw[red, fill=red] (1,2.2) circle (3pt); \draw[blue, fill=blue] (2,2.2) circle (3pt);
\node[below] at (0,0) {3}; \node[below] at (1,0) {2}; \node[below] at (2,0) {1}; \node[below] at (3,0) {0}; 
\node[above] at (1,2.2) {5}; \node[above] at (2,2.2) {0};
\draw[thick] (2,0) arc(0:180:0.5);
\draw[thick] (0,0)--(1,2.2);
\draw[thick] (3,0)--(2,2.2);
\end{tikzpicture}.
\]
An example configuration which realizes the second matching is given by
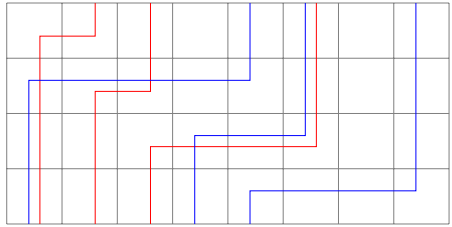
\[
\resizebox{0.4\textwidth}{!}{ 
\begin{tikzpicture}[baseline=(current bounding box.center)] 
\draw[help lines] (0,0) grid (8,4);

\draw[red] (0.6,0)--(0.6,3.4)--(1.6,3.4)--(1.6,4);
\draw[red] (1.6,0)--(1.6,2.4)--(2.6,2.4)--(2.6,4);
\draw[red] (2.6,0)--(2.6,1.4)--(5.6,1.4)--(5.6,4);

\draw[blue] (0.4,0)--(0.4,2.6)--(4.4,2.6)--(4.4,4);
\draw[blue] (3.4,0)--(3.4,1.6)--(5.4,1.6)--(5.4,4);
\draw[blue] (4.4,0)--(4.4,0.6)--(7.4,0.6)--(7.4,4);

\end{tikzpicture}
}\]
but it is impossible to realize the first matching with paths. That is, the map from configurations of the vertex model to non-crossing matchings is not generally onto.

However, if $\bg = \bm{\beta}$ is a tuple of straight shapes, each with the same number of parts, then we can indeed construct a path configuration for the vertex model corresponding to every non-crossing matching (provided that the vertex model has sufficiently many rows). We leave the construction of such configurations to interested readers. In what follows, we will consider either this case or specific examples in which it is clear that all matchings can be obtained.

\section{Linear Relations between LLT polynomials}\label{sect:relations}
In this section we show how we can use the techniques we have developed to construct linear relations between different LLT polynomials.

Let us start with a small example to illustrate the general strategy. Consider the tuples of partitions $\bm{\lambda}_1 = ((3,3),(1))$, $\bm{\lambda}_2 = ((2,2),(3))$, and $\bm{\lambda}_3 = ((3,2),(2))$. The vertex model boundary condition and the non-crossing matchings associated to the tuples are given by
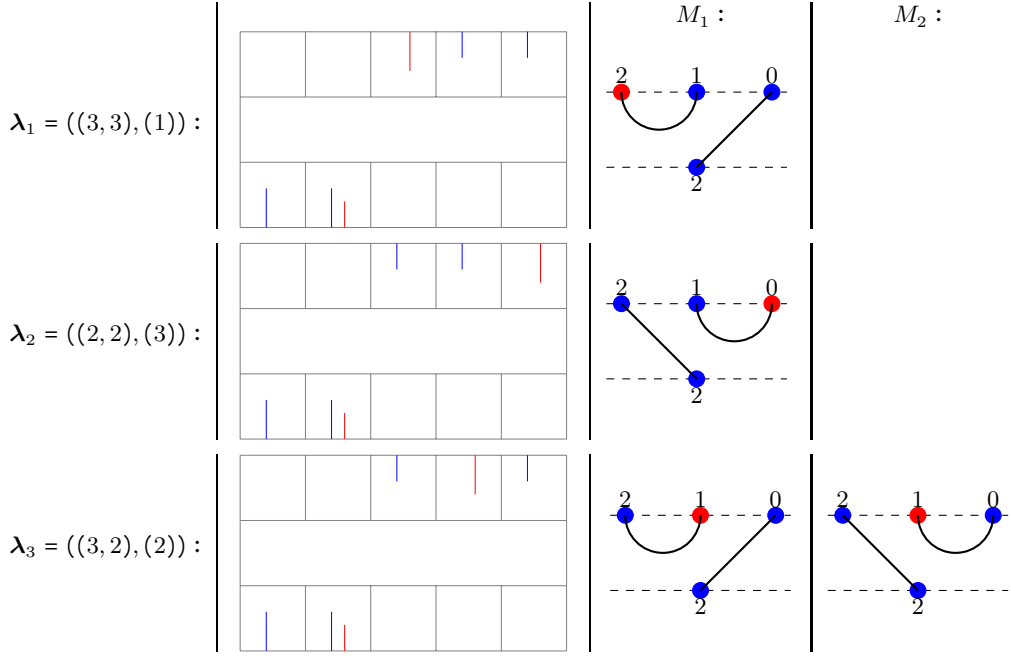
\begin{center}
\begin{tabular}{c|c|c|c}
& & $M_1:$ & $M_2:$
\\
$\bm{\lambda}_1 = ((3,3),(1)):$ &
\resizebox{0.3\textwidth}{!}{
\begin{tikzpicture}[baseline=(current bounding box.center)] 
\draw[help lines] (0,0) grid (5,1);
\draw[help lines] (0,2) grid (5,3);
\draw[help lines] (0,1)--(0,2); \draw[help lines] (5,1)--(5,2);
\draw[blue] (4.4,2.6)--(4.4,3); \draw[blue] (3.4,2.6)--(3.4,3); \draw[red] (2.6,2.4)--(2.6,3);
\draw[blue] (0.4,0)--(0.4,0.6); \draw[blue] (1.4,0)--(1.4,0.6); \draw[red] (1.6,0)--(1.6,0.4);
\end{tikzpicture}
}
&
\begin{tikzpicture}[baseline=(current bounding box.center)] 
\draw[dashed] (-0.2,1)--(2.2,1);  \draw[dashed] (-0.2,0)--(2.2,0);
\draw[red, fill=red] (0,1) circle (3pt); \draw[blue, fill=blue] (1,1) circle (3pt); \draw[blue, fill=blue] (2,1) circle (3pt);
\draw[blue, fill=blue] (1,0) circle (3pt);
\draw[thick] (0,1) arc(180:360:0.5);
\draw[thick] (1,0)--(2,1);
\node[above] at (0,1) {$2$}; \node[above] at (1,1) {$1$}; \node[above] at (2,1) {$0$};
\node[below] at (1,0) {$2$};
\end{tikzpicture}
\vspace{0.2cm}
\\
$\bm{\lambda}_2 = ((2,2),(3)):$ &
\resizebox{0.3\textwidth}{!}{
\begin{tikzpicture}[baseline=(current bounding box.center)] 
\draw[help lines] (0,0) grid (5,1);
\draw[help lines] (0,2) grid (5,3);
\draw[help lines] (0,1)--(0,2); \draw[help lines] (5,1)--(5,2);
\draw[red] (4.6,2.4)--(4.6,3); \draw[blue] (3.4,2.6)--(3.4,3); \draw[blue] (2.4,2.6)--(2.4,3);
\draw[blue] (0.4,0)--(0.4,0.6); \draw[blue] (1.4,0)--(1.4,0.6); \draw[red] (1.6,0)--(1.6,0.4);
\end{tikzpicture}
}
&
\begin{tikzpicture}[baseline=(current bounding box.center)] 
\draw[dashed] (-0.2,1)--(2.2,1);  \draw[dashed] (-0.2,0)--(2.2,0);
\draw[blue, fill=blue] (0,1) circle (3pt); \draw[blue, fill=blue] (1,1) circle (3pt); \draw[red, fill=red] (2,1) circle (3pt);
\draw[blue, fill=blue] (1,0) circle (3pt);
\draw[thick] (1,1) arc(180:360:0.5);
\draw[thick] (1,0)--(0,1);
\node[above] at (0,1) {$2$}; \node[above] at (1,1) {$1$}; \node[above] at (2,1) {$0$};
\node[below] at (1,0) {$2$};
\end{tikzpicture}
\vspace{0.2cm}
\\
$\bm{\lambda}_3 = ((3,2),(2)):$ &
\resizebox{0.3\textwidth}{!}{
\begin{tikzpicture}[baseline=(current bounding box.center)] 
\draw[help lines] (0,0) grid (5,1);
\draw[help lines] (0,2) grid (5,3);
\draw[help lines] (0,1)--(0,2); \draw[help lines] (5,1)--(5,2);
\draw[blue] (4.4,2.6)--(4.4,3); \draw[red] (3.6,2.4)--(3.6,3); \draw[blue] (2.4,2.6)--(2.4,3);
\draw[blue] (0.4,0)--(0.4,0.6); \draw[blue] (1.4,0)--(1.4,0.6); \draw[red] (1.6,0)--(1.6,0.4);
\end{tikzpicture}
}
&
\begin{tikzpicture}[baseline=(current bounding box.center)] 
\draw[dashed] (-0.2,1)--(2.2,1);  \draw[dashed] (-0.2,0)--(2.2,0);
\draw[blue, fill=blue] (0,1) circle (3pt); \draw[red, fill=red] (1,1) circle (3pt); \draw[blue, fill=blue] (2,1) circle (3pt);
\draw[blue, fill=blue] (1,0) circle (3pt);
\draw[thick] (0,1) arc(180:360:0.5);
\draw[thick] (1,0)--(2,1);
\node[above] at (0,1) {$2$}; \node[above] at (1,1) {$1$}; \node[above] at (2,1) {$0$};
\node[below] at (1,0) {$2$};
\end{tikzpicture}
&
\begin{tikzpicture}[baseline=(current bounding box.center)] 
\draw[dashed] (-0.2,1)--(2.2,1);  \draw[dashed] (-0.2,0)--(2.2,0);
\draw[blue, fill=blue] (0,1) circle (3pt); \draw[red, fill=red] (1,1) circle (3pt); \draw[blue, fill=blue] (2,1) circle (3pt);
\draw[blue, fill=blue] (1,0) circle (3pt);
\draw[thick] (1,1) arc(180:360:0.5);
\draw[thick] (1,0)--(0,1);
\node[above] at (0,1) {$2$}; \node[above] at (1,1) {$1$}; \node[above] at (2,1) {$0$};
\node[below] at (1,0) {$2$};
\end{tikzpicture}
\end{tabular}
\end{center} 

As there is a unique non-crossing matching associated to $\bm{\lambda}_1$, Theorem \ref{thm:swap} allows us to conclude that $\mathcal{L}_{\bm{\lambda}_1}(X_n;t) = t^2 \; \mathcal{L}_{((1),(3,3))}(X_n;t)$. However, we can do more. Given a path configuration with boundary condition given by $\bm{\lambda}_1$, we can switch the color of all path segments in the walk starting from the red path on the top boundary. For example,
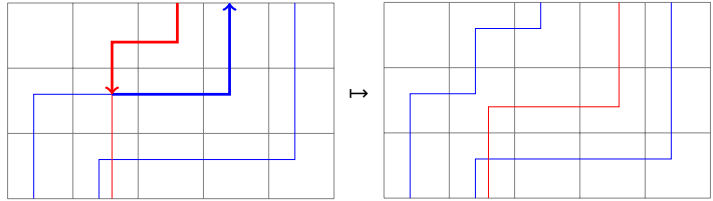
\[
\resizebox{0.3\textwidth}{!}{
\begin{tikzpicture}[baseline=(current bounding box.center)] 
\draw[help lines] (0,0) grid (5,3);
\draw[blue] (0.4,0)--(0.4,1.6)--(3.4,1.6)--(3.4,3); \draw[blue] (1.4,0)--(1.4,0.6)--(4.4,0.6)--(4.4,3); \draw[red] (1.6,0)--(1.6,2.4)--(2.6,2.4)--(2.6,3);
\draw[blue, very thick,->] (1.6,1.6)--(3.4,1.6)--(3.4,3); \draw[red, very thick, ->] (2.6,3)--(2.6,2.4)--(1.6,2.4)--(1.6,1.6);
\end{tikzpicture}
}
\mapsto
\resizebox{0.3\textwidth}{!}{
\begin{tikzpicture}[baseline=(current bounding box.center)] 
\draw[help lines] (0,0) grid (5,3);
\draw[blue] (0.4,0)--(0.4,1.6)--(1.4,1.6)--(1.4,2.6)--(2.4,2.6)--(2.4,3); \draw[blue] (1.4,0)--(1.4,0.6)--(4.4,0.6)--(4.4,3); \draw[red] (1.6,0)--(1.6,1.4)--(3.6,1.4)--(3.6,3);
\end{tikzpicture}.
}
\]
This results in a configuration with boundary condition given by $\bm{\lambda}_3$, in particular, a configuration which corresponds to matching $M_1$ of $\bm{\lambda}_3$. Doing this for every configuration gives a bijection
\[
\{\text{config. with boundary condition } \bm{\lambda}_1 \} \rightarrow \{\text{config. with boundary condition } \bm{\lambda}_3  \text{ corresponding to } M_1\}
\]
where under this mapping the change in the power of $t$ is given by 
\[
w\left( 
\begin{tikzpicture}[baseline=(current bounding box.center)] 
\draw[thick] (0,0) arc(180:360:0.5);
\draw[red, fill=red] (0,0) circle (3pt); \draw[blue, fill=blue] (1,0) circle (3pt);
\node[above] at (0,0) {$2$}; \node[above] at (1,0) {$1$}; 
\end{tikzpicture} 
\right) = t.
\]

Similarly, for any configuration in the vertex model for $\bm{\lambda}_2$, we can swap the color of all path segments along the walk starting at the red path on the top boundary. This gives a bijection
\[
\{\text{config. with boundary condition } \bm{\lambda}_2 \} \rightarrow \{\text{config. with boundary condition } \bm{\lambda}_3  \text{ corresponding to } M_2\}
\]
where under this mapping the change in the power of $t$ is given by 
\[
w\left( 
\begin{tikzpicture}[baseline=(current bounding box.center)] 
\draw[thick] (0,0) arc(180:360:0.5);
\draw[blue, fill=blue] (0,0) circle (3pt); \draw[red, fill=red] (1,0) circle (3pt);
\node[above] at (0,0) {$1$}; \node[above] at (1,0) {$0$}; 
\end{tikzpicture} 
\right) = t^{-1}.
\]
All together this shows that
\[
\mathcal{L}_{\bm{\lambda}_3}(X;t) = t^{-1} \; \mathcal{L}_{\bm{\lambda}_1}(X;t) + t \; \mathcal{L}_{\bm{\lambda}_2}(X;t).
\]

As one possible application of this type of calculation, we can reprove a relation between LLT polynomials indexed by single rows given in \cite{Tom}. Note that the precise powers of $t$ differ than that of \cite{Tom} as we are working with coinversion, rather than inversion, LLT polynomials. (Recall that the inversion and coinversion LLT polynomials themselves only differ by an overall power of $t$.)
\begin{lem}[Lemma 3.17, \cite{Tom}]
Let $\beta_1,\gamma_1,\beta_2,\gamma_2$ be positive integers such that $\gamma_1 < \gamma_2 \le \beta_1  < \beta_2$. Then 
\[
\mathcal{L}_{(\beta_1/\gamma_1, \beta_2/\gamma_2)}(X_n;t) = \mathcal{L}_{(\beta_2/\gamma_2, \beta_1/\gamma_1)}(X_n;t) + (t^{-1}-1) \mathcal{L}_{(\beta_2/\gamma_1, \beta_1/\gamma_2)}(X_n;t).
\]
\end{lem}
\begin{proof}
We begin by drawing the boundary condition and matchings corresponding to the tuple of partitions of each of these LLT polynomials:
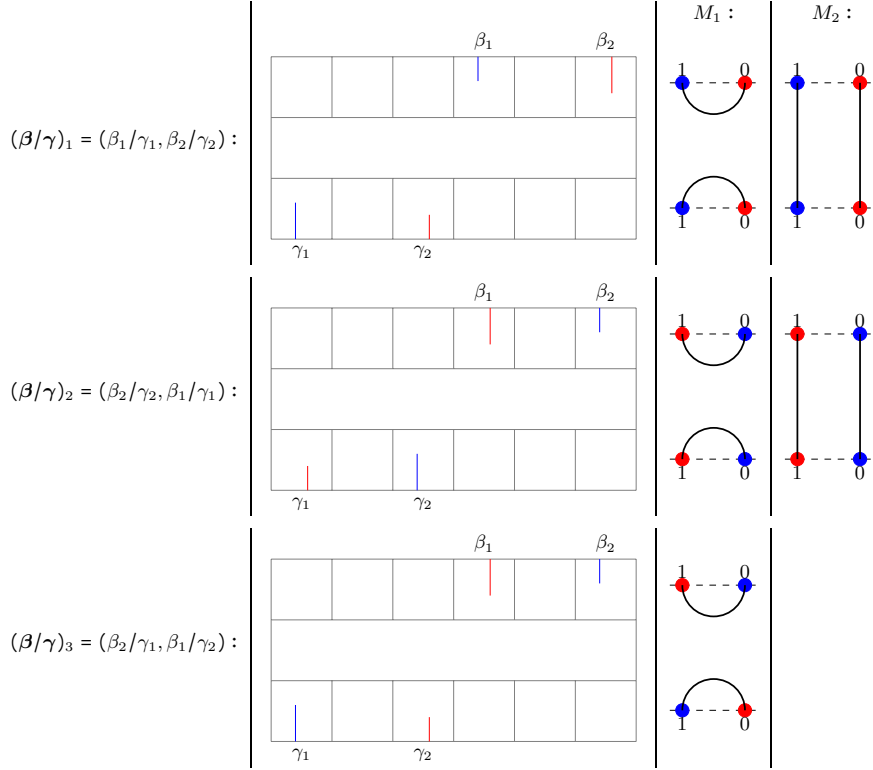
\begin{center}
\resizebox{0.8\textwidth}{!}{
\begin{tabular}{c|c|c|c}
& & $M_1:$ & $M_2:$
\\
$(\bg)_1=(\beta_1/\gamma_1, \beta_2/\gamma_2):$ &
\resizebox{0.4\textwidth}{!}{
\begin{tikzpicture}[baseline=(current bounding box.center)] 
\draw[help lines] (0,0) grid (6,1);
\draw[help lines] (0,2) grid (6,3);
\draw[help lines] (0,1)--(0,2); \draw[help lines] (6,1)--(6,2);
\draw[blue] (3.4,2.6)--(3.4,3); \draw[red] (5.6,2.4)--(5.6,3);
\draw[blue] (0.4,0)--(0.4,0.6); \draw[red] (2.6,0)--(2.6,0.4);
\node[below] at (0.5,0) {$\gamma_1$}; \node[below] at (2.5,0) {$\gamma_2$};
\node[above] at (3.5,3) {$\beta_1$}; \node[above] at (5.5,3) {$\beta_2$};
\end{tikzpicture}
}
&
\begin{tikzpicture}[baseline=(current bounding box.center)] 
\draw[dashed] (-0.2,2)--(1.2,2);  \draw[dashed] (-0.2,0)--(1.2,0);
\draw[blue, fill=blue] (0,2) circle (3pt); \draw[red, fill=red] (1,2) circle (3pt); 
\draw[blue, fill=blue] (0,0) circle (3pt); \draw[red, fill=red] (1,0) circle (3pt);
\draw[thick] (0,2) arc(180:360:0.5);
\draw[thick] (1,0) arc(0:180:0.5);
\node[above] at (0,2) {$1$}; \node[above] at (1,2) {$0$}; 
\node[below] at (0,0) {$1$}; \node[below] at (1,0) {$0$};
\end{tikzpicture}
&
\begin{tikzpicture}[baseline=(current bounding box.center)] 
\draw[dashed] (-0.2,2)--(1.2,2);  \draw[dashed] (-0.2,0)--(1.2,0);
\draw[blue, fill=blue] (0,2) circle (3pt); \draw[red, fill=red] (1,2) circle (3pt); 
\draw[blue, fill=blue] (0,0) circle (3pt); \draw[red, fill=red] (1,0) circle (3pt);
\draw[thick] (0,0)--(0,2);
\draw[thick] (1,0)--(1,2);
\node[above] at (0,2) {$1$}; \node[above] at (1,2) {$0$}; 
\node[below] at (0,0) {$1$}; \node[below] at (1,0) {$0$};
\end{tikzpicture}
\vspace{0.2cm}
\\
$(\bg)_2=(\beta_2/\gamma_2, \beta_1/\gamma_1):$ &
\resizebox{0.4\textwidth}{!}{
\begin{tikzpicture}[baseline=(current bounding box.center)] 
\draw[help lines] (0,0) grid (6,1);
\draw[help lines] (0,2) grid (6,3);
\draw[help lines] (0,1)--(0,2); \draw[help lines] (6,1)--(6,2);
\draw[blue] (5.4,2.6)--(5.4,3); \draw[red] (3.6,2.4)--(3.6,3);
\draw[blue] (2.4,0)--(2.4,0.6); \draw[red] (0.6,0)--(0.6,0.4);
\node[below] at (0.5,0) {$\gamma_1$}; \node[below] at (2.5,0) {$\gamma_2$};
\node[above] at (3.5,3) {$\beta_1$}; \node[above] at (5.5,3) {$\beta_2$};
\end{tikzpicture}
}
&
\begin{tikzpicture}[baseline=(current bounding box.center)] 
\draw[dashed] (-0.2,2)--(1.2,2);  \draw[dashed] (-0.2,0)--(1.2,0);
\draw[red, fill=red] (0,2) circle (3pt); \draw[blue, fill=blue] (1,2) circle (3pt); 
\draw[red, fill=red] (0,0) circle (3pt); \draw[blue, fill=blue] (1,0) circle (3pt);
\draw[thick] (0,2) arc(180:360:0.5);
\draw[thick] (1,0) arc(0:180:0.5);
\node[above] at (0,2) {$1$}; \node[above] at (1,2) {$0$}; 
\node[below] at (0,0) {$1$}; \node[below] at (1,0) {$0$};
\end{tikzpicture}
&
\begin{tikzpicture}[baseline=(current bounding box.center)] 
\draw[dashed] (-0.2,2)--(1.2,2);  \draw[dashed] (-0.2,0)--(1.2,0);
\draw[red, fill=red] (0,2) circle (3pt); \draw[blue, fill=blue] (1,2) circle (3pt); 
\draw[red, fill=red] (0,0) circle (3pt); \draw[blue, fill=blue] (1,0) circle (3pt);
\draw[thick] (0,0)--(0,2);
\draw[thick] (1,0)--(1,2);
\node[above] at (0,2) {$1$}; \node[above] at (1,2) {$0$}; 
\node[below] at (0,0) {$1$}; \node[below] at (1,0) {$0$};
\end{tikzpicture}
\vspace{0.2cm}
\\
$(\bg)_3=(\beta_2/\gamma_1, \beta_1/\gamma_2):$ &
\resizebox{0.4\textwidth}{!}{
\begin{tikzpicture}[baseline=(current bounding box.center)] 
\draw[help lines] (0,0) grid (6,1);
\draw[help lines] (0,2) grid (6,3);
\draw[help lines] (0,1)--(0,2); \draw[help lines] (6,1)--(6,2);
\draw[blue] (5.4,2.6)--(5.4,3); \draw[red] (3.6,2.4)--(3.6,3);
\draw[blue] (0.4,0)--(0.4,0.6); \draw[red] (2.6,0)--(2.6,0.4);
\node[below] at (0.5,0) {$\gamma_1$}; \node[below] at (2.5,0) {$\gamma_2$};
\node[above] at (3.5,3) {$\beta_1$}; \node[above] at (5.5,3) {$\beta_2$};
\end{tikzpicture}
}
&
\begin{tikzpicture}[baseline=(current bounding box.center)] 
\draw[dashed] (-0.2,2)--(1.2,2);  \draw[dashed] (-0.2,0)--(1.2,0);
\draw[red, fill=red] (0,2) circle (3pt); \draw[blue, fill=blue] (1,2) circle (3pt); 
\draw[blue, fill=blue] (0,0) circle (3pt); \draw[red, fill=red] (1,0) circle (3pt);
\draw[thick] (0,2) arc(180:360:0.5);
\draw[thick] (1,0) arc(0:180:0.5);
\node[above] at (0,2) {$1$}; \node[above] at (1,2) {$0$}; 
\node[below] at (0,0) {$1$}; \node[below] at (1,0) {$0$};
\end{tikzpicture}
\end{tabular}
}
\end{center} 
Note that 
\begin{itemize}
\item The terms corresponding to $M_2$ in $(\bg)_1$ and $(\bg)_2$ are equal as we can swap the colors along the arcs connecting the two rows without changing the weight.
\item The terms corresponding to $M_1$ in $(\bg)_1$ are equal to those in $(\bg)_3$ up to a factor of $t^{-1}$ coming from swapping the top arc.
\item The terms corresponding to $M_1$ in $(\bg)_2$ are equal to those in $(\bg)_3$ since there is no cost to swapping the bottom arc. 
\end{itemize}
We can use $(\bg)_3$ to cancel the terms corresponding to $M_1$ in both $(\bg)_1$ and $(\bg)_2$, then use the fact that the terms corresponding to $M_2$ in both $(\bg)_1$ and $(\bg)_2$ are equal. Putting this together, we find
\[
\mathcal{L}_{(\beta_1/\gamma_1, \beta_2/\gamma_2)}(X_n;t) - t^{-1}\;\mathcal{L}_{(\beta_2/\gamma_1, \beta_1/\gamma_2)}(X_n;t) = \mathcal{L}_{(\beta_2/\gamma_2, \beta_1/\gamma_1)}(X_n;t) - \mathcal{L}_{(\beta_2/\gamma_1, \beta_1/\gamma_2)}(X_n;t)
\]
from which the result follows.
\end{proof}

For a more complicated example, consider $2n$ real numbers $\beta_1>\beta_2>\ldots>\beta_{2n}\ge0$. Consider the family of tuples of partitions of the form 
\begin{equation}\label{eq:partfam}
\bm{\beta} = ((\beta_{i_1}-n+1,\ldots,\beta_{i_n}-n+n),(\beta_{j_1}-n+1,\ldots,\beta_{j_n}-n+n))
\end{equation}
where $\beta_{i_1}>\ldots>\beta_{i_n}$, $\beta_{j_1}>\ldots>\beta_{j_n}$, and the $\beta_i$'s and $\beta_j$'s partition the set $\{\beta_1,\ldots,\beta_{2n}\}$. There are $\binom{2n}{n}$ tuples of partitions in this family. We will show that these LLT polynomials can all be written as sums of a smaller collection of $C_n = \frac{1}{n+1}\binom{2n}{n}$ symmetric polynomials associated to the possible non-crossing matchings.
\begin{thm}\label{thm:bigfam}
Consider the family of partitions given in (\ref{eq:partfam}). Then for every $\bm{\beta}$ in this family, the LLT polynomial $\mathcal{L}_{\bm{\beta}}(X_n;t)$ can be written 
\[
\mathcal{L}_{\bm{\beta}}(X_n;t) = \sum_{j=1}^{C_n} t^{n_{j}(\bm{\beta})} g_j(X_n;t)
\]
where $C_n = \frac{1}{n+1}\binom{2n}{n}$ is the $n^{th}$ Catalan number, $n_{j}(\bm{\beta})\in \mathbb{Z}$ for each $j$ and $\bm{\beta}$, and the $g_i$ are polynomials symmetric in $X_n$.
\end{thm}
\begin{proof}
The given family of partitions corresponds to sequence of beads with only one row with $n$ red and $n$ blue beads. For example, the tuple of partitions $\bm{\beta}=((\beta_1-n+1,\ldots,\beta_n-n+n),(\beta_{n+1}-n+1,\ldots,\beta_{2n}-n+n))$ corresponds to the LLT polynomial with the sequence of beads
\[
\underbrace{
\begin{tikzpicture}[baseline=(current bounding box.center)] 
\draw[dashed] (-0.2,0)--(2.2,0);
\draw[red, fill=red] (0,0) circle (3pt); \node[below] at (1,0) {$\ldots$}; \draw[red, fill=red] (2,0) circle (3pt);
\node[above] at (0,0) {$2n-1$}; \node[above] at (2,0) {$n$};
\end{tikzpicture}
}_{n}
\underbrace{
\begin{tikzpicture}[baseline=(current bounding box.center)] 
\draw[dashed] (-0.2,0)--(2.2,0);
\draw[blue, fill=blue] (0,0) circle (3pt); \node[below] at (1,0) {$\ldots$}; \draw[blue, fill=blue] (2,0) circle (3pt);
\node[above] at (0,0) {$n-1$}; \node[above] at (2,0) {$0$};
\end{tikzpicture}
}_{n}.
\]
Ignoring the colors of beads, there are $C_n = \frac{1}{n+1}\binom{2n}{n}$ possible non-crossing matchings of the $2n$ beads. Let us put some ordering on the matchings and label them $M_1,\ldots,M_{C_n}$.

 Suppose we pick one of the non-crossing matchings $M_i$. Let us choose one of the many ways to color the beads so that each arc in the matching connects a red bead to a blue bead. This selects one of the partitions from our family. Call this partition $\bm{\beta}_i$. Clearly, we have that $\mathcal{L}_{\bm{\beta}_i}(X_n;t)$ has $M_i$ as one of its associated matchings, although it may have others. We define $g_i(X_n;t)$ as the sum of the weights of the configuration in the vertex model for $\mathcal{L}_{\bm{\beta}_i}(X_n;t)$ which gives the matching $M_i$ under the algorithm in Section \ref{sect:procedure}. 
 
 For any $\bm{\beta}$ each configuration belongs to a unique matching, so we must have
\[
\mathcal{L}_{\bm{\beta}}(X_n;t) = \sum_{j=1}^{C_n} t^{n_{j}(\bm{\beta})} g_{j}(X_n;t)
\]
where $n_{j}(\bm{\beta})$ is the power of $t$ needed to account for the difference in the coloring of the beads between $\bm{\beta}$ and $\bm{\beta}_j$. While it is clear we can write the LLT polynomials as the sum of the $g_i$'s, it is not obvious that the $g_i$'s are symmetric.

To show they are indeed symmetric, let $\bm{M}(t)$ be the matrix which takes  $(g_1,\ldots,g_{C_n})^T$ to $(\mathcal{L}_{\bm{\beta}_1},\ldots,\mathcal{L}_{\bm{\beta}_{C_n}})^T$, that is,
\[
\begin{pmatrix}
    \mathcal{L}_{\bm{\beta}_1} \\ \vdots \\\mathcal{L}_{\bm{\beta}_{C_n}}
\end{pmatrix}
=
\bm{M}(t) 
\begin{pmatrix}
    g_1 \\ \vdots \\g_{C_n}
\end{pmatrix}
\]
where $\bm{M}(t)_{ij} = t^{n_j(\bm{\beta}_i)}$.

We will show that by appropriately choosing the ordering of the matchings and appropriately selecting the colorings, the matrix $\bm{M}(t)$ will be lower triangular with ones on the diagonal. In particular, $\bm{M}(t)$ is invertible and $\bm{M}^{-1}(t)$ describes a way to write each $g_{i}(X_n;t)$ as a linear combination of the $\mathcal{L}_{\bm{\beta}_j}(X_n;t)$'s. It then follows that the $g_i$'s are symmetric in $X_n$.

We construct an order of these matchings inductively on the number of beads. When there are zero beads there is a unique non-crossing matching (the empty matching). Suppose we know how to correctly order the matchings of $2k$ beads, for all $k<n$.

Now suppose there are $2n$ beads. Recall that the beads in our matching are labeled, with the rightmost bead labeled by 0. There is an arc connecting this bead to some bead to its left, say with bead labeled $2(n-a)-1$ for some $a$. This partitions our matching into three pieces as shown below:
\[
\begin{tikzpicture}[baseline=(current bounding box.north)] 
\draw[dashed] (-0.2,0)--(5.2,0); 
\draw[black, fill=black] (2,0) circle (3pt); \draw[black, fill=black] (5,0) circle (3pt);
\draw[thick] (2,0) arc(180:360:1.5);
\node[below] at (1,0) {$M_b^{(a)}$}; \node[below] at (3.75,0) {$M_c^{(n-a+1)}$};
\node[above] at (2,0) {$2(n-a)-1$}; \node[above] at (5,0) {$0$};
\end{tikzpicture}
\]
We can now uniquely describe any matching by the triple $(a,b,c)$, $0\le a< n$, $0<b\le C_a$, $0<c\le C_{n-a-1}$, by decomposing it into the $b^{th}$ matching of $2a$ beads $M_b^{(a)}$ and the $c^{th}$ matching of $2(n-a-1)$ beads $M_c^{(n-a-1)}$.

We then order the matchings on $2n$ beads using dictionary ordering, that is, by the rule $(a_1,b_1,c_1) < (a_2,b_2,c_2)$ iff $a_1<a_2$, or $a_1 = a_2$ and $b_1<b_2$, or $a_1=a_2$ and $b_1=b_2$ and $c_1<c_2$.

For each of the matchings, we choose to color the beads such that the blue bead is always the right endpoint of an arc. If $M_i$ is the $i^{th}$ matching of $2n$ beads, our coloring picks out the unique partition in our family associated to that sequence of colored beads. Call this partition $\bm{\beta}_i$. We will now show that under these choices we have
\[
\mathcal{L}_{\bm{\beta}_i}(X_n;t) = g_i(X_n;t) + \sum_{j<i} t^{n_j(\bm{\beta}_i)} g_j(X_n;t),
\]
that is, $\bm{M}(t)$ is lower triangular with ones on the diagonal. 

The fact that the coefficient in front of the $g_i(X_n;t)$ for $\mathcal{L}_{\bm{\beta}_i}(X_n;t)$ is 1 follows from the definition of the $g_i$'s. To see that the matrix is lower triangular, suppose $M_i$ can be decomposed as $(a,b,c)$. We need to show that the sequence of colored beads associated to $\bm{\beta}_i$ does not have a matching that can be decomposed as $(a',b',c')$ with $(a',b',c')>(a,b,c)$.

Equivalently, we need to show that for any matching $(a,b,c)$ with our choice of coloring, any other matching $(a',b',c')$ that respects this coloring (its arcs must connect beads of different colors) cannot have $(a',b',c')>(a,b,c)$. If $n=1$, then there is only a single matching so this is true.
Assume that this holds for any matching with less than $2n$ beads.

Now, let $(a,b,c)$ be a matching of $2n$ beads with our choice of coloring. Let $(a',b',c')$ be any other matching that respects this choice of coloring. In the case $a'=a$, we must show that either we cannot have $b'>b$, or if $b'=b$, we cannot have $c'>c$. Each of these subcases involve matchings with fewer beads, so by induction the result holds. 

We are left with showing we cannot have $a'>a$. If $a=n-1$ there is nothing to check, so we assume $a<n-1$.

Recall, the beads are labeled $0$ to $2n-1$ going from right to left. Define a function $\rho_n:\{0,\ldots,2n-1\}\rightarrow\Z$ where $\rho_n(k)$ returns the number of blue beads minus the number of red beads up to and including the $k^{th}$ bead. Our choice of coloring guarantees $\rho_n(0)=1$. Note that it is impossible for the $0^{th}$ bead to match with the $k^{th}$ bead unless $\rho_n(k)=0$. So we are done if we can show that for our choice of coloring $\rho_n(k)\ne 0$ for $k<2(n-a)$. 

One can prove, by induction on $n$, that for our choice of colorings $\rho_n$ is always non-negative. This implies that for $k<2(n-a)$ we have $\rho_n(k)-\rho_n(0)\ge 0$ since this is exactly the function $\rho_{n-a-1}$ defined for the matching $M_c^{(n-a-1)}$. It follows that $\rho_n(k) \ge \rho_n(0)=1$, as desired.

We have shown that $\bm{M}(t)$ is lower triangular with ones on the diagonal. Therefore, it is invertible. This gives us a way to write the $g_i$'s as a sum of LLT polynomials, from which we see the $g_i$'s are symmetric.

\end{proof}

As a specific example of Theorem \ref{thm:bigfam}, consider the case when $n=3$ and $\beta_{6-k}=k+1$ for $k=0,\ldots,5$. There are $\binom{6}{3}=20$ tuples of partitions in this family:
\[
\begin{aligned}
\bm{\lambda}_1 = ((4,4,4),(1,1,1)), & \; \bm{\lambda}_2 = ((4,4,3),(2,1,1)) \\
\bm{\lambda}_3 = ((4,4,2),(2,2,1)), & \; \bm{\lambda}_4 = ((4,3,3),(3,1,1)) \\
\bm{\lambda}_5 = ((4,3,2),(3,2,1)), & \; \bm{\lambda}_6 = ((4,4,1),(2,2,2))\\
\bm{\lambda}_7 = ((4,1,1),(3,3,3)), & \; \bm{\lambda}_8 = ((4,2,2),(3,3,1)) \\
\bm{\lambda}_9 = ((4,2,1),(3,3,2)), & \; \bm{\lambda}_{10} = ((4,3,1),(3,2,2))
\end{aligned}
\]
with the tuples $\bm{\lambda}_{10+j}$ obtain by swapping the order of the partitions in $\bm{\lambda}_j$. These are all the tuples that can be obtained by starting with the boundary condition for $\bm{\lambda}_1$ 
\[
\resizebox{0.3\textwidth}{!}{
\begin{tikzpicture}[baseline=(current bounding box.center)] 
\draw[help lines] (0,0) grid (7,1);
\draw[help lines] (0,2) grid (7,3);
\draw[help lines] (0,1)--(0,2); \draw[help lines] (7,1)--(7,2);
\draw[blue] (4.6,2.4)--(4.6,3); \draw[blue] (5.6,2.4)--(5.6,3); \draw[blue] (6.6,2.4)--(6.6,3); \draw[red] (1.4,2.6)--(1.4,3); \draw[red] (3.4,2.6)--(3.4,3); \draw[red] (2.4,2.6)--(2.4,3);
\draw[blue] (0.4,0)--(0.4,0.6); \draw[blue] (1.4,0)--(1.4,0.6); \draw[blue] (2.4,0)--(2.4,0.6); \draw[red] (0.6,0)--(0.6,0.4); \draw[red] (1.6,0)--(1.6,0.4); \draw[red] (2.6,0)--(2.6,0.4);
\end{tikzpicture}
}
\]
but choosing a different coloring of the paths on the top boundary while fixing that 3 paths are red and three paths are blue. These tuples of partitions are all associated to matchings on six beads. If we ignore the colors of the beads, there are five possible non-crossing matchings:
\begin{center}
\resizebox{0.8\textwidth}{!}{
\begin{tabular}{c|c|c|c|c}
$M_1$ & $M_2$ & $M_3$ & $M_4$ & $M_5$ \\
\resizebox{0.2\textwidth}{!}{
\begin{tikzpicture}[baseline=(current bounding box.north)] 
\draw[dashed] (-0.2,0)--(5.2,0); 
\draw[black, fill=black] (0,0) circle (3pt); \draw[black, fill=black] (1,0) circle (3pt); \draw[black, fill=black] (2,0) circle (3pt); \draw[black, fill=black] (3,0) circle (3pt); \draw[black, fill=black] (4,0) circle (3pt); \draw[black, fill=black] (5,0) circle (3pt);
\draw[thick] (0,0) arc(180:360:2.5); \draw[thick] (1,0) arc(180:360:1.5); \draw[thick] (2,0) arc(180:360:0.5);
\node[above] at (0,0) {$5$}; \node[above] at (1,0) {$4$}; \node[above] at (2,0) {$3$}; \node[above] at (3,0) {$2$}; \node[above] at (4,0) {$1$}; \node[above] at (5,0) {$0$};
\end{tikzpicture}
}
&
\resizebox{0.2\textwidth}{!}{
\begin{tikzpicture}[baseline=(current bounding box.north)] 
\draw[dashed] (-0.2,0)--(5.2,0); 
\draw[black, fill=black] (0,0) circle (3pt); \draw[black, fill=black] (1,0) circle (3pt); \draw[black, fill=black] (2,0) circle (3pt); \draw[black, fill=black] (3,0) circle (3pt); \draw[black, fill=black] (4,0) circle (3pt); \draw[black, fill=black] (5,0) circle (3pt);
\draw[thick] (0,0) arc(180:360:2.5); \draw[thick] (1,0) arc(180:360:0.5); \draw[thick] (3,0) arc(180:360:0.5);
\node[above] at (0,0) {$5$}; \node[above] at (1,0) {$4$}; \node[above] at (2,0) {$3$}; \node[above] at (3,0) {$2$}; \node[above] at (4,0) {$1$}; \node[above] at (5,0) {$0$};
\end{tikzpicture}
}
&
\resizebox{0.2\textwidth}{!}{
\begin{tikzpicture}[baseline=(current bounding box.north)] 
\draw[dashed] (-0.2,0)--(5.2,0); 
\draw[black, fill=black] (0,0) circle (3pt); \draw[black, fill=black] (1,0) circle (3pt); \draw[black, fill=black] (2,0) circle (3pt); \draw[black, fill=black] (3,0) circle (3pt); \draw[black, fill=black] (4,0) circle (3pt); \draw[black, fill=black] (5,0) circle (3pt);
\draw[thick] (0,0) arc(180:360:0.5); \draw[thick] (2,0) arc(180:360:1.5); \draw[thick] (3,0) arc(180:360:0.5);
\node[above] at (0,0) {$5$}; \node[above] at (1,0) {$4$}; \node[above] at (2,0) {$3$}; \node[above] at (3,0) {$2$}; \node[above] at (4,0) {$1$}; \node[above] at (5,0) {$0$};
\end{tikzpicture}
}
&
\resizebox{0.2\textwidth}{!}{
\begin{tikzpicture}[baseline=(current bounding box.north)] 
\draw[dashed] (-0.2,0)--(5.2,0); 
\draw[black, fill=black] (0,0) circle (3pt); \draw[black, fill=black] (1,0) circle (3pt); \draw[black, fill=black] (2,0) circle (3pt); \draw[black, fill=black] (3,0) circle (3pt); \draw[black, fill=black] (4,0) circle (3pt); \draw[black, fill=black] (5,0) circle (3pt);
\draw[thick] (0,0) arc(180:360:1.5); \draw[thick] (1,0) arc(180:360:0.5); \draw[thick] (4,0) arc(180:360:0.5);
\node[above] at (0,0) {$5$}; \node[above] at (1,0) {$4$}; \node[above] at (2,0) {$3$}; \node[above] at (3,0) {$2$}; \node[above] at (4,0) {$1$}; \node[above] at (5,0) {$0$};
\end{tikzpicture}
}
&
\resizebox{0.2\textwidth}{!}{
\begin{tikzpicture}[baseline=(current bounding box.north)] 
\draw[dashed] (-0.2,0)--(5.2,0); 
\draw[black, fill=black] (0,0) circle (3pt); \draw[black, fill=black] (1,0) circle (3pt); \draw[black, fill=black] (2,0) circle (3pt); \draw[black, fill=black] (3,0) circle (3pt); \draw[black, fill=black] (4,0) circle (3pt); \draw[black, fill=black] (5,0) circle (3pt);
\draw[thick] (0,0) arc(180:360:0.5); \draw[thick] (2,0) arc(180:360:0.5); \draw[thick] (4,0) arc(180:360:0.5);
\node[above] at (0,0) {$5$}; \node[above] at (1,0) {$4$}; \node[above] at (2,0) {$3$}; \node[above] at (3,0) {$2$}; \node[above] at (4,0) {$1$}; \node[above] at (5,0) {$0$};
\end{tikzpicture}
}
\end{tabular}
}
\end{center}
which we order as in the proof of the theorem. We also choose a coloring of the beads for each matching as in the theorem:
\begin{center}
\resizebox{0.8\textwidth}{!}{
\begin{tabular}{c|c|c|c|c}
\resizebox{0.2\textwidth}{!}{
\begin{tikzpicture}[baseline=(current bounding box.north)] 
\draw[dashed] (-0.2,0)--(5.2,0); 
\draw[red, fill=red] (0,0) circle (3pt); \draw[red, fill=red] (1,0) circle (3pt); \draw[red, fill=red] (2,0) circle (3pt); \draw[blue, fill=blue] (3,0) circle (3pt); \draw[blue, fill=blue] (4,0) circle (3pt); \draw[blue, fill=blue] (5,0) circle (3pt);
\draw[thick] (0,0) arc(180:360:2.5); \draw[thick] (1,0) arc(180:360:1.5); \draw[thick] (2,0) arc(180:360:0.5);
\node[above] at (0,0) {$5$}; \node[above] at (1,0) {$4$}; \node[above] at (2,0) {$3$}; \node[above] at (3,0) {$2$}; \node[above] at (4,0) {$1$}; \node[above] at (5,0) {$0$};
\end{tikzpicture}
}
&
\resizebox{0.2\textwidth}{!}{
\begin{tikzpicture}[baseline=(current bounding box.north)] 
\draw[dashed] (-0.2,0)--(5.2,0); 
\draw[red, fill=red] (0,0) circle (3pt); \draw[red, fill=red] (1,0) circle (3pt); \draw[blue, fill=blue] (2,0) circle (3pt); \draw[red, fill=red] (3,0) circle (3pt); \draw[blue, fill=blue] (4,0) circle (3pt); \draw[blue, fill=blue] (5,0) circle (3pt);
\draw[thick] (0,0) arc(180:360:2.5); \draw[thick] (1,0) arc(180:360:0.5); \draw[thick] (3,0) arc(180:360:0.5);
\node[above] at (0,0) {$5$}; \node[above] at (1,0) {$4$}; \node[above] at (2,0) {$3$}; \node[above] at (3,0) {$2$}; \node[above] at (4,0) {$1$}; \node[above] at (5,0) {$0$};
\end{tikzpicture}
}
&
\resizebox{0.2\textwidth}{!}{
\begin{tikzpicture}[baseline=(current bounding box.north)] 
\draw[dashed] (-0.2,0)--(5.2,0); 
\draw[red, fill=red] (0,0) circle (3pt); \draw[blue, fill=blue] (1,0) circle (3pt); \draw[red, fill=red] (2,0) circle (3pt); \draw[red, fill=red] (3,0) circle (3pt); \draw[blue, fill=blue] (4,0) circle (3pt); \draw[blue, fill=blue] (5,0) circle (3pt);
\draw[thick] (0,0) arc(180:360:0.5); \draw[thick] (2,0) arc(180:360:1.5); \draw[thick] (3,0) arc(180:360:0.5);
\node[above] at (0,0) {$5$}; \node[above] at (1,0) {$4$}; \node[above] at (2,0) {$3$}; \node[above] at (3,0) {$2$}; \node[above] at (4,0) {$1$}; \node[above] at (5,0) {$0$};
\end{tikzpicture}
}
&
\resizebox{0.2\textwidth}{!}{
\begin{tikzpicture}[baseline=(current bounding box.north)] 
\draw[dashed] (-0.2,0)--(5.2,0); 
\draw[red, fill=red] (0,0) circle (3pt); \draw[red, fill=red] (1,0) circle (3pt); \draw[blue, fill=blue] (2,0) circle (3pt); \draw[blue, fill=blue] (3,0) circle (3pt); \draw[red, fill=red] (4,0) circle (3pt); \draw[blue, fill=blue] (5,0) circle (3pt);
\draw[thick] (0,0) arc(180:360:1.5); \draw[thick] (1,0) arc(180:360:0.5); \draw[thick] (4,0) arc(180:360:0.5);
\node[above] at (0,0) {$5$}; \node[above] at (1,0) {$4$}; \node[above] at (2,0) {$3$}; \node[above] at (3,0) {$2$}; \node[above] at (4,0) {$1$}; \node[above] at (5,0) {$0$};
\end{tikzpicture}
}
&
\resizebox{0.2\textwidth}{!}{
\begin{tikzpicture}[baseline=(current bounding box.north)] 
\draw[dashed] (-0.2,0)--(5.2,0); 
\draw[red, fill=red] (0,0) circle (3pt); \draw[blue, fill=blue] (1,0) circle (3pt); \draw[red, fill=red] (2,0) circle (3pt); \draw[blue, fill=blue] (3,0) circle (3pt); \draw[red, fill=red] (4,0) circle (3pt); \draw[blue, fill=blue] (5,0) circle (3pt);
\draw[thick] (0,0) arc(180:360:0.5); \draw[thick] (2,0) arc(180:360:0.5); \draw[thick] (4,0) arc(180:360:0.5);
\node[above] at (0,0) {$5$}; \node[above] at (1,0) {$4$}; \node[above] at (2,0) {$3$}; \node[above] at (3,0) {$2$}; \node[above] at (4,0) {$1$}; \node[above] at (5,0) {$0$};
\end{tikzpicture}
}
\end{tabular}.
}
\end{center}

The choice of coloring of the beads corresponds uniquely to one of the tuples of partitions in our family. In particular, for our choice of coloring the matching $M_i$ corresponds to $\mathcal{L}_{\bm{\lambda}_i}$ for each $i\in\{1,2,3,4,5\}$. Note that, by construction, $M_i$ is a non-crossing matching associated to $\mathcal{L}_{\bm{\lambda}_i}$, but $\mathcal{L}_{\bm{\lambda}_i}$ might have other non-crossing matchings associated to it as well.  Define the polynomial $g_i(X_n;t)$ as the terms in $\mathcal{L}_{\bm{\lambda}_i}(X_n;t)$ which correspond to the matching $M_i$ under the algorithm in Section \ref{sect:procedure}.

If $M_j$ is a non-crossing matching associated to $\mathcal{L}_{\bm{\lambda}_i}$, let $n_{j}(\bm{\lambda}_i)$ be the difference in the power of $t$ between the weight of the coloring of $M_j$ given by $\mathcal{L}_{\bm{\lambda}_i}$ and our choice of coloring of $M_j$. For example, there are are two non-crossing matchings associated to $\mathcal{L}_{\bm{\lambda}_2}$:
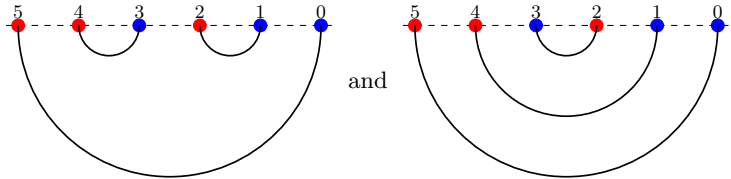
\[
\resizebox{0.3\textwidth}{!}{
\begin{tikzpicture}[baseline=(current bounding box.center)] 
\draw[dashed] (-0.2,0)--(5.2,0); 
\draw[red, fill=red] (0,0) circle (3pt); \draw[red, fill=red] (1,0) circle (3pt); \draw[blue, fill=blue] (2,0) circle (3pt); \draw[red, fill=red] (3,0) circle (3pt); \draw[blue, fill=blue] (4,0) circle (3pt); \draw[blue, fill=blue] (5,0) circle (3pt);
\draw[thick] (0,0) arc(180:360:2.5); \draw[thick] (1,0) arc(180:360:0.5); \draw[thick] (3,0) arc(180:360:0.5);
\node[above] at (0,0) {$5$}; \node[above] at (1,0) {$4$}; \node[above] at (2,0) {$3$}; \node[above] at (3,0) {$2$}; \node[above] at (4,0) {$1$}; \node[above] at (5,0) {$0$};
\end{tikzpicture}
}
\text{ and }
\resizebox{0.3\textwidth}{!}{
\begin{tikzpicture}[baseline=(current bounding box.center)] 
\draw[dashed] (-0.2,0)--(5.2,0); 
\draw[red, fill=red] (0,0) circle (3pt); \draw[red, fill=red] (1,0) circle (3pt); \draw[blue, fill=blue] (2,0) circle (3pt); \draw[red, fill=red] (3,0) circle (3pt); \draw[blue, fill=blue] (4,0) circle (3pt); \draw[blue, fill=blue] (5,0) circle (3pt);
\draw[thick] (0,0) arc(180:360:2.5); \draw[thick] (1,0) arc(180:360:1.5); \draw[thick] (2,0) arc(180:360:0.5);
\node[above] at (0,0) {$5$}; \node[above] at (1,0) {$4$}; \node[above] at (2,0) {$3$}; \node[above] at (3,0) {$2$}; \node[above] at (4,0) {$1$}; \node[above] at (5,0) {$0$};
\end{tikzpicture}
}.
\]
The first is precisely our choice of coloring of $M_2$, so $n_{2}(\bm{\lambda}_2)=0$. However, to make the coloring given here match our choice of coloring of $M_1$ we would need to swap the 3-2 arc which gives a weight of $t^{-1}$. So $n_{1}(\bm{\lambda}_2)=-1$. This allows us to write
\[
\mathcal{L}_{\bm{\lambda}_i} (X_n;t)= \sum_{j} t^{n_{j}(\bm{\lambda}_i)} g_{j}(X_n;t) 
\]
for each $\bm{\lambda}_i$.

Consider the matrix $\bm{M}(t)$ that in the $i^{th}$ row and $j^{th}$ column has $t^{n_{j}(\bm{\lambda}_i)}$ if $M_j$ is a matching associated to $\mathcal{L}_{\bm{\lambda}_i}$, and zero otherwise. Then $\bm{M}(t)$ is the matrix that transforms $(g_1,\ldots,g_5)^T$ into $(\mathcal{L}_{\bm{\lambda}_1},\ldots,\mathcal{L}_{\bm{\lambda}_5})^T$. For the above, this matrix is
\[
\bm{M}(t)=\begin{pmatrix}
1 & 0 & 0 & 0 & 0\\ t^{-1} & 1 & 0 & 0 & 0\\ 0 & t^{-1} & 1 & 0 & 0\\ 0 & t^{-1} & 0 & 1 & 0 \\ t^{-2} & t^{-2} & t^{-1} & t^{-1} & 1
\end{pmatrix}, \;\;\; \bm{M}^{-1}(t)=\begin{pmatrix}
1 & 0 & 0 & 0 & 0\\ -t^{-1} & 1 & 0 & 0 & 0 \\ t^{-2} & -t^{-1} & 1 & 0 & 0\\ t^{-2} & -t^{-1} & 0 & 1 & 0 \\ -t^{-3}-t^{-2} & t^{-2} & -t^{-1} & -t^{-1} & 1
\end{pmatrix} 
\]
where we also write its inverse. The inverse tells us how to write the $g_i$'s as a linear combination of the LLT polynomials. For example,  
\[
g_{3}(X_n;t) = \mathcal{L}_{\bm{\lambda}_3}(X_n;t) - t^{-1} \mathcal{L}_{\bm{\lambda}_2}(X_n;t) + t^2 \mathcal{L}_{\bm{\lambda}_1}(X_n;t).
\]

 Since every tuple of partitions in this family must be made up of configurations corresponding to these five matching, we have
\[
\mathcal{L}_{\bm{\lambda}_i}(X_n;t) = \sum_{j=1}^5 t^{n_{j}(\bm{\lambda}_i)} g_j(X_n;t)
\]
for some integers $n_{j}(\bm{\lambda}_i)$, for all $i=1,2,\ldots,20$ not just $i=1,\ldots,5$. For instance, there are two non-crossing matchings associated to $\bm{\lambda}_{8}$:
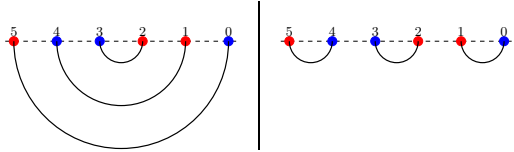
\begin{center}
\resizebox{0.5\textwidth}{!}{
\begin{tabular}{c|c}
\resizebox{0.2\textwidth}{!}{
\begin{tikzpicture}[baseline=(current bounding box.north)] 
\draw[dashed] (-0.2,0)--(5.2,0); 
\draw[red, fill=red] (0,0) circle (3pt); \draw[blue, fill=blue] (1,0) circle (3pt); \draw[blue, fill=blue] (2,0) circle (3pt); \draw[red, fill=red] (3,0) circle (3pt); \draw[red, fill=red] (4,0) circle (3pt); \draw[blue, fill=blue] (5,0) circle (3pt);
\draw[thick] (0,0) arc(180:360:2.5); \draw[thick] (1,0) arc(180:360:1.5); \draw[thick] (2,0) arc(180:360:0.5);
\node[above] at (0,0) {$5$}; \node[above] at (1,0) {$4$}; \node[above] at (2,0) {$3$}; \node[above] at (3,0) {$2$}; \node[above] at (4,0) {$1$}; \node[above] at (5,0) {$0$};
\end{tikzpicture}
}
&
\resizebox{0.2\textwidth}{!}{
\begin{tikzpicture}[baseline=(current bounding box.north)] 
\draw[dashed] (-0.2,0)--(5.2,0); 
\draw[red, fill=red] (0,0) circle (3pt); \draw[blue, fill=blue] (1,0) circle (3pt); \draw[blue, fill=blue] (2,0) circle (3pt); \draw[red, fill=red] (3,0) circle (3pt); \draw[red, fill=red] (4,0) circle (3pt); \draw[blue, fill=blue] (5,0) circle (3pt);
\draw[thick] (0,0) arc(180:360:0.5); \draw[thick] (2,0) arc(180:360:0.5); \draw[thick] (4,0) arc(180:360:0.5);
\node[above] at (0,0) {$5$}; \node[above] at (1,0) {$4$}; \node[above] at (2,0) {$3$}; \node[above] at (3,0) {$2$}; \node[above] at (4,0) {$1$}; \node[above] at (5,0) {$0$};
\end{tikzpicture}
}
\end{tabular}
}
\end{center}
For the first matching, we can get from coloring given here to our choice of coloring for $M_1$ by changing the color along the the arc connecting the beads labeled $4$ and $1$ at the cost of a $t^{-2}$, and along the arc connecting the beads labeled $3$ and $2$ at the cost of a $t^{-1}$. For the second matching, we can get from this coloring to our coloring of $M_5$ by swapping the color along the arc connecting beads labeled 3 and 2 at the cost of a $t^{-1}$. All together we have
\[
\begin{aligned}
\mathcal{L}_{\bm{\lambda}_{8}}(X_n;t) = & t^{-3} \; g_1(X_n;t) + t^{-1}\; g_5(X_n;t).
\end{aligned}
\]

While we chose a specific family of LLT polynomials in this example, the same result holds for any family of tuples of partitions which are associated to the same sequence of beads (allowing for changes in the labeling of the beads). The only difference will be the specific powers of $t$ that will appear which one can compute from the matchings.

\section{Conclusion}

In this paper we prove a sufficient condition for when $\mathcal{L}_{(\beta^{(1)}/\gamma^{(1)},\beta^{(2)}/\gamma^{(2)})}(X_n;t)$ is equivalent to $\mathcal{L}_{(\beta^{(2)}/\gamma^{(2)},\beta^{(1)}/\gamma^{(1)})}(X_n;t)$. We do so by giving a bijection between configurations of the vertex model for each of the LLT polynomials. We show that the change in weights under this map is determined by a matching of a sequence of colored beads that can be associated to the boundary condition of the vertex model. When the sequence of beads has a unique non-crossing matching, the bijection is weight-preserving up to an overall power of $t$ and it follows that the LLT polynomials are equivalent. 

Using these techniques we are also able to construct linear relations between different LLT polynomials. Knowing these types of relations has been instrumental in proving results about the expansion of LLT polynomials into Schur and $k$-Schur polynomials \cite{A,Lee,ChrisThesis,Tom}. Our new techniques give a systematic way to determine these relations. While we focused on a few specific families of partitions, we wish to emphasize that these techniques could be applied to many different families without much additional effort.

\section{A: Classification}\label{sect:classify}
In this section we classify the sequences of beads for which there is a unique non-crossing matching. Since we can associate a sequence of beads to every $\bg$, this gives a sufficient condition on tuples of partitions for which the bijection $\Phi$ is weight-preserving up to an overall power of $t$.

\begin{lem}\label{lem:beadexist}
A sequence of beads has at least one non-crossing matching if and only if the difference between the number of red beads and blue beads on the top row equals the difference in the number of red beads and blue beads on the bottom row. In particular, if there is only one row then it must have an equal number of red and blue beads.
\end{lem}
\begin{proof}
Given a matching, the difference in the number of red beads and the number of blue beads in either row is given by the number of arcs going from the top row to the bottom connecting red beads minus those connecting blue beads. 
It follows that if we have a matching then the difference between the number of red beads and blue beads on the top row must equal the difference in the number of red beads and blue beads on the bottom row.

Now assume that the constraint on the beads holds. As a base case, when there are two beads we can either have a red and blue bead in the same row, or beads of the same color in different rows. In either case we have a matching. 

Suppose there are $n$ beads. For concreteness, let the leftmost bead on the top row be red. Order the beads from 1 to $n$, going from from left to right on the top row, then from right to left on the bottom row. Let $f:\{1,\ldots,n\}\rightarrow \Z$ be the function such that $f(i)$ returns the difference between the number of red beads and blue beads on the top row minus the difference in the number of red beads and blue beads on the bottom row in the first $i$ beads in our sequence. Note that $f(n)=0$ due to the constraint on the beads, and $f(1)=1$ as we assume the first bead is red.

Let $k$ the smallest integer under this ordering such that $f(k)=0$. Then $f(k-1)$ must be positive, as $|f(i+1)-f(i)|=1$ and $f(0)=1>0$. This implies that if the $k^{th}$ bead is in the top row, it must be blue since for beads in the top row $f(i) < f(i-1)$ iff the $i^{th}$ bead is blue. Similarly, if the $k^{th}$ bead is in the bottom row, then it must be red. 

Adding an arc that connects the first bead to the $k^{th}$ bead splits the sequence of beads into two pieces: those beads between the $1^{st}$ and the $k^{th}$ bead (inclusively), and those after the $k^{th}$. For each piece the difference between the number of red beads and blue beads on the top row equals the difference in the number of red beads and blue beads on the bottom row. By induction each of the pieces has a non-crossing matching. Thus the whole sequence of beads has a non-crossing matching.
\end{proof}

\begin{prop}\label{prop:beads1row}
For a single row, sequences of beads that have a unique non-crossing matching are of the form
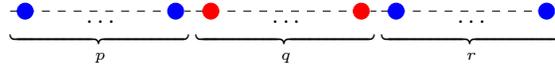
\[
\underbrace{
\begin{tikzpicture}[baseline=(current bounding box.center)] 
\draw[dashed] (-0.2,0)--(2.2,0);
\draw[blue, fill=blue] (0,0) circle (3pt); \node[below] at (1,0) {$\ldots$}; \draw[blue, fill=blue] (2,0) circle (3pt);
\end{tikzpicture}
}_{p}
\underbrace{
\begin{tikzpicture}[baseline=(current bounding box.center)] 
\draw[dashed] (-0.2,0)--(2.2,0);
\draw[red, fill=red] (0,0) circle (3pt); \node[below] at (1,0) {$\ldots$}; \draw[red, fill=red] (2,0) circle (3pt);
\end{tikzpicture}
}_{q}
\underbrace{
\begin{tikzpicture}[baseline=(current bounding box.center)] 
\draw[dashed] (-0.2,0)--(2.2,0);
\draw[blue, fill=blue] (0,0) circle (3pt); \node[below] at (1,0) {$\ldots$}; \draw[blue, fill=blue] (2,0) circle (3pt);
\end{tikzpicture}
}_{r}
\]
with $p+r=q$, or the same configurations as above with red and blue swapped.
\end{prop}
\begin{proof}
One can see that such a sequence admits only one non-crossing matching with the $p$ blue beads on the left matching with the $p$ leftmost red beads and the $r$ blue beads on the right matching with the $r$ rightmost read beads. It is left to show that if we have any other type of configuration there are multiple non-crossing matchings. In particular, if we have a sequence of beads not of the above form it must contain a subsequence of the form
\[
\begin{tikzpicture}[baseline=(current bounding box.center)] 
\draw[dashed] (-0.2,0)--(6.2,0);
\draw[blue, fill=blue] (0,0) circle (3pt); \node[below] at (1,0) {$\ldots$}; \draw[red, fill=red] (2,0) circle (3pt);  \node[below] at (3,0) {$\ldots$}; \draw[blue, fill=blue] (4,0) circle (3pt); \node[below] at (5,0) {$\ldots$}; \draw[red, fill=red] (6,0) circle (3pt);
\end{tikzpicture}
\]
or
\[
\begin{tikzpicture}[baseline=(current bounding box.center)] 
\draw[dashed] (-0.2,0)--(6.2,0);
\draw[red, fill=red] (0,0) circle (3pt); \node[below] at (1,0) {$\ldots$}; \draw[blue, fill=blue] (2,0) circle (3pt);  \node[below] at (3,0) {$\ldots$}; \draw[red, fill=red] (4,0) circle (3pt); \node[below] at (5,0) {$\ldots$}; \draw[blue, fill=blue] (6,0) circle (3pt);
\end{tikzpicture}.
\]
We will show that in this case there are at least two non-crossing matchings. We will do this by constructing two matching. As a base case, when there are two blue and two red beads, we have the two matchings
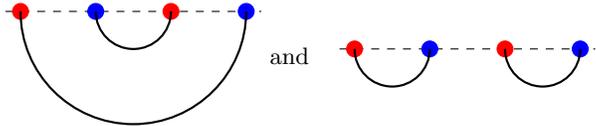
\[
\begin{aligned}
\begin{tikzpicture}[baseline=(current bounding box.center)] 
\draw[dashed] (-0.2,0)--(3.2,0);
\draw[red, fill=red] (0,0) circle (3pt);  \draw[blue, fill=blue] (1,0) circle (3pt); \draw[red, fill=red] (2,0) circle (3pt); \draw[blue, fill=blue] (3,0) circle (3pt);
\draw[thick] (0,0) arc(180:360:1.5); \draw[thick] (1,0) arc(180:360:0.5);
\end{tikzpicture} \text{ and } & \;\;\
\begin{tikzpicture}[baseline=(current bounding box.center)] 
\draw[dashed] (-0.2,0)--(3.2,0);
\draw[red, fill=red] (0,0) circle (3pt);  \draw[blue, fill=blue] (1,0) circle (3pt); \draw[red, fill=red] (2,0) circle (3pt); \draw[blue, fill=blue] (3,0) circle (3pt);
\draw[thick] (0,0) arc(180:360:0.5); \draw[thick] (2,0) arc(180:360:0.5);
\end{tikzpicture}.
\end{aligned}
\]

Now suppose we have $n$ red and $n$ blue beads. Suppose our sequence of beads contains a subsequence of the form red-blue-red-blue (the case of blue-red-blue-red is handled the same way with all the colors swapped). We can always choose the subsequence so the first red and blue beads in the sequence are adjacent.

From Lemma \ref{lem:beadexist} we know there is at least one matching. Let us consider where this first red bead matches. We mark both the red bead and the adjacent blue bead with stars so that they are easily distinguished. There are three cases:

\noindent Case 1: The red matches to the right of the blue. Then we must have a matching of the form
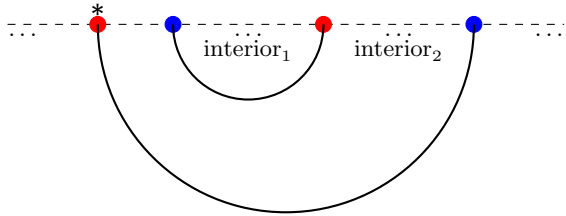
\[
\begin{tikzpicture}[baseline=(current bounding box.center)] 
\draw[dashed] (-1.2,0)--(6.2,0);
 \node[below] at (-1,0) {$\ldots$}; \draw[red, fill=red] (0,0) circle (3pt); \draw[blue, fill=blue] (1,0) circle (3pt);  \node[below] at (2,0) {$\ldots$}; \draw[red, fill=red] (3,0) circle (3pt); \node[below] at (4,0) {$\ldots$}; \draw[blue, fill=blue] (5,0) circle (3pt);  \node[below] at (6,0) {$\ldots$};
 \node[below] at (2,-0.1) {interior$_1$};  \node[below] at (4,-0.1) {interior$_2$};
 \draw[thick] (0,0) arc(180:360:2.5); \draw[thick] (1,0) arc(180:360:1.0);
 \node[above] at (0,0) {$*$};
 \node[above] at (1,0) {$*$};
\end{tikzpicture}.
\]
Note that any arcs in interior$_1$ must stay in interior$_1$, and similarly for arcs in interior$_2$. Here we can simply exhibit a second matching
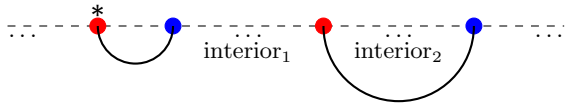
\[
\begin{tikzpicture}[baseline=(current bounding box.center)] 
\draw[dashed] (-1.2,0)--(6.2,0);
\node[below] at (-1,0) {$\ldots$}; \draw[red, fill=red] (0,0) circle (3pt); \draw[blue, fill=blue] (1,0) circle (3pt);  \node[below] at (2,0) {$\ldots$}; \draw[red, fill=red] (3,0) circle (3pt); \node[below] at (4,0) {$\ldots$}; \draw[blue, fill=blue] (5,0) circle (3pt);  \node[below] at (6,0) {$\ldots$};
\node[below] at (2,-0.1) {interior$_1$};  \node[below] at (4,-0.1) {interior$_2$};
\draw[thick] (0,0) arc(180:360:0.5); \draw[thick] (3,0) arc(180:360:1.0);
\node[above] at (0,0) {$*$};
\node[above] at (1,0) {$*$};
\end{tikzpicture}
\]
where the matching of all the other beads remains the same.

\noindent Case 2: The red matches to the left. Here we have two subcases, the blue matching to the left or the blue matching to the right. The matchings take the form
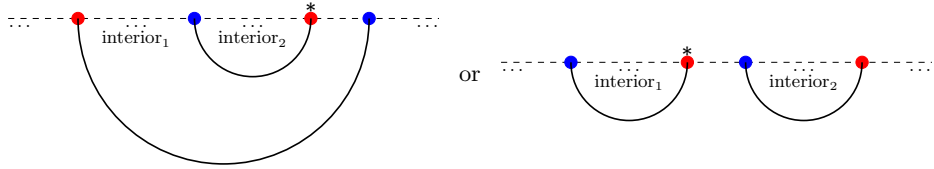
\[
\begin{aligned}
\resizebox{0.4\textwidth}{!}{
\begin{tikzpicture}[baseline=(current bounding box.center)] 
\draw[dashed] (-1.2,0)--(6.2,0);
\node[below] at (-1,0) {$\ldots$}; \draw[red, fill=red] (0,0) circle (3pt); \node[below] at (1,0) {$\ldots$}; \draw[blue, fill=blue] (2,0) circle (3pt);  \node[below] at (3,0) {$\ldots$}; \draw[red, fill=red] (4,0) circle (3pt);  \draw[blue, fill=blue] (5,0) circle (3pt); \node[below] at (6,0) {$\ldots$};
 \node[below] at (1,-0.1) {interior$_1$};  \node[below] at (3,-0.1) {interior$_2$};
 \draw[thick] (0,0) arc(180:360:2.5); \draw[thick] (2,0) arc(180:360:1.0);
 \node[above] at (4,0) {$*$};
 \node[above] at (5,0) {$*$};
\end{tikzpicture}
}
 \text{ or }
 \resizebox{0.4\textwidth}{!}{
\begin{tikzpicture}[baseline=(current bounding box.center)] 
\draw[dashed] (-1.2,0)--(6.2,0);
\node[below] at (-1,0) {$\ldots$}; \draw[blue, fill=blue] (0,0) circle (3pt);  \node[below] at (1,0) {$\ldots$}; \draw[red, fill=red] (2,0) circle (3pt);  \draw[blue, fill=blue] (3,0) circle (3pt); \node[below] at (4,0) {$\ldots$}; \draw[red, fill=red] (5,0) circle (3pt); \node[below] at (6,0) {$\ldots$};
 \node[below] at (1,-0.1) {interior$_1$};  \node[below] at (4,-0.1) {interior$_2$};
 \draw[thick] (0,0) arc(180:360:1.0); \draw[thick] (3,0) arc(180:360:1.0);
 \node[above] at (2,0) {$*$};
 \node[above] at (3,0) {$*$};
\end{tikzpicture}
}
\end{aligned}
\]
where again any arcs in interior$_1$ must stay in interior$_1$, and similarly for arcs in interior$_2$. Again we can simply exhibit a second matching
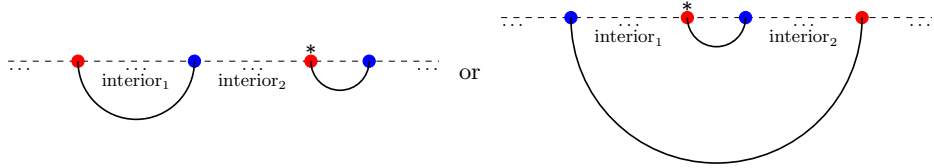
\[
\begin{aligned}
\resizebox{0.4\textwidth}{!}{
\begin{tikzpicture}[baseline=(current bounding box.center)] 
\draw[dashed] (-1.2,0)--(6.2,0);
\node[below] at (-1,0) {$\ldots$}; \draw[red, fill=red] (0,0) circle (3pt); \node[below] at (1,0) {$\ldots$}; \draw[blue, fill=blue] (2,0) circle (3pt);  \node[below] at (3,0) {$\ldots$}; \draw[red, fill=red] (4,0) circle (3pt);  \draw[blue, fill=blue] (5,0) circle (3pt); \node[below] at (6,0) {$\ldots$};
 \node[below] at (1,-0.1) {interior$_1$};  \node[below] at (3,-0.1) {interior$_2$};
 \draw[thick] (0,0) arc(180:360:1.0); \draw[thick] (4,0) arc(180:360:0.5);
 \node[above] at (4,0) {$*$};
 \node[above] at (5,0) {$*$};
\end{tikzpicture}
}
 \text{ or }
\resizebox{0.4\textwidth}{!}{
\begin{tikzpicture}[baseline=(current bounding box.center)] 
\draw[dashed] (-1.2,0)--(6.2,0);
\node[below] at (-1,0) {$\ldots$}; \draw[blue, fill=blue] (0,0) circle (3pt);  \node[below] at (1,0) {$\ldots$}; \draw[red, fill=red] (2,0) circle (3pt);  \draw[blue, fill=blue] (3,0) circle (3pt); \node[below] at (4,0) {$\ldots$}; \draw[red, fill=red] (5,0) circle (3pt); \node[below] at (6,0) {$\ldots$};
\node[below] at (1,-0.1) {interior$_1$};  \node[below] at (4,-0.1) {interior$_2$};
\draw[thick] (0,0) arc(180:360:2.5); \draw[thick] (2,0) arc(180:360:0.5);
\node[above] at (2,0) {$*$};
\node[above] at (3,0) {$*$};
\end{tikzpicture}
}
\end{aligned}
\]
where the matching of all the other beads remains the same.

\noindent Case 3: The red matches with the blue. Since we know this pair of beads is the first red-blue in a red-blue-red-blue subsequence we can instead look at the second red-blue pair. Again we can always choose them to be adjacent. Repeating the analysis from case 1 and 2 above, we are left only with the case when the the red and blue beads in this pair also match. The matching then looks like
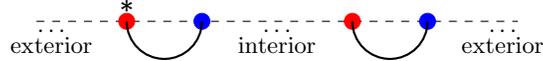
\[
\begin{tikzpicture}[baseline=(current bounding box.center)] 
\draw[dashed] (-1.2,0)--(5.2,0);
 \node[below] at (-1,0) {$\ldots$}; \draw[red, fill=red] (0,0) circle (3pt); \draw[blue, fill=blue] (1,0) circle (3pt);  \node[below] at (2,0) {$\ldots$}; \draw[red, fill=red] (3,0) circle (3pt); \draw[blue, fill=blue] (4,0) circle (3pt);  \node[below] at (5,0) {$\ldots$};
 \draw[thick] (0,0) arc(180:360:0.5); \draw[thick] (3,0) arc(180:360:0.5);
 \node[below] at (2,-0.1) {interior}; \node[below] at (-1,-0.1) {exterior}; \node[below] at (5,-0.1) {exterior};
 \node[above] at (0,0) {$*$};
 \node[above] at (1,0) {$*$};
\end{tikzpicture}
\] 
where we label the other beads as being in the interior or exterior. 

If the interior is empty we can exhibit a second matching as in the base case. Similarly if beads in the interior only match with other beads in the interior, we can exhibit a second matching. So suppose there are beads in the interior which match with beads in the exterior. 

If a red bead in the interior matches with a blue bead in the left exterior, then these beads and all those in between them form a sequence that contains a blue-red-blue-red subsequence which by induction has more than one matching. Similarly, if a blue bead in the interior matches with a red bead in the right exterior, by induction we have more than one matching. So we are left to consider the case in which either some blue beads in the interior match with red beads in the left exterior, some red beads in the interior match with blue beads in the right exterior, or both. 

Let us look only at the rightmost blue bead in the interior that matches with the left exterior (or the original left red-blue pair if no such blue bead exists) and the leftmost red bead in the interior that matches with the right exterior (or the original right red-blue pair if no such red bead exists). The matching takes the form
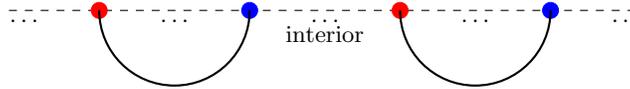
\[
\begin{tikzpicture}[baseline=(current bounding box.center)] 
\draw[dashed] (-1.2,0)--(7.2,0);
 \node[below] at (-1,0) {$\ldots$}; \draw[red, fill=red] (0,0) circle (3pt); \node[below] at (1,0) {$\ldots$}; \draw[blue, fill=blue] (2,0) circle (3pt);  \node[below] at (3,0) {$\ldots$}; \draw[red, fill=red] (4,0) circle (3pt); \node[below] at (5,0) {$\ldots$}; \draw[blue, fill=blue] (6,0) circle (3pt);  \node[below] at (7,0) {$\ldots$};
 \draw[thick] (0,0) arc(180:360:1.0); \draw[thick] (4,0) arc(180:360:1.0);
 \node[below] at (3,-0.1) {interior}; 
 \end{tikzpicture}
\]
where now beads in the interior must match only with other beads in the interior. Now we can exhibit a second matching
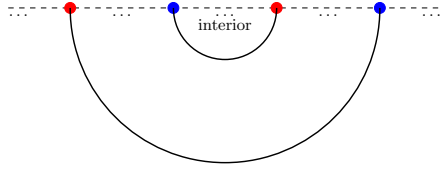
\[
\resizebox{0.4\textwidth}{!}{
\begin{tikzpicture}[baseline=(current bounding box.center)] 
\draw[dashed] (-1.2,0)--(7.2,0);
 \node[below] at (-1,0) {$\ldots$}; \draw[red, fill=red] (0,0) circle (3pt); \node[below] at (1,0) {$\ldots$}; \draw[blue, fill=blue] (2,0) circle (3pt);  \node[below] at (3,0) {$\ldots$}; \draw[red, fill=red] (4,0) circle (3pt); \node[below] at (5,0) {$\ldots$}; \draw[blue, fill=blue] (6,0) circle (3pt);  \node[below] at (7,0) {$\ldots$};
 \draw[thick] (0,0) arc(180:360:3.0); \draw[thick] (2,0) arc(180:360:1.0);
 \node[below] at (3,-0.1) {interior}; 
 \end{tikzpicture}
}.
\]

We see that if the sequence of beads has a red-blue-red-blue subsequence, there are at least two non-crossing matchings. Repeating the analysis but swapping all the colors gives the same result for blue-red-blue-red subsequences. Thus the only sequences of beads that have a unique non-crossing matching are those in the statement of the proposition.
\end{proof}

\begin{prop}\label{prop:beads2rows}
With two rows, sequences of beads that have a unique non-crossing matching are given by
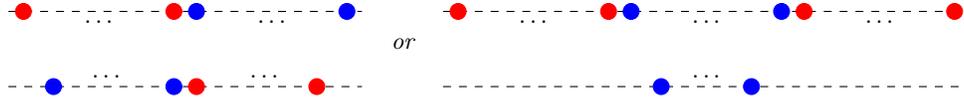
\[
\begin{aligned}
\begin{tikzpicture}[baseline=(current bounding box.center)] 
\draw[dashed] (-0.2,1)--(4.5,1);
\draw[dashed] (-0.2,0)--(4.5,0);
\draw[red, fill=red] (0,1) circle (3pt); \node[below] at (1,1) {$\ldots$}; \draw[red, fill=red] (2,1) circle (3pt);
\draw[blue, fill=blue] (2.3,1) circle (3pt); \node[below] at (3.3,1) {$\ldots$}; \draw[blue, fill=blue] (4.3,1) circle (3pt);
\draw[blue, fill=blue] (0.4,0) circle (3pt); \node[above] at (1.1,0) {$\ldots$}; \draw[blue, fill=blue] (2,0) circle (3pt);
\draw[red, fill=red] (2.3,0) circle (3pt); \node[above] at (3.2,0) {$\ldots$}; \draw[red, fill=red] (3.9,0) circle (3pt);
\node[above] at (1,1.1) {$p$};
\node[above] at (3.3,1.1) {$q$};
\node[below] at (1.2,-0.1) {$r$};
\node[below] at (3.1,-0.1) {$s$};
\draw[decorate,decoration=brace] (-0.1,1.1)--(2.1,1.1);
\draw[decorate,decoration=brace] (2.2,1.1)--(4.4,1.1);
\draw[decorate,decoration=brace] (2.1,-0.1)--(0.3,-0.1);
\draw[decorate,decoration=brace] (4.0,-0.1)--(2.2,-0.1);
\end{tikzpicture}
\;\;\;
\text{ or }
\;\;\;
\begin{tikzpicture}[baseline=(current bounding box.center)] 
\draw[dashed] (-0.2,1)--(6.8,1);
\draw[dashed] (-0.2,0)--(6.8,0);
\draw[red, fill=red] (0,1) circle (3pt); \node[below] at (1,1) {$\ldots$}; \draw[red, fill=red] (2,1) circle (3pt);
\draw[blue, fill=blue] (2.3,1) circle (3pt); \node[below] at (3.3,1) {$\ldots$}; \draw[blue, fill=blue] (4.3,1) circle (3pt);
\draw[red, fill=red] (4.6,1) circle (3pt); \node[below] at (5.6,1) {$\ldots$}; \draw[red, fill=red] (6.6,1) circle (3pt);
\draw[blue, fill=blue] (2.7,0) circle (3pt); \node[above] at (3.3,0) {$\ldots$}; \draw[blue, fill=blue] (3.9,0) circle (3pt);
\node[above] at (1,1.1) {$p$};
\node[above] at (3.3,1.1) {$q$};
\node[above] at (5.6,1.1) {$r$};
\node[below] at (3.3,-0.1) {$s$};
\draw[decorate,decoration=brace] (-0.1,1.1)--(2.1,1.1);
\draw[decorate,decoration=brace] (2.2,1.1)--(4.4,1.1);
\draw[decorate,decoration=brace] (4.5,1.1)--(6.7,1.1);
\draw[decorate,decoration=brace] (4.0,-0.1)--(2.6,-0.1);
\end{tikzpicture}
\end{aligned}
\]
or the same sequences as above with the rows or colors swapped, where 
\[
p-q=s-r \quad\text{or}\quad p-q+r=-s
\]
respectively. In other words, the difference between the number of red and blue beads in the top row is equal to the difference in the number of red and blue beads in the bottom row.
\end{prop}
\begin{proof}
One can check that the above sequences of beads do in fact have a unique non-crossing matching. It is left to show that these are the only sequences for which this holds. We will do this by induction the difference between the number of red and blue beads in a row, $|s-r|$ or $s$ in the respective cases. 

As a base case, suppose there is an equal number of red and blue beads in each row. Then viewing each row individually, from Lemma \ref{lem:beadexist} we know there exist a matching of each row on its own. So there is a matching of the two rows with no arcs connecting them. For there to be a unique non-crossing matching then each row individually must have a unique non-crossing matching. Let us assume the leftmost bead on the top row is red. If there are no beads in the bottom row, we know the possible configurations of the top row from Proposition \ref{prop:beads1row} which agrees with the case in this proposition. 

Now suppose the bottom row is non-empty and the leftmost bead of the bottom row is red. We know there is a matching in which these beads match with blue beads to their right in their own row. Given this matching we construct a second matching
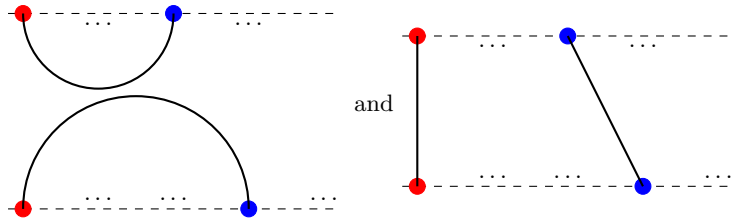
\[
\begin{aligned}
\begin{tikzpicture}[baseline=(current bounding box.center)] 
\draw[dashed] (-0.2,2.6)--(4.2,2.6); \draw[dashed] (-0.2,0)--(4.2,0);
\draw[red, fill=red] (0,2.6) circle (3pt); \node[below] at (1,2.6) {$\ldots$}; \draw[blue, fill=blue] (2,2.6) circle (3pt);  \node[below] at (3,2.6) {$\ldots$}; 
\draw[red, fill=red] (0,0) circle (3pt); \node[above] at (1,0) {$\ldots$}; \node[above] at (2,0) {$\ldots$}; \draw[blue, fill=blue] (3,0) circle (3pt);  \node[above] at (4,0) {$\ldots$}; 
 \draw[thick] (0,2.6) arc(180:360:1.0); 
 \draw[thick] (0,0) arc(180:0:1.5);
\end{tikzpicture}
\text{  and  }
\begin{tikzpicture}[baseline=(current bounding box.center)] 
\draw[dashed] (-0.2,2)--(4.2,2);\draw[dashed] (-0.2,0)--(4.2,0);
\draw[red, fill=red] (0,2) circle (3pt); \node[below] at (1,2) {$\ldots$}; \draw[blue, fill=blue] (2,2) circle (3pt);  \node[below] at (3,2) {$\ldots$}; 
\draw[red, fill=red] (0,0) circle (3pt); \node[above] at (1,0) {$\ldots$}; \node[above] at (2,0) {$\ldots$}; \draw[blue, fill=blue] (3,0) circle (3pt);  \node[above] at (4,0) {$\ldots$}; 
\draw[thick] (0,2)--(0,0); \draw[thick] (2,2)--(3,0);
\end{tikzpicture}
\end{aligned}
\]
by making the reds match with reds, and blues match with blues, across the rows. 

We see for there to be a unique non-crossing matching, the bottom row must start with a blue bead.  If the bottom row ends with a blue bead, we  similarly have two matchings
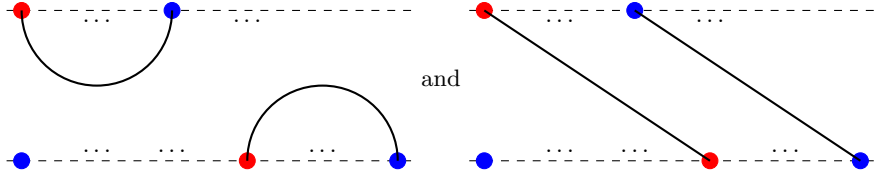
\[
\begin{aligned}
\begin{tikzpicture}[baseline=(current bounding box.center)] 
\draw[dashed] (-0.2,2)--(5.2,2); \draw[dashed] (-0.2,0)--(5.2,0);
\draw[red, fill=red] (0,2.0) circle (3pt); \node[below] at (1,2.0) {$\ldots$}; \draw[blue, fill=blue] (2,2.0) circle (3pt);  \node[below] at (3,2.0) {$\ldots$}; 
\draw[blue, fill=blue] (0,0) circle (3pt); \node[above] at (1,0) {$\ldots$}; \node[above] at (2,0) {$\ldots$}; \draw[red, fill=red] (3,0) circle (3pt);  \node[above] at (4,0) {$\ldots$}; \draw[blue, fill=blue] (5,0) circle (3pt);
 \draw[thick] (0,2.0) arc(180:360:1.0); 
 \draw[thick] (3,0) arc(180:0:1.0);
\end{tikzpicture}
\text{  and  }
\begin{tikzpicture}[baseline=(current bounding box.center)] 
\draw[dashed] (-0.2,2)--(5.2,2); \draw[dashed] (-0.2,0)--(5.2,0);
\draw[red, fill=red] (0,2.0) circle (3pt); \node[below] at (1,2.0) {$\ldots$}; \draw[blue, fill=blue] (2,2.0) circle (3pt);  \node[below] at (3,2.0) {$\ldots$}; 
\draw[blue, fill=blue] (0,0) circle (3pt); \node[above] at (1,0) {$\ldots$}; \node[above] at (2,0) {$\ldots$}; \draw[red, fill=red] (3,0) circle (3pt);  \node[above] at (4,0) {$\ldots$}; \draw[blue, fill=blue] (5,0) circle (3pt);
\draw[thick] (0,2)--(3,0); \draw[thick] (2,2)--(5,0);
\end{tikzpicture}
\end{aligned}
\]
where in the first matching no arcs connect the two rows. So the bottom row must end in a red bead. Applying this argument a third time we can show that for there to be a unique non-crossing matching, the top row must end in a blue bead.

The only sequences of beads satisfying these constraints, as well as those of Proposition \ref{prop:beads1row}, are of the form
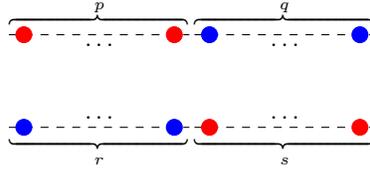
\[
\begin{aligned}
\overbrace{
\begin{tikzpicture}[baseline=(current bounding box.center)] 
 \draw[dashed] (-0.2,0)--(2.2,0);
\draw[red, fill=red] (0,0) circle (3pt); \node[below] at (1,0) {$\ldots$}; \draw[red, fill=red] (2,0) circle (3pt);
\end{tikzpicture}
}^{p}
\overbrace{
\begin{tikzpicture}[baseline=(current bounding box.center)] 
 \draw[dashed] (-0.2,0)--(2.2,0);
\draw[blue, fill=blue] (0,0) circle (3pt); \node[below] at (1,0) {$\ldots$}; \draw[blue, fill=blue] (2,0) circle (3pt);
\end{tikzpicture}
}^{q} \\ 
\vspace{1cm}
\\
\underbrace{
\begin{tikzpicture}[baseline=(current bounding box.center)] 
 \draw[dashed] (-0.2,0)--(2.2,0);
\draw[blue, fill=blue] (0,0) circle (3pt); \node[above] at (1,0) {$\ldots$}; \draw[blue, fill=blue] (2,0) circle (3pt);
\end{tikzpicture}
}_{r}
\underbrace{
\begin{tikzpicture}[baseline=(current bounding box.center)]
 \draw[dashed] (-0.2,0)--(2.2,0); 
\draw[red, fill=red] (0,0) circle (3pt); \node[above] at (1,0) {$\ldots$}; \draw[red, fill=red] (2,0) circle (3pt);
\end{tikzpicture}
}_{s}
\end{aligned}
\]
with $q-p=r-s=0$, as desired.

Now suppose we are in the case where the are $k$ more blue beads than red in each row. Any matching must have at least $k$ arcs connecting the two rows. Given any matching, consider the leftmost such arc. Suppose it connects between two blue beads, the case where it connects between two red beads can be done by swapping all the colors. This divides the sequence of beads into two pieces
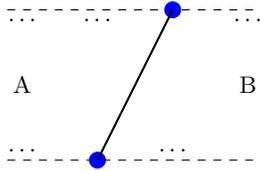
\[
\begin{tikzpicture}[baseline=(current bounding box.center)] 
 \draw[dashed] (-0.2,0)--(3.2,0);  \draw[dashed] (-0.2,2)--(3.2,2);
\node[below] at (0,2) {$\ldots$}; \node[below] at (1,2) {$\ldots$}; \draw[blue, fill=blue] (2,2) circle (3pt);  \node[below] at (3,2) {$\ldots$}; 
\node[above] at (0,0) {$\ldots$}; \draw[blue, fill=blue] (1,0) circle (3pt);  \node[above] at (2,0) {$\ldots$}; 
\draw[thick] (2,2)--(1,0);
\node at (0,1) {A}; \node at (3,1) {B};
\end{tikzpicture}
\]
For the whole sequence to have a unique non-crossing matching each piece must also have a unique non-crossing matching. Let us consider piece A. Since the difference between the number of blue and red beads in each row is zero in this piece, by induction we can see that this portion must take the form
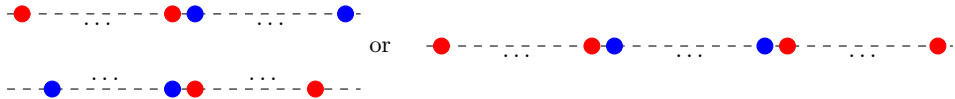
\[
\begin{aligned}
\begin{tikzpicture}[baseline=(current bounding box.center)] 
 \draw[dashed] (-0.2,0)--(4.5,0); \draw[dashed] (-0.2,1)--(4.5,1);
\draw[red, fill=red] (0,1) circle (3pt); \node[below] at (1,1) {$\ldots$}; \draw[red, fill=red] (2,1) circle (3pt);
\draw[blue, fill=blue] (2.3,1) circle (3pt); \node[below] at (3.3,1) {$\ldots$}; \draw[blue, fill=blue] (4.3,1) circle (3pt);
\draw[blue, fill=blue] (0.4,0) circle (3pt); \node[above] at (1.1,0) {$\ldots$}; \draw[blue, fill=blue] (2,0) circle (3pt);
\draw[red, fill=red] (2.3,0) circle (3pt); \node[above] at (3.2,0) {$\ldots$}; \draw[red, fill=red] (3.9,0) circle (3pt);
\end{tikzpicture}
&
\text{ or }
&
\begin{tikzpicture}[baseline=(current bounding box.center)] 
 \draw[dashed] (-0.2,1)--(6.8,1);
\draw[red, fill=red] (0,1) circle (3pt); \node[below] at (1,1) {$\ldots$}; \draw[red, fill=red] (2,1) circle (3pt);
\draw[blue, fill=blue] (2.3,1) circle (3pt); \node[below] at (3.3,1) {$\ldots$}; \draw[blue, fill=blue] (4.3,1) circle (3pt);
\draw[red, fill=red] (4.6,1) circle (3pt); \node[below] at (5.6,1) {$\ldots$}; \draw[red, fill=red] (6.6,1) circle (3pt);
\end{tikzpicture}
\end{aligned}
\]
or the same sequences as above with the rows or colors swapped. Including the arc between the two blue beads, the only case in which we are unable to construct a second non-crossing matching is
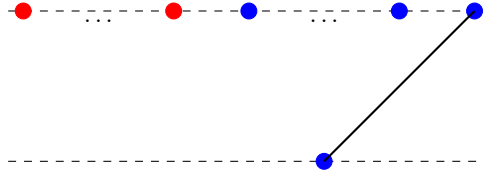
\[
\begin{tikzpicture}[baseline=(current bounding box.center)] 
\draw[dashed] (-0.2,0)--(6.2,0);\draw[dashed] (-0.2,2)--(6.2,2);
\draw[red, fill=red] (0,2) circle (3pt); \node[below] at (1,2) {$\ldots$}; \draw[red, fill=red] (2,2) circle (3pt); \draw[blue, fill=blue] (3,2) circle (3pt);  \node[below] at (4,2) {$\ldots$}; \draw[blue, fill=blue] (5,2) circle (3pt); \draw[blue, fill=blue] (6,2) circle (3pt); 
\draw[blue, fill=blue] (4,0) circle (3pt);
\draw[thick] (6,2)--(4,0);
\end{tikzpicture}
\]
or the same sequences with the rows swapped. For example, if both rows are non-empty we have the two matchings 
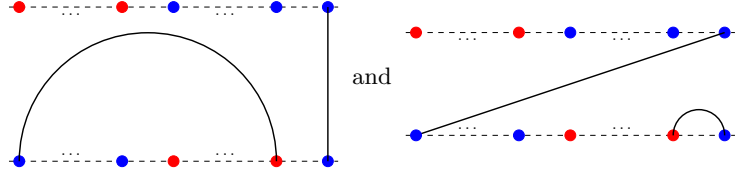
\[
\resizebox{0.3\textwidth}{!}{
\begin{tikzpicture}[baseline=(current bounding box.center)] 
\draw[dashed] (-0.2,3)--(6.2,3);  \draw[dashed] (-0.2,0)--(6.2,0);
\draw[red, fill=red] (0,3) circle (3pt); \node[below] at (1,3) {$\ldots$}; \draw[red, fill=red] (2,3) circle (3pt); \draw[blue, fill=blue] (3,3) circle (3pt);  \node[below] at (4,3) {$\ldots$}; \draw[blue, fill=blue] (5,3) circle (3pt); \draw[blue, fill=blue] (6,3) circle (3pt); 
\draw[blue, fill=blue] (6,0) circle (3pt);
\draw[thick] (6,3)--(6,0);
\draw[blue, fill=blue] (0,0) circle (3pt); \node[above] at (1,0) {$\ldots$}; \draw[blue, fill=blue] (2,0) circle (3pt);
\draw[red, fill=red] (3,0) circle (3pt); \node[above] at (4,0) {$\ldots$}; \draw[red, fill=red] (5,0) circle (3pt);
\draw[thick] (0,0) arc(180:0:2.5);
\end{tikzpicture}
}
\text{ and }
\resizebox{0.3\textwidth}{!}{
\begin{tikzpicture}[baseline=(current bounding box.center)] 
\draw[dashed] (-0.2,2)--(6.2,2);  \draw[dashed] (-0.2,0)--(6.2,0);
\draw[red, fill=red] (0,2) circle (3pt); \node[below] at (1,2) {$\ldots$}; \draw[red, fill=red] (2,2) circle (3pt); \draw[blue, fill=blue] (3,2) circle (3pt);  \node[below] at (4,2) {$\ldots$}; \draw[blue, fill=blue] (5,2) circle (3pt); \draw[blue, fill=blue] (6,2) circle (3pt); 
\draw[blue, fill=blue] (6,0) circle (3pt);
\draw[thick] (6,2)--(0,0);
\draw[blue, fill=blue] (0,0) circle (3pt); \node[above] at (1,0) {$\ldots$}; \draw[blue, fill=blue] (2,0) circle (3pt);
\draw[red, fill=red] (3,0) circle (3pt); \node[above] at (4,0) {$\ldots$}; \draw[red, fill=red] (5,0) circle (3pt);
\draw[thick] (5,0) arc(180:0:0.5);
\end{tikzpicture}
}
\]
where the first is the unique non-crossing matching of the beads in piece A (we draw the arc of particular interest to us) as well as the arc connecting the two rightmost blue beads, and in the second the matching changes as shown.  

Now consider piece B (including the arc between the two blue beads). The difference between the number of blue and red beads in each row is $k-1$ in this piece, so again by induction we can see that this portion must take the forms
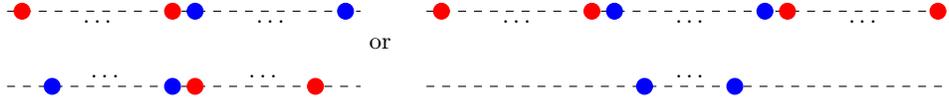
\[
\begin{aligned}
\begin{tikzpicture}[baseline=(current bounding box.center)] 
\draw[dashed] (-0.2,1)--(4.5,1);  \draw[dashed] (-0.2,0)--(4.5,0);
\draw[red, fill=red] (0,1) circle (3pt); \node[below] at (1,1) {$\ldots$}; \draw[red, fill=red] (2,1) circle (3pt);
\draw[blue, fill=blue] (2.3,1) circle (3pt); \node[below] at (3.3,1) {$\ldots$}; \draw[blue, fill=blue] (4.3,1) circle (3pt);
\draw[blue, fill=blue] (0.4,0) circle (3pt); \node[above] at (1.1,0) {$\ldots$}; \draw[blue, fill=blue] (2,0) circle (3pt);
\draw[red, fill=red] (2.3,0) circle (3pt); \node[above] at (3.2,0) {$\ldots$}; \draw[red, fill=red] (3.9,0) circle (3pt);
\end{tikzpicture}
&
\text{ or }
&
\begin{tikzpicture}[baseline=(current bounding box.center)] 
\draw[dashed] (-0.2,1)--(6.8,1);  \draw[dashed] (-0.2,0)--(6.8,0);
\draw[red, fill=red] (0,1) circle (3pt); \node[below] at (1,1) {$\ldots$}; \draw[red, fill=red] (2,1) circle (3pt);
\draw[blue, fill=blue] (2.3,1) circle (3pt); \node[below] at (3.3,1) {$\ldots$}; \draw[blue, fill=blue] (4.3,1) circle (3pt);
\draw[red, fill=red] (4.6,1) circle (3pt); \node[below] at (5.6,1) {$\ldots$}; \draw[red, fill=red] (6.6,1) circle (3pt);
\draw[blue, fill=blue] (2.7,0) circle (3pt); \node[above] at (3.3,0) {$\ldots$}; \draw[blue, fill=blue] (3.9,0) circle (3pt);
\end{tikzpicture}
\end{aligned}
\]
or the same sequences as above with the rows or colors swapped. Including the arc between the two blue beads, the only sequence of beads that still has a unique non-crossing matching is
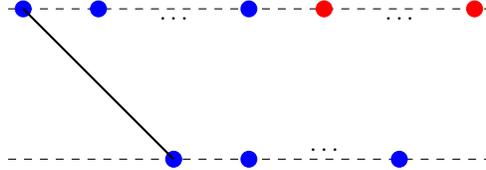
\[
\begin{tikzpicture}[baseline=(current bounding box.center)] 
\draw[dashed] (-0.2,2)--(6.2,2);  \draw[dashed] (-0.2,0)--(6.2,0);
\draw[blue, fill=blue] (0,2) circle (3pt); \draw[blue, fill=blue] (1,2) circle (3pt);  \node[below] at (2,2) {$\ldots$}; \draw[blue, fill=blue] (3,2) circle (3pt); \draw[red, fill=red] (4,2) circle (3pt); \node[below] at (5,2) {$\ldots$}; \draw[red, fill=red] (6,2) circle (3pt);
\draw[blue, fill=blue] (2,0) circle (3pt); \draw[blue, fill=blue] (3,0) circle (3pt);  \node[above] at (4,0) {$\ldots$}; \draw[blue, fill=blue] (5,0) circle (3pt); 
\draw[thick] (0,2)--(2,0);
\end{tikzpicture}
\]
or the same configuration with the rows swapped.

Finally, combining the two pieces we see that the configuration of beads must be in the form given in the statement of the proposition.
\end{proof}

We can pull this constraint on matching back to a constraint on the partition $\bg = (\beta^{(1)}/\gamma^{(1)},\beta^{(2)}/\gamma^{(2)})$. Recall that the top row of beads corresponds to the singleton paths in the top boundary condition of the vertex model, which depends only on $\beta^{(1)}$ and $\beta^{(2)}$. Similarly, the bottom row of beads depends only on $\gamma^{(1)}$ and $\gamma^{(2)}$.  For simplicity, we will assume $l(\beta^{(1)}) = l(\beta^{(2)})=n$. Let $\delta_n=(n-1,n-2,\ldots,0)$ be the staircase partition of length $n$. Consider the strict partitions $\tilde \beta^{(1)}=\beta^{(1)}+\delta_n$ and $\tilde \beta^{(2)}=\beta^{(2)}+\delta_n$. Recall that each row of the partitions $\beta^{(1)}/\gamma^{(1)}, \beta^{(2)}/\gamma^{(2)}$ corresponds to a path in the vertex model. The length of the rows of $\tilde \beta^{(1)}$ and those of $\tilde \beta^{(2)}$ encode the number of horizontal steps the paths take. 

We can construct a third strict partition $\sigma_\beta$ whose parts are the parts of $\tilde \beta^{(1)}$ and those of $\tilde \beta^{(2)}$ sorted into decreasing order, along with the condition that anytime a part of $\tilde \beta^{(1)}$ is equal to a part of $\tilde \beta^{(2)}$ both are removed. The order of the parts in this partition gives the order of the singleton boundary paths from rightmost to leftmost. If we color the parts coming from $\tilde \beta^{(1)}$ blue and those coming from $\tilde \beta^{(2)}$ red then this is precisely the ordering of the sequence of colored beads associated to the top boundary. We see that the top row of the sequence of colored beads is equivalent to the order of the parts of $\sigma_\beta$ if we keep track of each part's color. The bottom row of the sequence of colored beads comes from the sorted partition $\sigma_\gamma$.

For example, when $\bg = ((5,5,1)/(2,1,0),(4,3,2)/(1,0,0))$. Then $\tilde \beta^{(1)} = (7,6,1)$, $\tilde \beta^{(2)} = (6,4,2)$, $\tilde \gamma^{(1)}=(4,2,0)$, and $\tilde \gamma^{(2)}=(3,1,0)$. The Young diagrams for $\sigma_\beta$ and $\sigma_\gamma$ are then
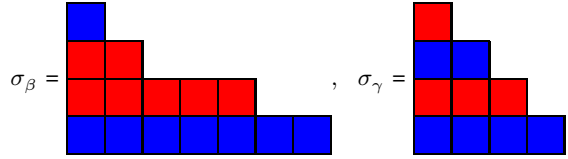
\[
\begin{aligned}
\sigma_\beta = 
	\ytableausetup{nosmalltableaux}
	\ytableausetup{nobaseline}
	\ytableausetup{aligntableaux=center}
	\begin{ytableau}
	*(blue)  \\
	*(red)  & *(red) \\
	*(red)  & *(red) & *(red)  & *(red) \\
	*(blue) & *(blue) & *(blue) & *(blue) & *(blue) & *(blue) & *(blue)
\end{ytableau},\;\;\;
& 
\sigma_\gamma = 
	\ytableausetup{nosmalltableaux}
	\ytableausetup{nobaseline}
	\ytableausetup{aligntableaux=center}
	\begin{ytableau}
	*(red)  \\
	*(blue)  & *(blue) \\
	*(red)  & *(red) & *(red) \\
	*(blue) & *(blue) & *(blue) & *(blue)
\end{ytableau}
\end{aligned}
\]
corresponding to the vertex model boundary condition and sequence of colored beads
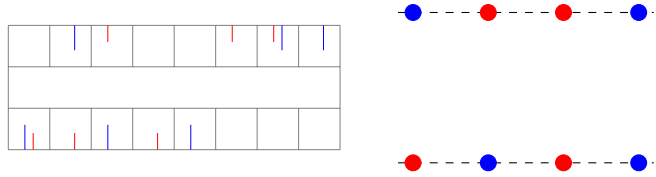
\[
\resizebox{0.3\textwidth}{!}{
\begin{tikzpicture}[baseline=(current bounding box.center)] 
\draw[help lines] (0,0) grid (8,1);
\draw[help lines] (0,2) grid (8,3);
\draw[help lines] (0,1)--(0,2); \draw[help lines] (8,1)--(8,2);
\draw[blue] (7.6,2.4)--(7.6,3); \draw[blue] (6.6,2.4)--(6.6,3); \draw[blue] (1.6,2.4)--(1.6,3); \draw[red] (6.4,2.6)--(6.4,3); \draw[red] (4.4,2.6)--(4.4,3); \draw[red] (2.4,2.6)--(2.4,3);
\draw[blue] (4.4,0)--(4.4,0.6); \draw[blue] (2.4,0)--(2.4,0.6); \draw[blue] (0.4,0)--(0.4,0.6); \draw[red] (3.6,0)--(3.6,0.4); \draw[red] (1.6,0)--(1.6,0.4); \draw[red] (0.6,0)--(0.6,0.4);
\end{tikzpicture}
}\;\;\;\;\;\;\;\;
\begin{tikzpicture}[baseline=(current bounding box.center)] 
\draw[dashed] (-0.2,2)--(3.2,2); \draw[dashed] (-0.2,0)--(3.2,0);
\draw[blue, fill=blue] (0,2) circle (3pt);\draw[red, fill=red] (1,2) circle (3pt); \draw[red, fill=red] (2,2) circle (3pt); \draw[blue, fill=blue] (3,2) circle (3pt);
\draw[red, fill=red] (0,0) circle (3pt); \draw[blue, fill=blue] (1,0) circle (3pt); \draw[red, fill=red] (2,0) circle (3pt); \draw[blue, fill=blue] (3,0) circle (3pt);
\end{tikzpicture}
\]
 The constraints on the sequence of beads can then be translated to constraints on $\sigma_\beta$ and $\sigma_\gamma$.

\end{document}